\documentclass[preprint,10pt]{elsarticle}

\usepackage{amsmath}
\usepackage{amssymb}
\usepackage{amsthm}
\usepackage{lipsum}
\usepackage{amsfonts}
\usepackage{graphicx}
\usepackage{epstopdf}
\usepackage{algorithmic}
\usepackage{physics}
\usepackage{dsfont}
\usepackage{mathtools}
\usepackage{enumerate}
\usepackage{color}
\usepackage{hyperref}
\usepackage{caption}
\usepackage[pagewise]{lineno}
\usepackage{subcaption}
\usepackage{geometry}
\usepackage[title,toc,titletoc]{appendix}
\usepackage{titlesec}

\numberwithin{equation}{section}

\newtheorem{theorem}{Theorem}[section]
\newtheorem{remark}{Remark}[section]

\newtheorem{lemma}{Lemma}[section]
\newtheorem{assumption}{Assumption}

\newtheorem{example}{Example}[section]

\DeclareMathOperator*{\esssup}{ess\,sup}

\begin{document}

\begin{frontmatter}

\title{Maximum Principle for State-Constrained Optimal Control Problems of Volterra Integral Equations having Singular and Nonsingular Kernels\tnoteref{mytitlenote}
}
\tnotetext[mytitlenote]{This research was supported in part by the National Research Foundation of Korea (NRF) Grant funded by the Ministry of Science and ICT, South Korea (NRF-2021R1A2C2094350) and in part by Institute of Information \& communications Technology Planning \& Evaluation (IITP) grant funded by the Korea government (MSIT) (No.2020-0-01373, Artificial Intelligence Graduate School Program (Hanyang University)).}

\author{Jun Moon}
\address{Department of Electrical Engineering, Hanyang University, Seoul 04763, South Korea}
\ead{junmoon@hanyang.ac.kr}

\begin{abstract}
In this paper, we study the optimal control problem with terminal and inequality state constraints for state equations described by Volterra integral equations having singular and nonsingular kernels. The singular kernel introduces abnormal behavior of the state trajectory with respect to the parameter of $\alpha \in (0,1)$. Our state equation is able to cover various state dynamics such as any types of Volterra integral equations with nonsingular kernels only, fractional differential equations (in the sense of Riemann-Liouville or Caputo), and ordinary differential state equations. We obtain the well-posedness (in $L^p$ and $C$ spaces) and precise estimates of the state equation using the generalized Gronwall's inequality and the proper regularities of integrals having singular and nonsingular integrands. We then prove the maximum principle for the corresponding state-constrained optimal control problem. In the derivation of the maximum principle, due the presence of the state constraints and the control space being only a separable metric space, we have to employ the Ekeland variational principle and the spike variation technique, together with the intrinsic properties of distance functions and the generalized Gronwall's inequality, to obtain the desired necessary conditions for optimality. In fact, as the state equation has both singular and nonsingular kernels, the maximum principle of this paper is new, where its proof is more involved than that for the problems of Volterra integral equations studied in the existing literature. Examples are provided to illustrate the theoretical results of this paper.
\end{abstract}

\end{frontmatter}

\tableofcontents

\section{Introduction}\label{Section_1}

In this paper, we consider the optimal control problem of 
\begin{align}	
\label{eq_intro_1_1}
\textbf{(P)}~~J(x_0,u(\cdot)) = \int_0^T l(r,x(r),u(r) \dd r + h(x_0,x(T)),
\end{align}
subject to the following state equation with $\alpha \in (0,1)$,
\begin{align}
\label{eq_intro_1_2}
x(t) = x_0 + \int_0^t \frac{f(t,s,x(s),u(s))}{(t-s)^{1-\alpha}} \dd s + \int_0^t g(t,s,x(s),u(s)) \dd s,~ \textrm{a.e.}~ t \in [0,T],
\end{align}
and the state constraints
\begin{align}
\label{eq_intro_1_3}
\begin{cases}
(x_0, x(T)) \in F, & \textrm{(terminal state constraint)},\\
	G^i(t,x(t)) \leq 0,~\forall t \in [0,T],~ i=1,\ldots,m, & \textrm{(inequality state constraint).} 
\end{cases}
\end{align}
The precise problem statement of \textbf{(P)} including the space of admissible controls and the standing assumptions for (\ref{eq_intro_1_1})-(\ref{eq_intro_1_3}) is given in Section \ref{Section_2_2}. We mention that the optimal control problems with state constraints capture various practical aspects of systems in science, biology, engineering, and economics \cite{Hartl_SICON_1995, Arutyunov_JOTA_2020, Evans_2010, Vinter_book}.

The state equation in (\ref{eq_intro_1_2}) is known as a class of Volterra integral equations. The main feature of Volterra integral equations is the effect of memories, which does not appear in ordinary (state) differential equations. In fact, Volterra integral equations of various kinds have been playing an important role in modeling and analyzing of practical physical, biological, engineering, and other phenomena that are governed by memory effects \cite{Burton_book}. We note that one major distinction between (\ref{eq_intro_1_2}) and other existing Volterra integral equations is that (\ref{eq_intro_1_2}) has two different kernels $\frac{f(t,s,x,u)}{(t-s)^{1-\alpha}}$ and $g(t,s,x,u)$, in which the first kernel $\frac{f(t,s,x,u)}{(t-s)^{1-\alpha}}$ becomes singular at $s=t$, while the second kernel $g(t,s,x,u)$ is nonsingular.  In fact, $\alpha \in (0,1)$ in the singular kernel of (\ref{eq_intro_1_2}) determines the amount of the singularity, in which the large singular behavior occurs with small $\alpha \in (0,1)$. 

Optimal control problems for various kinds of Volterra integral equations via the maximum principle have been studied extensively in the literature; see \cite{Vinokurov_SICON_1969, Angell_JOTA_1976, Kamien_RES_1976, Medhin_JMAA_1988, Carlson_JOTA_1987, Burnap_IMA_Control_1999, Vega_JOTA_2006, Bonnans_Vega_JOTA_2013, Dmitruk_MCRF_2017, Dmitruk_SICON_2014, Bonnans_SVA_2010} and the references therein. Specifically, the first study on optimal control for Volterra integral equations (using the maximum principle) can be traced back to \cite{Vinokurov_SICON_1969}. Several different formulations (with/without state constraints, with/without delay, with/without additional equality and/or inequality constraints) of optimal control for Volterra integral equations and their generalizations are reported in  \cite{Angell_JOTA_1976, Kamien_RES_1976, Medhin_JMAA_1988, Carlson_JOTA_1987, Burnap_IMA_Control_1999, Vega_JOTA_2006, Belbas_AMC_2007, Bonnans_SVA_2010}. Some recent progress in different directions including the stochastic framework can be found in \cite{Bonnans_Vega_JOTA_2013, Dmitruk_MCRF_2017, Dmitruk_SICON_2014, Wang_ESAIM_2018, Hamaguchi_Arvix_2021}. We note that the above-mentioned existing works considered the situation with nonsingular kernels only in Volterra integral equations, which corresponds to $f \equiv 0$ in (\ref{eq_intro_1_2}). Hence, the problem settings in the earlier works can be viewed as a special case of \textbf{(P)}.

Recently, the optimal control problem for Volterra integral equations having singular kernels only (equivalently, $g \equiv 0$ in (\ref{eq_intro_1_2})) was studied in \cite{Lin_Yong_SICON_2020}. Due to the presence of the singular kernel, the technical analysis including the maximum principle (without state constraints) in \cite{Lin_Yong_SICON_2020} should be different from that of the existing works mentioned above. In particular, the proof for the well-posedness and estimates of Volterra integral equations in \cite[Theorem 3.1]{Lin_Yong_SICON_2020} require a new type of the Gronwall's inequality. Furthermore, the maximum principle (without state constraints) in \cite[Theorem 4.3]{Lin_Yong_SICON_2020} needs a different duality analysis for variational and adjoint integral equations, induced by the variational approach. More recently, linear-quadratic optimal control problem (without state constraints) for linear Volterra integral equations with singular kernels only was studied in \cite{Han_Arxiv_2021}.

We note that Volterra integral equations having singular and nonsingular kernels are strongly related to classical state equations and fractional order differential equations in the sense of Riemann-Liouville or Caputo \cite{Kilbas_book}. For the case with singular kernels only, a similar argument is given in \cite[Section 3.2]{Lin_Yong_SICON_2020}. In particular, let $\mathcal{D}_{\alpha} ^C [x(\cdot)]$ be the fractional derivative operator of order $\alpha \in (0,1)$ in the sense of Caputo \cite[Chapter 2.4]{Kilbas_book}. Then applying \cite[Theorem 3.24 and Corollary 3.23]{Kilbas_book} to (\ref{eq_intro_1_2}) yields
\begin{subequations}
\begin{align}
\label{eq_intro_1_4}
\mathcal{D}_{\alpha}^C [x(\cdot)](t) = f(t,x(t),u(t)) & ~\Leftrightarrow~ x(t) = x_0 + \frac{1}{\Gamma(\alpha)} \int_0^t \frac{f(s,x(t),u(s))}{(t-s)^{1-\alpha}} \dd s,~ \textrm{a.e.}~ t \in [0,T], \\
\label{eq_intro_1_5}
\frac{ \dd x(t)}{\dd t} = g(t,x(t),u(t)) & ~\Leftrightarrow~ x(t) = x_0 + \int_0^t g(s,x(s),u(s)) \dd s,~ \textrm{a.e.}~ t \in [0,T],
\end{align}
\end{subequations}
where $\Gamma(\cdot)$ is the gamma function. Note that while (\ref{eq_intro_1_4}) is a class of fractional differential equations in the sense of Caputo, (\ref{eq_intro_1_5}) is a classical ordinary differential equation.
%(equivalent to the fractional differential equation in the sense of Caputo with $\alpha = 1$)
Instead of $\mathcal{D}_{\alpha}^C [x(\cdot)]$ in (\ref{eq_intro_1_4}), we may use the fractional derivative of order $\alpha \in (0,1)$ in the sense of Riemann-Liouville \cite[Chapter 2.1 and Theorem 3.1]{Kilbas_book}. Hence, we observe that (\ref{eq_intro_1_4}) and (\ref{eq_intro_1_5}) are special cases of our state equation in (\ref{eq_intro_1_2}). This implies that the state equation in (\ref{eq_intro_1_2}) is able to describe various types of differential equations including combinations of fractional (in Riemann-Liouville- or Caputo-type) and ordinary differential state equations. We also mention that there are several different results on optimal control for fractional differential equations; see \cite{Agrawal_ND_2004, Bourdin_arvix_2012, Kamocki_AMC_2014,  Gomoyunov_SICON_2020} and the references therein.

The aim of this paper is to study the optimal control problem stated in \textbf{(P)}. As noted above, since (\ref{eq_intro_1_2}) has both singular and nonsingular kernels, when $f \equiv 0$, (\ref{eq_intro_1_2}) is reduced to the Volterra integral equation with singular kernels only studied in \cite{Lin_Yong_SICON_2020}. Since \cite{Lin_Yong_SICON_2020} did not consider the state-constrained control problem, \textbf{(P)} can be viewed as a generalization of \cite{Lin_Yong_SICON_2020} to the state-constrained control problem for Volterra integral equations having singular and nonsingular kernels. Moreover, with $g \equiv 0$, (\ref{eq_intro_1_2}) is reduced to the classical Volterra integral equation with nonsingular kernels only (e.g. \cite{Dmitruk_MCRF_2017, Dmitruk_SICON_2014, Burnap_IMA_Control_1999, Medhin_JMAA_1988, Kamien_RES_1976, Bonnans_SVA_2010}). Hence, \textbf{(P)} also covers the optimal control problems for Volterra integral equations with nonsingular kernels only.

Under mild assumptions on $f$ and $g$, we first obtain the well-posedness (in $L^p$ and $C$ spaces) and precise estimates for generalized Volterra integral equations of (\ref{eq_intro_1_2}) when the initial condition of (\ref{eq_intro_1_2}) also depends on $t$ (see Lemma \ref{Lemma_2_1} and Appendix \ref{Appendix_B}). This requires the extensive use of the generalized Gronwall's inequality with singular and nonsingular kernels, together with the several different regularities of integrals having singular and nonsingular integrands, where their results (including the generalized Gronwall's inequality) are obtained in Appendix \ref{Appendix_A}. Note that the main technical analysis for the well-posedness and estimates of (\ref{eq_intro_1_2}) (see Lemma \ref{Lemma_2_1} and Appendix \ref{Appendix_B}) should be different from those for the case with singular kernels only in \cite{Lin_Yong_SICON_2020}, as the presence of the singular and nonsingular kernels in (\ref{eq_intro_1_2}) causes various cross coupling characteristics.

Next, we obtain the maximum principle for \textbf{(P)} (see Theorem \ref{Theorem_3_1}). Due the presence of the state constraints in (\ref{eq_intro_1_3}) and the control space being only a separable metric space (that does not necessarily have any algebraic structure), the derivation of the maximum principle in this paper must be different from that for the unconstrained case with singular kernels only studied in \cite[Theorem 4.3]{Lin_Yong_SICON_2020}. Specifically, we have to employ the Ekeland variational principle and the spike variation technique, together with the intrinsic properties of distance functions and the generalized Gronwall's inequality (see Appendix \ref{Appendix_A}), to establish the duality analysis for Volterra-type variational and adjoint equations, which leads to the desired necessary conditions for optimality. 
%Hence, the proof for the maximum principle of this paper must be different from that for the unconstrained case with singular kernels only studied in \cite[Theorem 4.3]{Lin_Yong_SICON_2020}. 
Furthermore, as (\ref{eq_intro_1_2}) has both singular and nonsingular kernels, the proof for the maximum principle of this paper should be more involved than that for the classical state-constrained maximum principle without singular kernels studied in the existing literature (e.g. \cite[Theorem 1]{Bonnans_SVA_2010} and  \cite{Dmitruk_MCRF_2017, Dmitruk_SICON_2014, Burnap_IMA_Control_1999, Medhin_JMAA_1988, Kamien_RES_1976}). In fact, the analysis of the maximum principle for state-constrained optimal control problems is entirely different from that of the problems without state constraints \cite{Hartl_SICON_1995, Bourdin_arxiv_2016}. We also note that different from existing works for classical optimal control of Volterra integral equations (e.g. \cite{Dmitruk_MCRF_2017, Dmitruk_SICON_2014, Burnap_IMA_Control_1999, Medhin_JMAA_1988, Kamien_RES_1976, Bonnans_SVA_2010}), our paper does not assume the differentiability of (singular and nonsingular) kernels in $(t,s,u)$ (time and control variables) and the convexity of the control space.

The rest of this paper is organized as follows. The notation and the problem statement of \textbf{(P)} are given in Section \ref{Section_2}. The statement of the maximum principle for \textbf{(P)} is provided in Section \ref{Section_3}. Some examples of \textbf{(P)} are studied in Section \ref{Section_5}. The proof of the maximum principle for \textbf{(P)} is given in Section \ref{Section_4}. Appendices \ref{Appendix_A}-\ref{Appendix_D} give some preliminary results and lemmas including the well-posedness and estimates of (\ref{eq_intro_1_2}).

\section{Notation and Problem Formulation}\label{Section_2}

%We first provide the notation used in this paper. The precise problem statement of \textbf{(P)} is then provided.

\subsection{Notation}\label{Section_2_1}

Let $\mathbb{R}_+$ and $\mathbb{R}_-$ be the sets of nonnegative and nonpositive numbers, respectively. Let $\mathbb{R}^n$ be the $n$-dimensional Euclidean space, where $\langle x,y \rangle_{\mathbb{R}^n \times \mathbb{R}^n} := x^\top y$ is the inner product and $|x|_{\mathbb{R}^n} := \langle x,x \rangle^{1/2}_{\mathbb{R}^n \times \mathbb{R}^n}$ is the norm for $x,y \in \mathbb{R}^n$. We sometimes write $\langle \cdot,\cdot \rangle$ and $|\cdot|$ when there is no confusion. For $A \in \mathbb{R}^{m \times n}$, $A^\top$ denotes the transpose of $A$. Let $I_{n}$ be an $n \times n$ identity matrix.
%Let $T > 0$ be a fixed horizon.  
Let $\Delta := \{(t,s) \in [0,T] \times [0,T]~|~ 0 \leq s \leq t \leq T \}$ with $T > 0$ being a fixed horizon. Define $\mathds{1}_{A}(\cdot)$ by the indicator function of any set $A$. A modulus of continuity is any increasing real-valued function $\omega :[0,\infty) \rightarrow [0,\infty)$, vanishing at $0$, i.e., $\lim_{t \downarrow 0} \omega(t) = 0$, and continuous at $0$. In this paper, the constant $C$ denotes the generic constant, whose value is different from line to line.

For any differentiable function $f:\mathbb{R}^n \rightarrow \mathbb{R}^l$, let $f_x : \mathbb{R}^n \rightarrow \mathbb{R}^{l \times n}$ be the partial derivative of $f$ with respect to $x \in \mathbb{R}^n$. Note that $f_x =  \begin{bmatrix}
	f_{1,x}^\top & \cdots & f_{l,x}^\top \end{bmatrix}^\top$ with $f_{j,x} \in  \mathbb{R}^{1 \times n}$, and when $l=1$, $f_x \in \mathbb{R}^{1 \times n}$. For any differentiable function $f : \mathbb{R}^n \times \mathbb{R}^l \rightarrow \mathbb{R}^l$, $f_{x} : \mathbb{R}^n \times \mathbb{R}^l \rightarrow \mathbb{R}^{l \times n}$ for $x \in \mathbb{R}^n$, and $f_y: \mathbb{R}^n \times \mathbb{R}^l \rightarrow \mathbb{R}^{l \times l}$ for $y \in \mathbb{R}^l$. 
	
%	Note that for any differentiable map $f:[0,T] \times \mathbb{R}^n \rightarrow \mathbb{R}^l$.
	
%	For any twice differentiable map $f :\mathbb{R}^n \rightarrow \mathbb{R}$, $f_{xx} : \mathbb{R}^n \rightarrow \mathbb{R}^{n \times n}$ is the second-order partial derivative (Hessian) of $f$, and for $f : \mathbb{R}^n \times \mathbb{R}^l \rightarrow \mathbb{R}$, $f_{xy} = (f_x^\top)_y : \mathbb{R}^n \times \mathbb{R}^l \rightarrow \mathbb{R}^{n \times l}$ is the partial derivative of $f$ with respect to $(x,y) \in \mathbb{R}^n \times \mathbb{R}^l$. 

For $1 \leq p < \infty$, we define the following spaces:
\begin{itemize}
\item $L^p([0,T];\mathbb{R}^n)$: the space of functions $\psi:[0,T] \rightarrow \mathbb{R}^n$ such that $\psi$ is measurable and satisfies $\|\psi(\cdot)\|_{L^p([0,T];\mathbb{R}^n)} := \Bigl ( \int _0^T |\psi(t)|^p_{\mathbb{R}^n} \dd t \Bigr )^{1/p}$;
\item $L^{\infty}([0,T];\mathbb{R}^n)$: the space of functions $\psi:[0,T] \rightarrow \mathbb{R}^n$ such that $\psi$ is measurable and satisfies $\|\psi(\cdot)\|_{L^{\infty}([0,T];\mathbb{R}^n)} :=  \esssup_{t \in [0,T]} |\psi(t)|_{\mathbb{R}^n} < \infty$;
\item $C([0,T];\mathbb{R}^n)$: the space of functions $\psi:[0,T] \rightarrow \mathbb{R}^n$ such that $\psi$ is continuous and satisfies $\|\psi(\cdot)\|_{\infty} := \sup_{t \in [0,T]} |\psi(t)|_{\mathbb{R}^n} < \infty $;
%$\|\psi(\cdot)\|_{C([0,T];\mathbb{R}^n)} := \sup_{t \in [0,T]} |\psi(t)|_{\mathbb{R}^n} < \infty $;
\item $\textsc{BV}([0,T];\mathbb{R}^n)$: the space of functions $\psi:[0,T] \rightarrow \mathbb{R}^n$ such that $\psi$ is a function with bounded variation on $[0,T]$.
\end{itemize}
The norm on $\textsc{BV}([0,T];\mathbb{R}^n)$ is defined by $\|\psi(\cdot)\|_{\textsc{BV}([0,T];\mathbb{R}^n)} := \psi(0) + \textsc{TV}(\psi)$, where $\textsc{TV}(\psi) := \sup_{(t_k)_k} \bigl \{ \sum_{k} |\psi(t_{k+1}) - \psi(t_k)|_{\mathbb{R}^n} \bigr \} < \infty$ with the supremum being taken by all partitions of $[0,T]$. Let $\textsc{NBV}([0,T];\mathbb{R}^n)$ be the space of functions $\psi(\cdot) \in  \textsc{BV}([0,T];\mathbb{R}^n)$ such that $\psi(\cdot) \in  \textsc{BV}([0,T];\mathbb{R}^n)$ is normalized, i.e., $\psi(0) = 0$ and $\psi$ is left continuous. The norm on $\textsc{NBV}([0,T];\mathbb{R}^n)$ is defined by $\|\psi(\cdot)\|_{\textsc{NBV}([0,T];\mathbb{R}^n)} := \textsc{TV}(\psi)$. When $\psi(\cdot) \in \textsc{NBV}([0,T];\mathbb{R})$ is monotonically nondecreasing, we have $\|\psi(\cdot)\|_{\textsc{NBV}([0,T];\mathbb{R}} = \psi(T)$. Note that both $(\textsc{BV}([0,T];\mathbb{R}^n), \|\cdot\|_{\textsc{BV}([0,T];\mathbb{R}^n)})$ and $(\textsc{NBV}([0,T];\mathbb{R}^n), \|\cdot\|_{\textsc{NBV}([0,T];\mathbb{R}^n)})$ are Banach spaces. 

\subsection{Problem Formulation}\label{Section_2_2}

\begin{figure}
\centering
\includegraphics[scale=0.35]{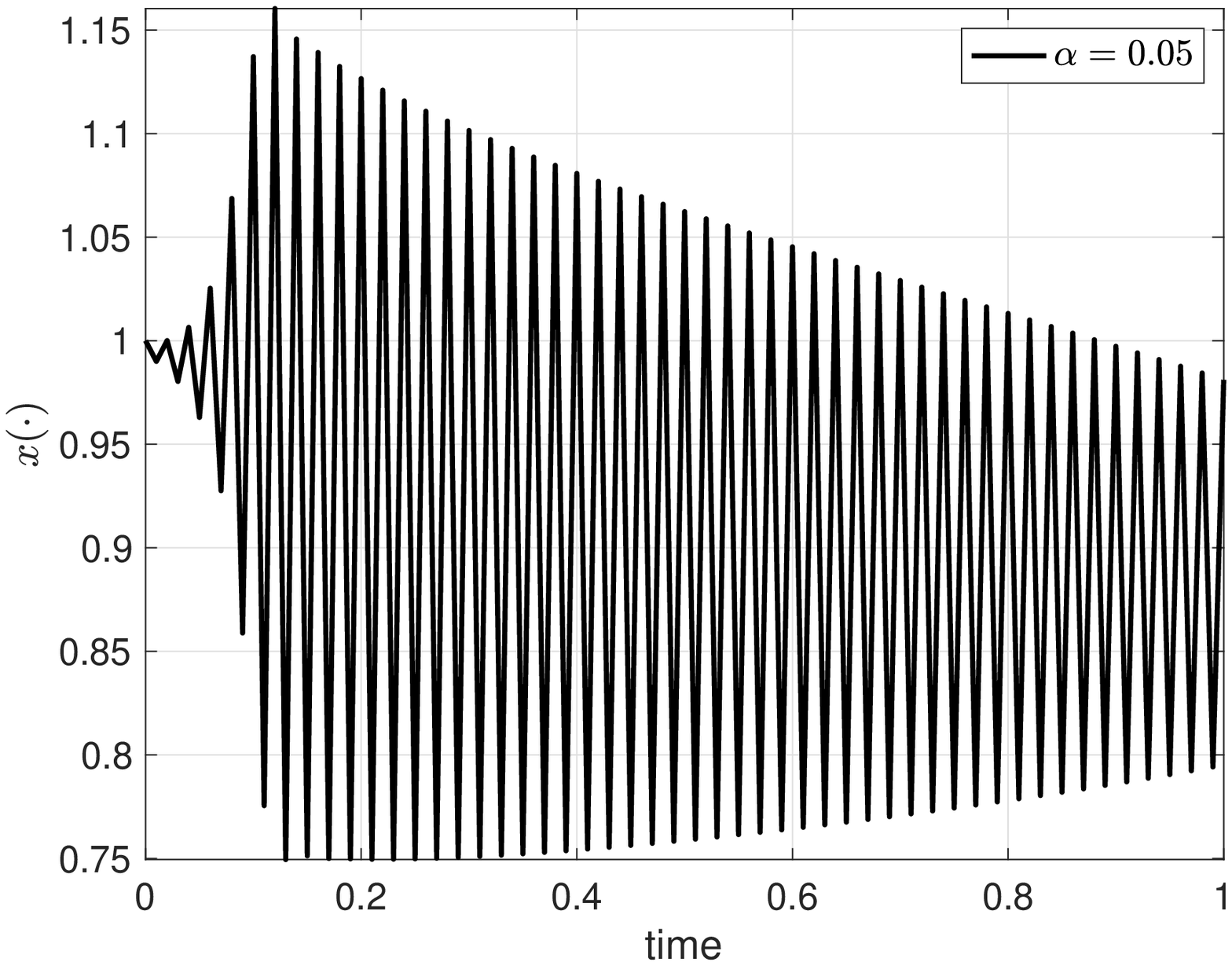}~~~
\includegraphics[scale=0.35]{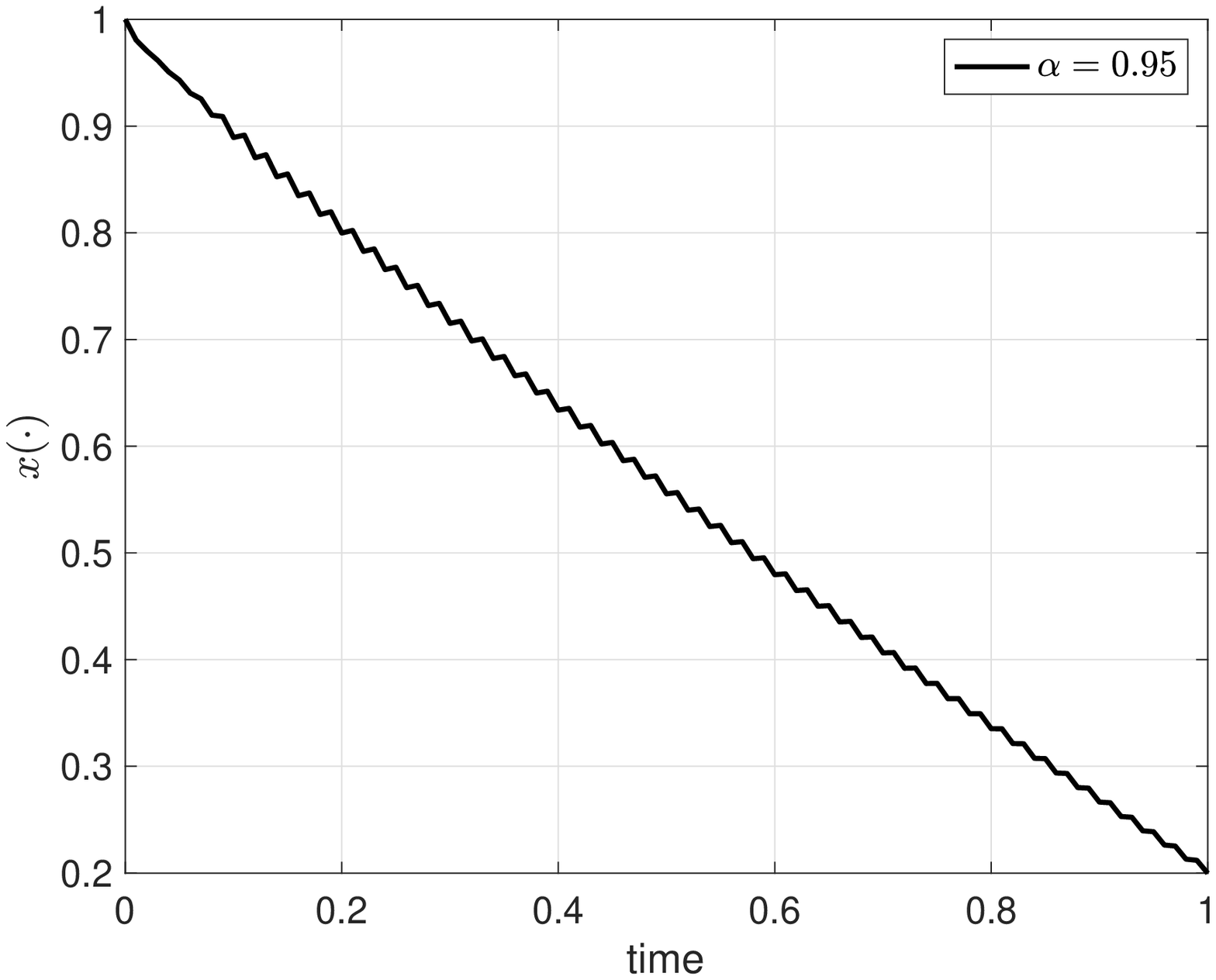}	
\caption{State trajectories when $x_0 = 1$, $f(t,s,x,u) = -0.4 \sin(2\pi x)$, and $g(t,s,x,u) = -x$. Note that the state trajectory shows more singular behavior with small $\alpha \in (0,1)$.}
\label{Fig_1_1_1_1}
\end{figure}

%In this subsection, we provide the precise problem statement of \textbf{(P)}. 
Consider the following Volterra integral equation:
\begin{align}
\label{eq_1}
x(t) = x_0 + \int_0^t \frac{f(t,s,x(s),u(s))}{(t-s)^{1-\alpha}} \dd s + \int_0^t g(t,s,x(s),u(s)) \dd s,~ \textrm{a.e.}~ t \in [0,T],
\end{align}
where $\alpha \in (0,1)$ is the parameter of singularity, $x(\cdot) \in \mathbb{R}^n$ is the state with the initial condition $x_0 \in \mathbb{R}^n$, and $u(\cdot) \in U \subset \mathbb{R}^d$ is the control with $U$ being the control space. In (\ref{eq_1}), $\frac{f(t,s,x,u)}{(t-s)^{1-\alpha}}$ is the singular kernel (with the singularity appearing at $s=t$) and $g(t,s,x,u)$ is the nonsingular kernel, where $f,g:\Delta \times \mathbb{R}^n \times U \rightarrow \mathbb{R}^n$ are generators. We note that $\alpha \in (0,1)$ determines the level of singularity of (\ref{eq_1}); see Figure \ref{Fig_1_1_1_1}. Notice also that $f$ and $g$ are dependent on two time parameters, $t$ and $s$, where their roles are different. While $t$ is the outer time variable to determine the current time, $s$ is the inner time variable describing the path or memory of the state equation from $0$ to $t$. We sometimes use the notation $x(\cdot;x_0,u) := x(\cdot)$ to emphasize the dependence on the initial state and the control.

\begin{assumption}\label{Assumption_2_1}
\begin{enumerate}[(i)]
\item $(U,\rho)$ is a separable metric space, where $U \subset \mathbb{R}^d$ and $\rho$ is the metric induced by the standard Euclidean norm $|\cdot|_{\mathbb{R}^d}$;
\item There is a constant $K \geq 0$ such that for some modulus of continuity $\omega$,
\begin{align*}
\begin{cases}
	|f(t,s,x,u) -  f(t^\prime,s,x,u)| + |g(t,s,x,u) -  g(t^\prime,s,x,u)| \leq K \omega(|t-t^\prime|)(1+|x|), \\
		~~~~~~~~~~ \forall (t,s),(t^\prime,s) \in \Delta, ~x \in \mathbb{R}^n,~ u \in U; 
\end{cases}	
\end{align*}
\item For $p > \frac{1}{\alpha}$, there are nonnegative functions $K_0(\cdot) \in L^{ \frac{1}{\alpha} +}([0,T];\mathbb{R})$ and $K(\cdot) \in L^{ \frac{p}{\alpha p - 1} +}([0,T];\mathbb{R})$, where $L^{p+}([0,T];\mathbb{R}^n) := \cup_{r > p} L^{r}([0,T];\mathbb{R}^n)$ for $1 \leq p < \infty$, such that
\begin{align*}
\begin{cases}
	|f(t,s,x,u) - f(t,s,x^\prime,u^\prime)| + |g(t,s,x,u) - g(t,s,x^\prime,u^\prime)| \leq K(s) (|x-x^\prime| + \rho(u,u^\prime)) , \\
	~~~~~~~~~~ \forall (t,s) \in \Delta,~ x,x^\prime \in \mathbb{R}^n,~ u,u^\prime \in U, \\
	|f(t,s,0,u)| + |g(t,s,0,u)| \leq K_0(s),~ \forall (t,s) \in \Delta,~ u \in U;
\end{cases}	
\end{align*}
\item $f$ and $g$ are of class $C^1$ (continuously differentiable) in $x$, which are bounded and continuous in $(x,u) \in \mathbb{R}^n \times U$.
\end{enumerate}		
\end{assumption}

For $p \geq 1$ and $u_0 \in U$, the space of admissible controls for (\ref{eq_1}) is defined by
\begin{align*}
\mathcal{U}^p[0,T] = \Bigl \{u:[0,T] \rightarrow U~|~ \textrm{$u$ is measurable in $t \in [0,T]$} ~\&~ \rho(u(\cdot),u_0) \in L^p([0,T];\mathbb{R}_+) \Bigr \}
\end{align*}
We state the following lemma; the proof is provided in Appendix \ref{Appendix_B} (see Lemmas \ref{Lemma_B_1} and \ref{Lemma_B_2}).

\begin{lemma}\label{Lemma_2_1}
Let (i)-(iii) of Assumption \ref{Assumption_2_1} hold. Then the following results hold:
\begin{enumerate}[(i)]
\item For any $(x_0,u(\cdot)) \in \mathbb{R}^n \times \mathcal{U}^p[0,T]$, (\ref{eq_1}) admits a unique solution in $C([0,T];\mathbb{R}^n)$, i.e., $x(\cdot;x_0,u) \in C([0,T];\mathbb{R}^n)$, and there is a constant $C \geq 0$ such that 
\begin{align*}
\Bigl \|x (\cdot;x_0,u) \Bigr \|_{L^p([0,T];\mathbb{R}^n)} \leq C \Bigl (1 + |x_0|_{\mathbb{R}^n} + \Bigl \|\rho(u(\cdot),u_0) \Bigr \|_{L^p([0,T];\mathbb{R}_+)} \Bigr )	;
\end{align*}	
\item For any $x_0,x_0^\prime \in \mathbb{R}^n$ and $u(\cdot),u^\prime (\cdot) \in \mathcal{U}^p[0,T]$, there is a constant $C \geq 0$ such that
\begin{align*}
& \Bigl \|x(\cdot;x_0,u) - 	x(\cdot;x_0^\prime,u^\prime) \Bigr \|_{L^p([0,T];\mathbb{R}^n)}  \leq C  |x_0 - x_0^\prime |_{\mathbb{R}^n} \\
&~~~~~ + C \Biggl [ \int_0^T \Bigl ( \int_0^t \frac{|f(t,s,x(s;x_0,u),u(s)) - f(t,s,x(s;x_0,u),u^\prime(s))|}{(t-s)^{1-\alpha}} \dd s \Bigr )^p \dd t \Biggr]^{\frac{1}{p}} \\
&~~~~~ + C \Biggl [ \int_0^T \Bigl ( \int_0^t | g(t,s,x(s;x_0,u),u(s)) - g(t,s,x(s;x_0,u),u^\prime(s)) | \dd s \Bigr)^p  \dd t \Biggr ]^{\frac{1}{p}}.
\end{align*}
\end{enumerate}
\end{lemma}

We introduce the following objective functional:
\begin{align}
\label{eq_2}
J(x_0,u(\cdot)) = \int_0^T l(r,x(r),u(r) \dd r + h(x_0,x(T)).
\end{align}
Then the main objective of this paper is to solve the following optimal control problem:
\begin{align*}
\textbf{(P)}~ \inf_{u(\cdot) \in \mathcal{U}^p[0,T]} J(x_0,u(\cdot)),~\textrm{subject to (\ref{eq_1}),}
\end{align*}
and the state constraints given by
\begin{align}
\label{eq_3}
\begin{cases}
(x_0, x(T;x_0,u)) \in F, & \textrm{(terminal state constraint)},\\
	G^i(t,x(t;x_0,u)) \leq 0,~\forall t \in [0,T],~ i=1,\ldots,m, & \textrm{(inequality state constraint).} 
\end{cases}
\end{align}

\begin{assumption}\label{Assumption_2_2}
	\begin{enumerate}[(i)]
	\item $l:[0,T] \times \mathbb{R}^n \times U \rightarrow \mathbb{R}$ is continuous in $t \in [0,T]$, and is of class $C^1$ in $x$, which is bounded and continuous in $(x,u) \in \mathbb{R}^n \times U$. Moreover, there is a constant $K \geq 0$ such that
	\begin{align*}
	\begin{cases}
	|l(s,x,u) - l(s,x^\prime,u^\prime)| \leq K (|x-x^\prime| + \rho(u,u^\prime)) ,~ \forall s \in [0,T],~ x,x^\prime \in \mathbb{R}^n,~ u,u^\prime \in U, \\
	|l(s,0,u)| \leq K,~ \forall s \in [0,T],~ u \in U;
	\end{cases}	
	\end{align*}
	\item $h : \mathbb{R}^n \times \mathbb{R}^n \rightarrow \mathbb{R}$ is of class $C^1$ in both variables, which are bounded. Let $h_{x}$ and $h_{x_0}$ be partial derivatives of $h$ with respect to $x$ and $x_0$, respectively. Moreover, there is a constant $K \geq 0$ such that
	\begin{align*}
	|h(x_0,x) - h(x_0^\prime,x^\prime)| \leq K ( |x_0 - x_0^\prime |	+ |x^\prime - x^\prime | ),~ \forall (x_0,x),(x_0^\prime,x^\prime) \in \mathbb{R}^n \times \mathbb{R}^n;
	\end{align*}
	\item $F$ is a nonempty closed convex subset of $\mathbb{R}^{2n}$;
	\item For $i=1,\ldots,m$, $G^i:[0,T] \times \mathbb{R}^n \rightarrow \mathbb{R}$ is continuous in $t \in [0,T]$ and is of class $C^1$ in $x$, which is bounded in both variables.
	\end{enumerate}
\end{assumption}

Under Assumptions \ref{Assumption_2_1} and \ref{Assumption_2_2}, the main objective of this paper is to derive the Pontryagin-type maximum principle for \textbf{(P)}, which constitutes the necessary conditions for optimality. Note that Assumptions \ref{Assumption_2_1} and \ref{Assumption_2_2} are crucial for the well-posedness of the state equation in (\ref{eq_1}) by Lemma \ref{Lemma_2_1} (see also Appendix \ref{Appendix_B}) as well as the maximum principle of \textbf{(P)}. Assumptions similar to Assumptions \ref{Assumption_2_1} and \ref{Assumption_2_2} have been used in various optimal control problems and their maximum principles; see \cite{Yong_book, Li_Yong_book, Lin_Yong_SICON_2020, Bettiol_CVOC_2021, Bourdin_arxiv_2016, Bonnans_SVA_2010, Carlson_JOTA_1987, Vega_JOTA_2006, Moon_Automatica_2020_1, Dmitruk_MCRF_2017, Dmitruk_SICON_2014, Bonnans_Vega_JOTA_2013, Vinokurov_SICON_1969, Bourdin_MP_2020, Hamaguchi_Arvix_2021} and the references therein.

\section{Statement of the Maximum Principle}\label{Section_3}

We provide the statement of the maximum principles for \textbf{(P)}. The proof is given in Section \ref{Section_4}. 

\begin{theorem}\label{Theorem_3_1}
	Let Assumptions \ref{Assumption_2_1} and \ref{Assumption_2_2} hold. Suppose that $(\overline{u}(\cdot), \overline{x}(\cdot)) \in \mathcal{U}^p[0,T] \times C([0,T];\mathbb{R}^n)$ is the optimal pair for \textbf{(P)}, i.e., $\overline{u}(\cdot) \in \mathcal{U}^p[0,T]$ and the optimal solution to \textbf{(P)}, where $\overline{x}(\cdot;\overline{x}_0,\overline{u}) := \overline{x}(\cdot) \in C([0,T];\mathbb{R}^n)$ is the corresponding optimal state trajectory of (\ref{eq_1}). Then there exists the tuple $(\lambda,\xi,\theta_1,\ldots,\theta_m)$, where $\lambda \in \mathbb{R}$, $\xi \in \mathbb{R}^{2n}$ with $(\xi_1,\xi_2) \in \mathbb{R}^n \times \mathbb{R}^n$, and $\theta(\cdot) := (\theta_1(\cdot),\ldots,\theta_m(\cdot)) \in \textsc{NBV}([0,T];\mathbb{R}^m)$ with $\theta_i(\cdot) \in  \textsc{NBV}([0,T];\mathbb{R})$ for $i=1,\ldots,m$, such that the following conditions are satisfied:
\begin{itemize}
	\item Nontriviality condition: the tuple $(\lambda,\xi,\theta_1(\cdot),\ldots,\theta_m(\cdot))$ is not trivial, i.e., it holds that \\ $(\lambda,\xi,\theta_1(\cdot),\ldots,\theta_m(\cdot)) \neq 0$, where
	\begin{align*}
	\begin{cases}
	\lambda \geq  0, \\
	\xi = \begin{bmatrix}
 	\xi_1 \\
 	\xi_2
 \end{bmatrix}  \in N_F \Bigl (\begin{bmatrix}
	\overline{x}_0 \\
	\overline{x}(T)
\end{bmatrix} \Bigr ), \\
	\theta_i(\cdot)	\in \textsc{NBV}([0,T];\mathbb{R})~ \textrm{with}~  \|\theta_i(\cdot)\|_{\textsc{NBV}([0,T];\mathbb{R})} = \theta_i(T) \geq 0,~\forall i=1,\ldots,m,
	\end{cases}
	\end{align*}
	with $N_F(x)$ being the normal cone to the convex set $F$ defined in (\ref{eq_4_1}), and $\theta_i(\cdot) \in  \textsc{NBV}([0,T];\mathbb{R})$, $i=1,\ldots,m$, being finite, nonnegative, and monotonically nondecreasing on $[0,T]$;
\item Nonnegativity condition:
	\begin{align*}
	\begin{cases}
	\lambda \geq 0, \\
	\dd \theta_i(s) \geq 0,~ \forall s \in [0,T],~i=1, \ldots, m,
	\end{cases}
	\end{align*}
	where $\dd \theta_i$ denotes the Lebesgue-Stieltjes measure on $[0,T]$ corresponding to $\theta_i$, $i=1,\ldots,m$;
	\item Adjoint equation: there exists a nontrivial $p(\cdot) \in L^p([0,T];\mathbb{R}^n)$ such that $p$ is the unique solution to the following backward Volterra integral equation having singular and nonsingular kernels:
\begin{align*}
	p(t) & = \int_t^T \frac{f_x(r,t,\overline{x}(t),\overline{u}(t))^\top}{(r-t)^{1-\alpha}} p(r) \dd r - \mathds{1}_{[0,T)} (t) \frac{f_x(T,t,\overline{x}(t),\overline{u}(t))^\top}{(T-t)^{1-\alpha}} \Bigl ( \lambda h_x(\overline{x}_0,\overline{x}(T)) + \xi_2^\top \Bigr )^\top \\
&~~~ + \int_t^T g_x(r,t,\overline{x}(t),\overline{u}(t))^\top p(r) \dd r	 - g_x(T,t,\overline{x}(t),\overline{u}(t))^\top \Bigl ( \lambda h_x(\overline{x}_0,\overline{x}(T)) + \xi_2^\top \Bigr )^\top \\
&~~~ - \lambda l_x(t,\overline{x}(t),\overline{u}(t))^\top  - \sum_{i=1}^m G_x^{i}(t,\overline{x}(t))^\top \frac{\dd \theta_i(t)}{\dd t},~  \textrm{a.e.}~ t \in [0,T];
\end{align*}
	\item Transversality condition:
	\begin{align*}	
	0 & \leq 	\Bigl \langle \xi_1, \overline{x}_0 - y_1 \Bigr \rangle_{\mathbb{R}^n \times \mathbb{R}^n} + \Bigl \langle \xi_2, \overline{x}(T) - y_2 \Bigr \rangle_{\mathbb{R}^n \times \mathbb{R}^n},~ \forall y = \begin{bmatrix}
			 y_1 \\
			 y_2
		\end{bmatrix} \in F, \\
	\int_0^T p(t) \dd t &= \xi_1 + \xi_2 + \lambda h_{x_0}(\overline{x}_0,\overline{x}(T))^\top  +\lambda h_x(\overline{x}_0,\overline{x}(T))^\top;
	\end{align*}
	\item Complementary slackness condition:
	\begin{align*}	
	\int_0^T G^i(t,\overline{x}(t;\overline{x}_0,\overline{u})) \dd \theta_i(t) = 0,~ \forall i=1,\ldots,m,
	\end{align*}
	which is equivalent to
	\begin{align*}
	\textsc{supp}(\dd \theta_i(\cdot)) \subset \{ t \in [0,T]~|~ G^i(t,\overline{x}(t;\overline{x}_0,\overline{u})= 0\},~ \forall i=1,\ldots,m,	
	\end{align*}
	where $\textsc{supp}(\dd \theta_i(\cdot))$ denotes the support of the measure $\dd \theta_i$, $i=1,\ldots,m$;
	\item Hamiltonian-like maximum condition: 
	\begin{align*}
&\int_t^T p(r)^\top  \frac{f(r,t,\overline{x}(t),\overline{u}(t))}{(r-t)^{1-\alpha}} \dd r - \mathds{1}_{[0,T)}(t) \Bigl ( \lambda h_x(\overline{x}_0,\overline{x}(T)) + \xi_2^\top \Bigr ) \frac{f(T,t,\overline{x}(t),\overline{u}(t))}{(T-t)^{1-\alpha}} \\
&~~~ + \int_t^T p(r)^\top  g(r,t,\overline{x}(t),\overline{u}(t)) \dd r - \Bigl ( \lambda h_x(\overline{x}_0,\overline{x}(T)) + \xi_2^\top \Bigr )  g(T,t,\overline{x}(t),\overline{u}(t)) \\
&~~~ - \lambda l(t,\overline{x}(t),\overline{u}(t)) \\
& = \max_{u \in U} \Biggl \{ \int_t^T p(r)^\top  \frac{f(r,t,\overline{x}(t),u)}{(r-t)^{1-\alpha}} \dd r - \mathds{1}_{[0,T)}(t) \Bigl ( \lambda h_x(\overline{x}_0,\overline{x}(T)) + \xi_2^\top \Bigr ) \frac{f(T,t,\overline{x}(t),u)}{(T-t)^{1-\alpha}} \\
&~~~ + \int_t^T p(r)^\top  g(r,t,\overline{x}(t),u) \dd r - \Bigl ( \lambda h_x(\overline{x}_0,\overline{x}(T)) + \xi_2^\top \Bigr )  g(T,t,\overline{x}(t),u)  \\
&~~~ - \lambda l(t,\overline{x}(t),u) \Biggr \},~\textrm{a.e. $ t \in [0,T]$.}
\end{align*}
	\end{itemize}
\end{theorem}

Several important remarks are given below.

\begin{remark}\label{Remark_3_3}
The adjoint equation $p$ in Theorem \ref{Theorem_3_1} includes the (strong or distributional (or weak)) derivative of $\theta$, which is expressed as $\frac{\dd \theta_i(t)}{\dd t}$, $i=1,\ldots,m$. Notice that $\theta_i$, $i=1,\ldots,m$, are finite and monotonically nondecreasing by Theorem \ref{Theorem_3_1}, where their corresponding Lebesgue-Stieltjes measures, denoted by $\dd \theta_i$, $i=1,\ldots,m$, are nonnegative, i.e., $\dd \theta_i(s) \geq 0$, for $s \in [0,T]$ and $i=1,\ldots,m$. In fact, $([0,T],\mathcal{B}([0,T]))$, where $\mathcal{B}$ is the Borel $\sigma$-algebra generated by subintervals of $[0,T]$, is a measurable space on which the two nonnegative measures $\dd \theta_i$ and $\dd t $ are defined. Then we can easily see that $\dd \theta_i \ll \dd t$, i.e., $\dd \theta_i$ is absolutely continuous with respect to $\dd t$. That is, $\dd \theta_i(B) = 0$ whenever $\dd t(B) =  0$ for $B \in \mathcal{B}([0,T])$ and $i=1,\ldots,m$ \cite[Appendix C]{Conway_2000_book}. By the Radon-Nikodym theorem (see \cite[Appendix C]{Conway_2000_book}), this implies that there is a unique Radon-Nikodym derivative $\Theta_i(\cdot) \in L^1([0,T];\mathbb{R})$, $i=1,\ldots,m$, such that 
\begin{align*}
\frac{\dd \theta_i(t)}{\dd t} = \Theta_i(t)~\Leftrightarrow~  	\theta_i(t) = \int_0^t \Theta_i(s) \dd s,~ \forall i=1,\ldots,m,~  \textrm{a.e.}~ t\in [0,T].
\end{align*}
Hence, with the Radon-Nikodym derivative $\Theta_i(\cdot)$, $i=1,\ldots,m$, the adjoint equation $p$ in Theorem \ref{Theorem_3_1} can be written as
\begin{align}
\label{eq_3_1}
	p(t) & = \int_t^T \frac{f_x(r,t,\overline{x}(t),\overline{u}(t))^\top}{(r-t)^{1-\alpha}} p(r) \dd r - \mathds{1}_{[0,T)} (t) \frac{f_x(T,t,\overline{x}(t),\overline{u}(t))^\top}{(T-t)^{1-\alpha}} \Bigl ( \lambda h_x(\overline{x}_0,\overline{x}(T)) + \xi_2^\top \Bigr )^\top  \\
&~~~ + \int_t^T g_x(r,t,\overline{x}(t),\overline{u}(t))^\top p(r) \dd r	 - g_x(T,t,\overline{x}(t),\overline{u}(t))^\top \Bigl ( \lambda h_x(\overline{x}_0,\overline{x}(T)) + \xi_2^\top \Bigr )^\top \nonumber \\
&~~~ - \lambda l_x(t,\overline{x}(t),\overline{u}(t))^\top  - \sum_{i=1}^m G_x^{i}(t,\overline{x}(t))^\top \Theta_i(t),~  \textrm{a.e.}~ t \in [0,T]. \nonumber
\end{align}
Note that the well-posedness (existence and uniqueness of the solution) of the adjoint equation in (\ref{eq_3_1}) follows from Theorem \ref{Theorem_3_1} (see also Lemma \ref{Lemma_B_5} in Appendix \ref{Appendix_B}). 
\end{remark}

\begin{remark}
The strategy of the proof for Theorem \ref{Theorem_3_1} is based on the Ekeland variational principle. Moreover, as $U$ is only the (separable) metric space and does not have any algebraic structure, the spike variation technique has to be employed. In contrast to other classical approaches, our proof needs to deal with the Volterra-type variational and adjoint equations having singular and nonsingular kernels in the variational analysis. 
\end{remark}

\begin{remark}\label{Remark_3_2}
The nontrivial tuple $(\lambda, \xi, \dd \theta_1,\ldots,\dd \theta_m, p)$ is a Lagrange multiplier, which is said to be normal when $\lambda > 0$ and abnormal when $\lambda = 0$. In the normal case, we may assume the Lagrange multiplier to have been normalized so that $\lambda =1$.
\end{remark}

\begin{remark}
The necessary conditions in Theorem \ref{Theorem_3_1} are of interest only when the terminal state constraint is nondegenerate in the sense that $G_x^i(t,\overline{x}(t))^\top \neq 0$ whenever $G^i(t,\overline{x}(t)) = 0$ for all $t \in [0,T]$ and $i=1,\ldots,m$. 
%This is because when $G_x^i(\tilde{t},\overline{x}(\tilde{t}))^\top \neq 0$ and $G^i(\tilde{t},\overline{x}(\tilde{t})) = 0$ for some $\tilde{t} \in [0,T]$, we may choose $\lambda = \mathds{1}_{t = \tilde{t}}(t)$ with $p \equiv 0$
A similar remark is given in \cite[page 330, Remarks (b)]{Vinter_book} for the classical state-constrained optimal control problem for ordinary state equations. 
\end{remark}

\begin{remark}\label{Remark_3_4}
Without the state constraints in (\ref{eq_3}), Theorem \ref{Theorem_3_1} holds with $\lambda = 1$, $\xi = 0$, and $\theta = 0$. This is equivalent to the following statement (see also \cite[Theorem 4.3]{Lin_Yong_SICON_2020} for the case with singular kernels only): If $(\overline{u}(\cdot), \overline{x}(\cdot)) \in \mathcal{U}^p[0,T] \times C([0,T];\mathbb{R}^n)$ is the optimal pair for \textbf{(P)}, then the following conditions hold:
\begin{itemize}
\item Adjoint equation:	$p(\cdot) \in L^p([0,T];\mathbb{R}^n)$ is the unique solution of the following backward Volterra integral equation having singular and nonsingular kernels:
\begin{align*}
	p(t) & = \int_t^T \frac{f_x(r,t,\overline{x}(t),\overline{u}(t))^\top}{(r-t)^{1-\alpha}} p(r) \dd r - \mathds{1}_{[0,T)} (t) \frac{f_x(T,t,\overline{x}(t),\overline{u}(t))^\top}{(T-t)^{1-\alpha}}  h_x(\overline{x}_0,\overline{x}(T))^\top \\
&~~~ + \int_t^T g_x(r,t,\overline{x}(t),\overline{u}(t))^\top p(r) \dd r	 - g_x(T,t,\overline{x}(t),\overline{u}(t))^\top  h_x(\overline{x}_0,\overline{x}(T))^\top  \\
&~~~ - l_x(t,\overline{x}(t),\overline{u}(t))^\top,~  \textrm{a.e.}~ t \in [0,T];
\end{align*}
\item Hamiltonian-like maximum condition: 
	\begin{align*}
&\int_t^T p(r)^\top  \frac{f(r,t,\overline{x}(t),\overline{u}(t))}{(r-t)^{1-\alpha}} \dd r - \mathds{1}_{[0,T)}(t)  h_x(\overline{x}_0,\overline{x}(T))  \frac{f(T,t,\overline{x}(t),\overline{u}(t))}{(T-t)^{1-\alpha}} \\
&+ \int_t^T p(r)^\top  g(r,t,\overline{x}(t),\overline{u}(t)) \dd r - h_x(\overline{x}_0,\overline{x}(T))   g(T,t,\overline{x}(t),\overline{u}(t))   -  l(t,\overline{x}(t),\overline{u}(t)) \\
& = \max_{u \in U} \Biggl \{ \int_t^T p(r)^\top  \frac{f(r,t,\overline{x}(t),u)}{(r-t)^{1-\alpha}} \dd r - \mathds{1}_{[0,T)}(t) h_x(\overline{x}_0,\overline{x}(T))  \frac{f(T,t,\overline{x}(t),u)}{(T-t)^{1-\alpha}} \\
&~~~ + \int_t^T p(r)^\top  g(r,t,\overline{x}(t),u) \dd r - h_x(\overline{x}_0,\overline{x}(T))  g(T,t,\overline{x}(t),u)   -  l(t,\overline{x}(t),u) \Biggr \},~\textrm{a.e. $ t \in [0,T]$.}
\end{align*}
\end{itemize}
\end{remark}

\begin{remark}
By taking  $f \equiv 0$ in Theorem \ref{Theorem_3_1}, we can obtain the maximum principle for classical Volterra integral equations with nonsingular kernels only. Note that Theorem \ref{Theorem_3_1} is different from the classical maximum principles for Volterra integral equations with nonsingular kernels only studied in the existing literature   (e.g. \cite[Theorem 1]{Bonnans_SVA_2010} and \cite{Dmitruk_MCRF_2017, Dmitruk_SICON_2014, Medhin_JMAA_1988}),  where Theorem \ref{Theorem_3_1} does not need differentiability of kernels with respect to time variables and the adjoint equation in Theorem \ref{Theorem_3_1} is expressed by the integral form. 
\end{remark}

\section{Examples}\label{Section_5}  

In this section, we provide two examples of \textbf{(P)}.

\begin{example}	\label{Example_1}
\normalfont
Consider the minimization of the following objective functional
\begin{align*}
J(x_0,u(\cdot)) = \int_0^3 [ x(s) + \frac{1}{2} u(s)^2 ] \dd s + (x_0 + x(3)),	
\end{align*}
subject to the Volterra integral equation with singular and nonsingular kernels given by
\begin{align}
\label{eq_s_4_1}
x(t) = x_0 + \int_0^t \frac{u(s)}{(t-s)^{1-\alpha}} \dd s	+ \int_0^t u(s) \dd s,~\textrm{a.e. $ t \in [0,3]$,}
\end{align}
and the state constraints
\begin{align}
\label{eq_s_4_2}
%x_0 = 10,~ G(t,x(t)) = - x(t) - \Bigl (\frac{t^2}{5} + 20 \Bigr ) \leq 0,~ \forall t \in [0,3].
\begin{cases}
	(x_0,x(3)) \in F=\{10\} \times \{-16\}, & \textrm{(terminal state constraint)}, \\
G(t,x(t)) = - x(t) - \Bigl (\frac{t^2}{5} + 20 \Bigr ) \leq 0,~ \forall t \in [0,3], & \textrm{(inequality state constraint)}.
\end{cases}
\end{align}
We assume that the control space $U$ is an appropriate sufficiently large compact subset of $\mathbb{R}^d$ to satisfy Assumption \ref{Assumption_2_2}.

Note that $F$ is singleton, which is closed and convex. Hence, by (\ref{eq_s_4_2}), we can choose $\xi = 0$. This implies that the (candidate) optimal state trajectory holds $\overline{x}_0 = 10$ and $\overline{x}(3) = -10$. In addition, the transversality condition leads to $\int_0^3 p(t) \dd t = 2 \lambda$.  Assume by contradiction that $\lambda = 0$. Then the adjoint equation holds that $p(t) = \frac{\dd \theta(t)}{\dd t}$. This implies $\int_0^3 p(t) \dd t = \int_0^3 \dd \theta(t) =  \theta(3)  - \theta(0) =  \theta(3) = 0$, which, by the fact that $\theta(0) = 0$ and $\theta$ is monotonically nondecreasing, contradicts the nontriviality condition of $\theta$ as well as the adjoint equation $p$ in Theorem \ref{Theorem_3_1}. Therefore, $\lambda \neq 0$, and we may take the normalized case with $\lambda = 1$. Based on the preceding discussion and by Theorem \ref{Theorem_3_1}, the following conditions hold:
\begin{itemize}
\item Nontriviality and nonnegativity conditions:
\begin{itemize}
\item $\lambda = 1$ and $\theta(\cdot)	\in \textsc{NBV}([0,3];\mathbb{R})$ with $\|\theta(\cdot)\|_{\textsc{NBV}([0,3];\mathbb{R})} = \theta(3) \geq 0$, $\theta$ being finite and monotonically nondecreasing on $[0,3]$, and $\dd \theta(t) \geq 0$ for $t \in [0,3]$;
\end{itemize}
\item Adjoint equation:
\begin{align}
\label{eq_s_4_3}
p(t) = -  1 + \frac{\dd \theta(t)}{\dd t},~  \textrm{a.e.}~ t \in [0,3];
\end{align}
\item Transversality condition: 
\begin{align}
\label{eq_s_4_4}
\int_0^3 p(t) \dd t = \int_0^3 \Bigl [ -1 + \frac{ \dd \theta(t)}{ \dd t} \Bigr ] \dd t = -3 + \theta(3) = 2~ \Rightarrow~ \theta(3) = 5 > 0;
\end{align}
\item Complementary slackness condition:
\begin{align}	
\label{eq_s_4_5}
	\int_0^3 \Bigl [ - \overline{x}(t) - \Bigl (\frac{t^2}{5} + 20 \Bigr ) \Bigr ] \dd \theta(t) = 0	;
\end{align}
\item Hamiltonian-like maximum condition: the first-order optimality condition implies
\begin{align}
\label{eq_s_4_6}
\overline{u}(t)	= - 1 + \int_t^3 p(r) \dd r - \frac{\mathds{1}_{[0,3)}(t)}{(3-t)^{1-\alpha}} + \int_t^3  \frac{p(r)}{(r-t)^{1-\alpha}} \dd r,~\textrm{a.e. $ t \in [0,3]$.}
\end{align}
\end{itemize}

The numerical simulation results of Example \ref{Example_1} with $\alpha = 0.8$  and $\alpha = 0.5$ are given in Figures \ref{Fig_1} and \ref{Fig_11111}. One can easily observe that for each case, the optimal state trajectory holds the terminal condition as well as the inequality constraint in (\ref{eq_s_4_2}). In addition, $\theta(\cdot) \in \textsc{NBV}([0,3];\mathbb{R})$, where $\theta$ is finite and monotonically nondecreasing on $[0,3]$ and $\dd \theta(t) \geq 0$ for $t \in [0,3]$, and the adjoint equation holds $p(\cdot) \in L^p([0,3];\mathbb{R})$. The (candidate) optimal solution is obtained from the Hamiltonian-like maximum condition in (\ref{eq_s_4_6}). Note that the numerical approach that we adopt is as follows:
\begin{enumerate}
\setlength{\itemindent}{0.2in}
\item[(s.1)] Given $\theta(0) = 0$ and $\theta(3) > 0$, provide a guess of the measure $\dd \theta$ and then construct $\theta$;
\item[(s.2)] Compute the adjoint equation in (\ref{eq_s_4_3});
\item[(s.3)] Compute the optimal solution in (\ref{eq_s_4_6});
\item[(s.4)] Compute the controlled state equation in (\ref{eq_s_4_1}) under the optimal solution (\ref{eq_s_4_6}), which needs to satisfy the terminal and inequality constraints in (\ref{eq_s_4_2}); 
\item[(s.5)] Check the complementary slackness condition in (\ref{eq_s_4_5}) and the transversality condition in (\ref{eq_s_4_4});
\item[(s.6)] If the constraints and conditions in (s.4) and (s.5) hold, stop the algorithm. Otherwise, we iterate (s.1)-(s.5).
\end{enumerate}
\begin{figure}[t]
\centering
\includegraphics[scale=0.28]{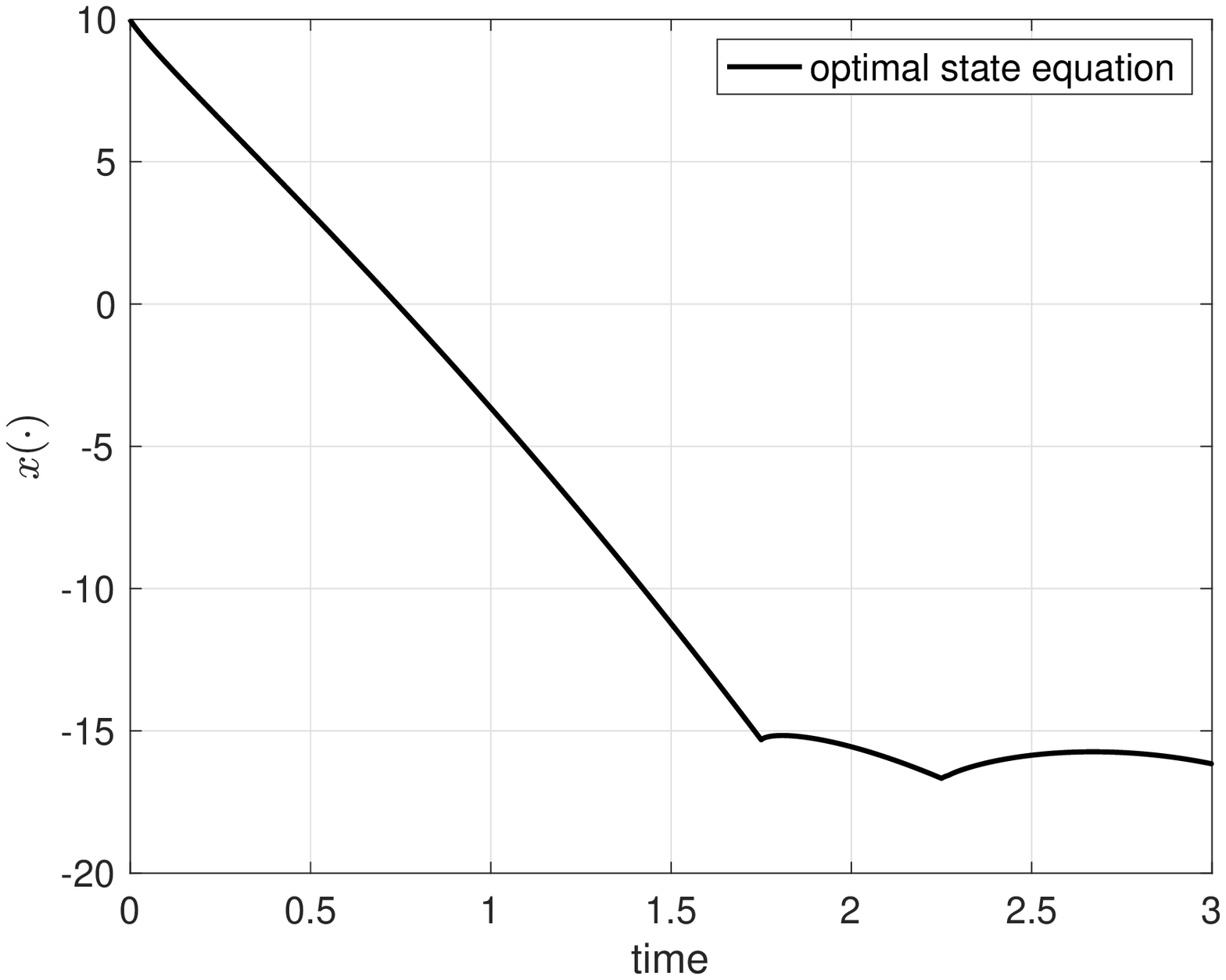}
\includegraphics[scale=0.28]{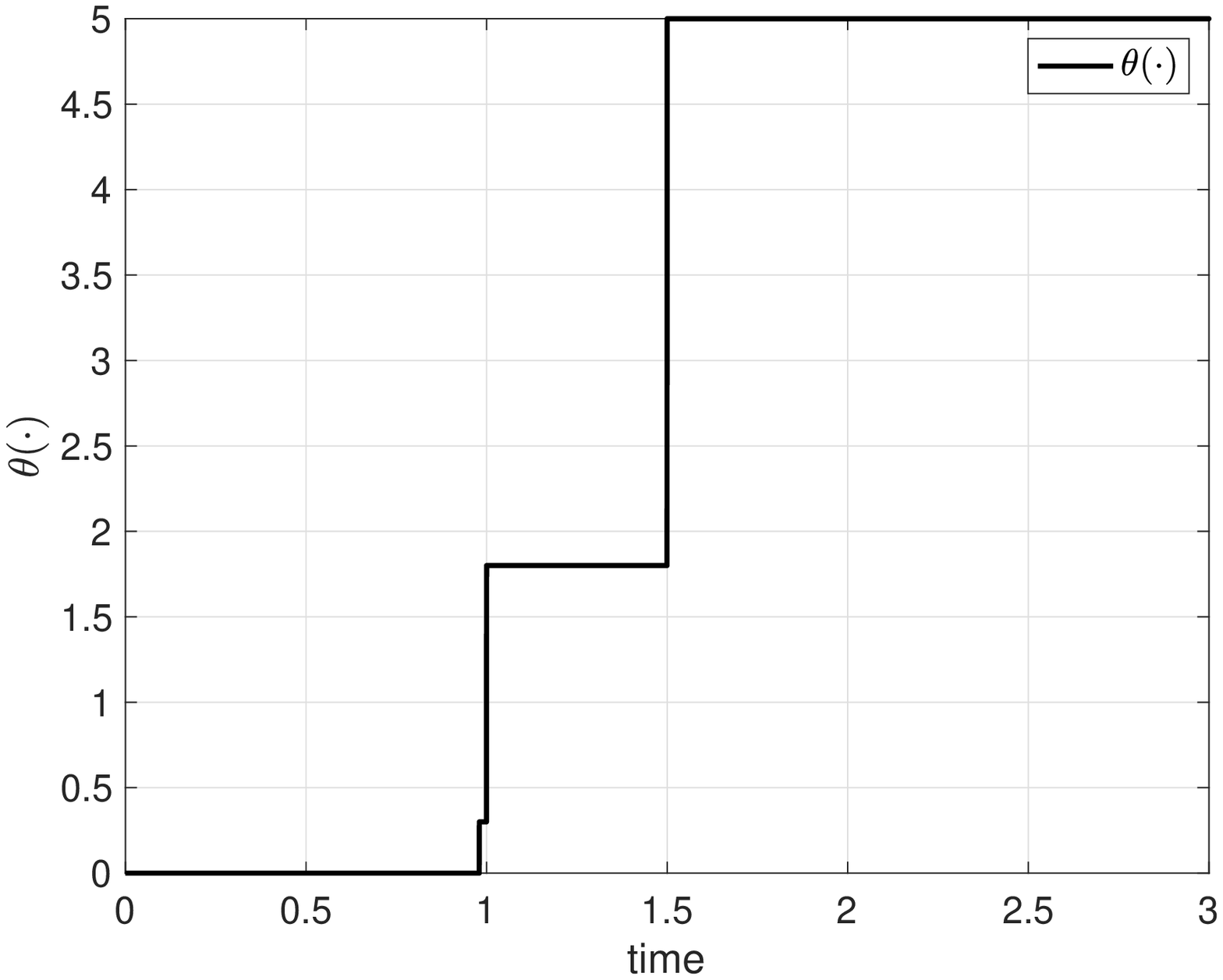}
\includegraphics[scale=0.28]{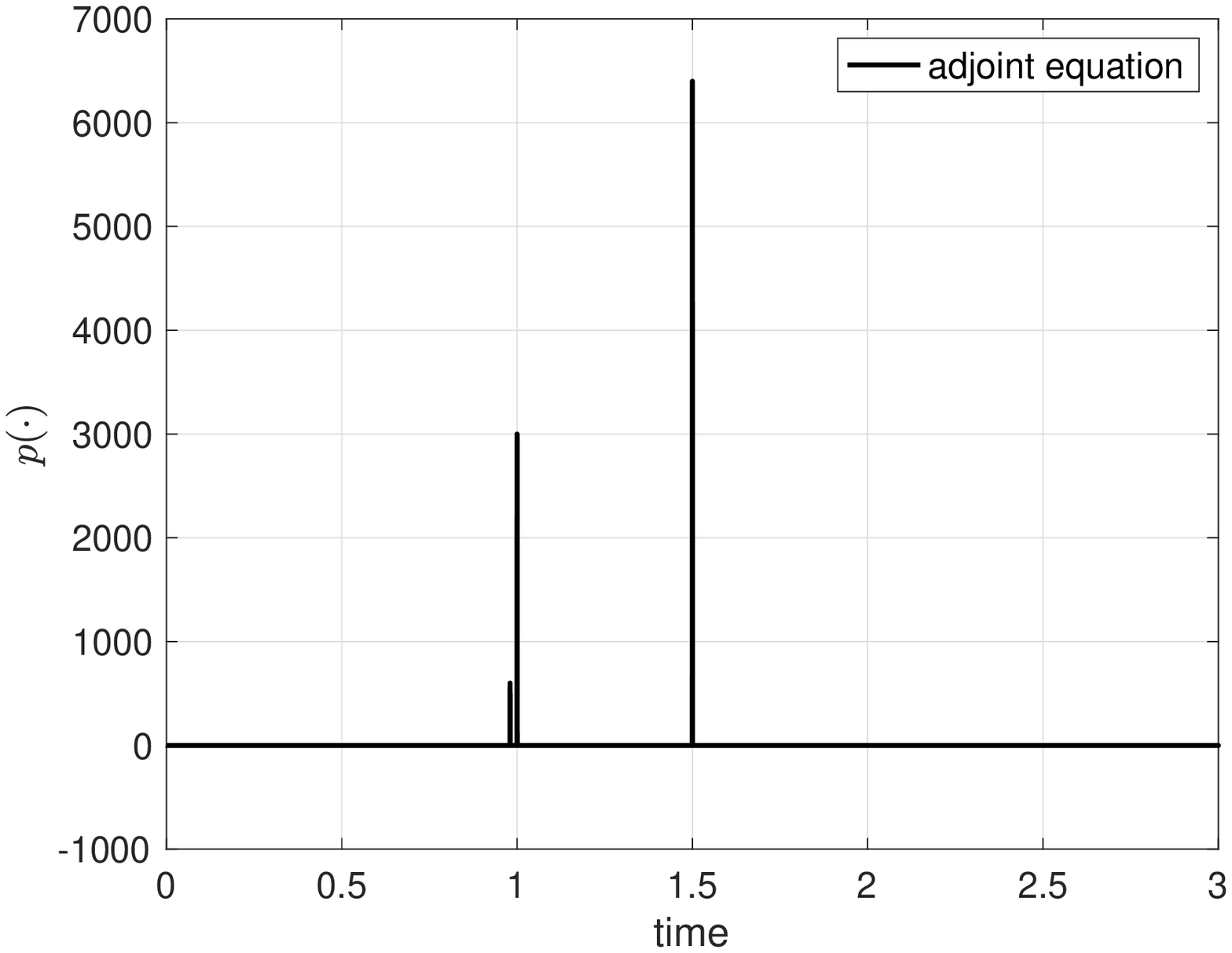}
\includegraphics[scale=0.28]{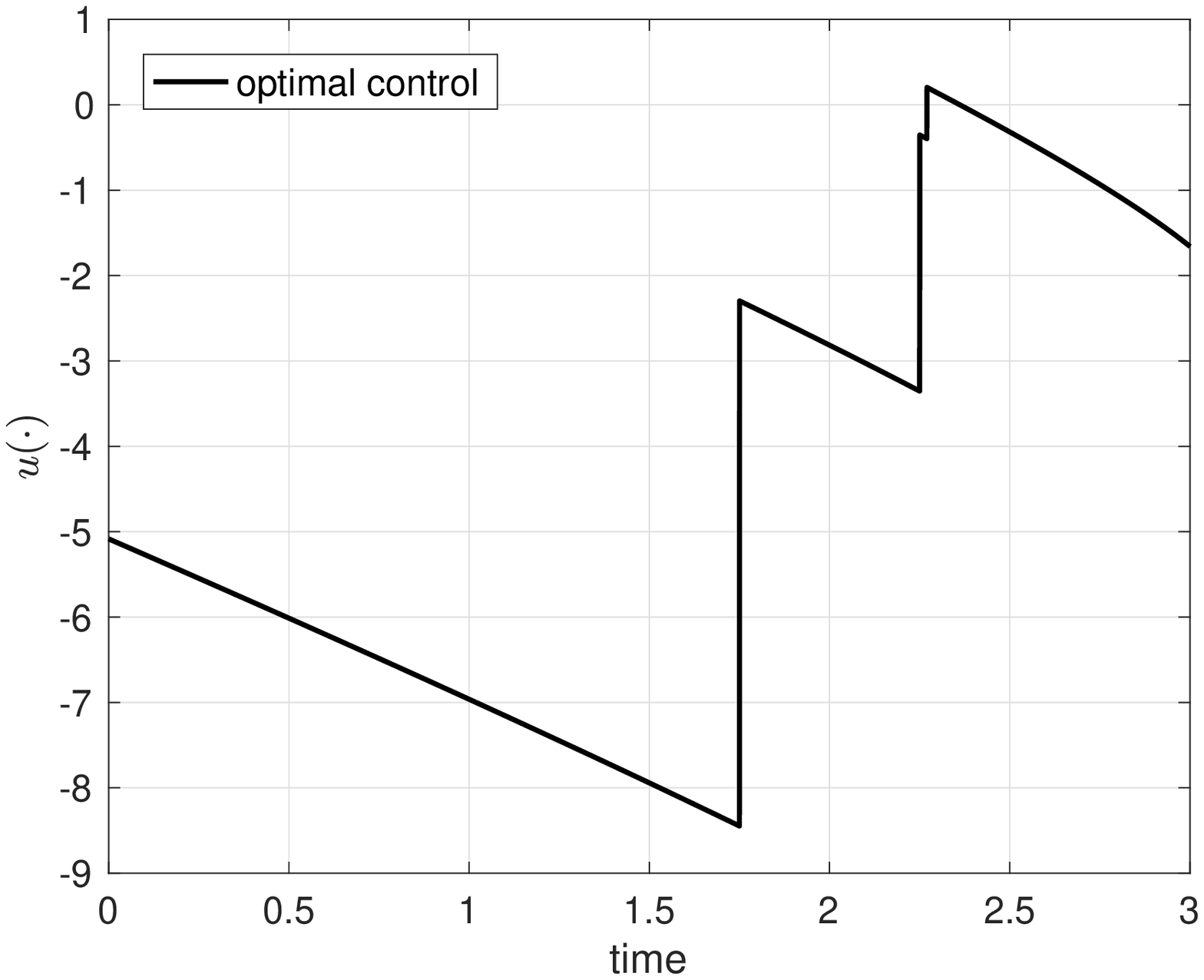}
\includegraphics[scale=0.28]{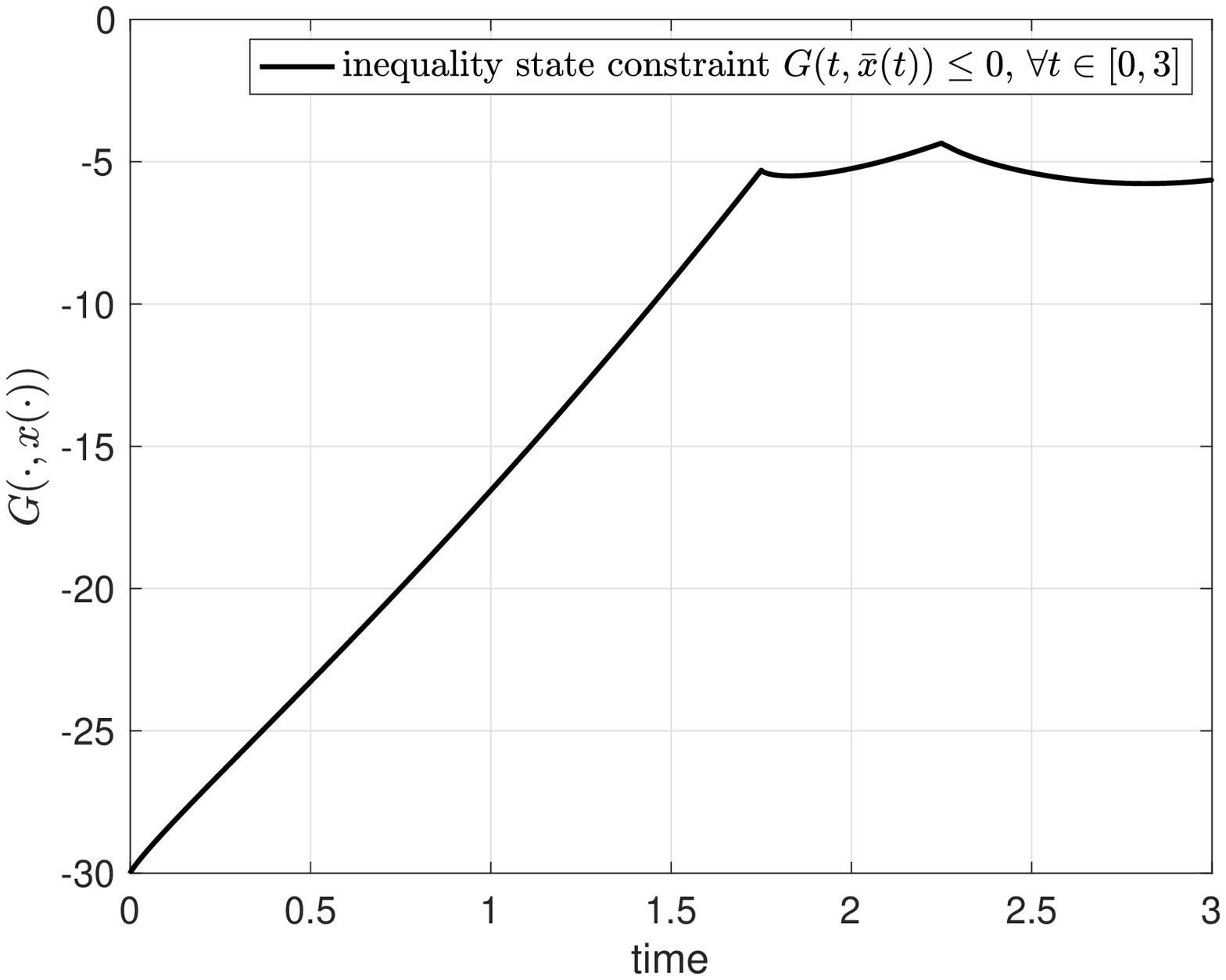}
\caption{Simulation results of Example \ref{Example_1} with $\alpha = 0.8$.}
\label{Fig_1}
\end{figure}

\begin{figure}[t]
\centering
\includegraphics[scale=0.28]{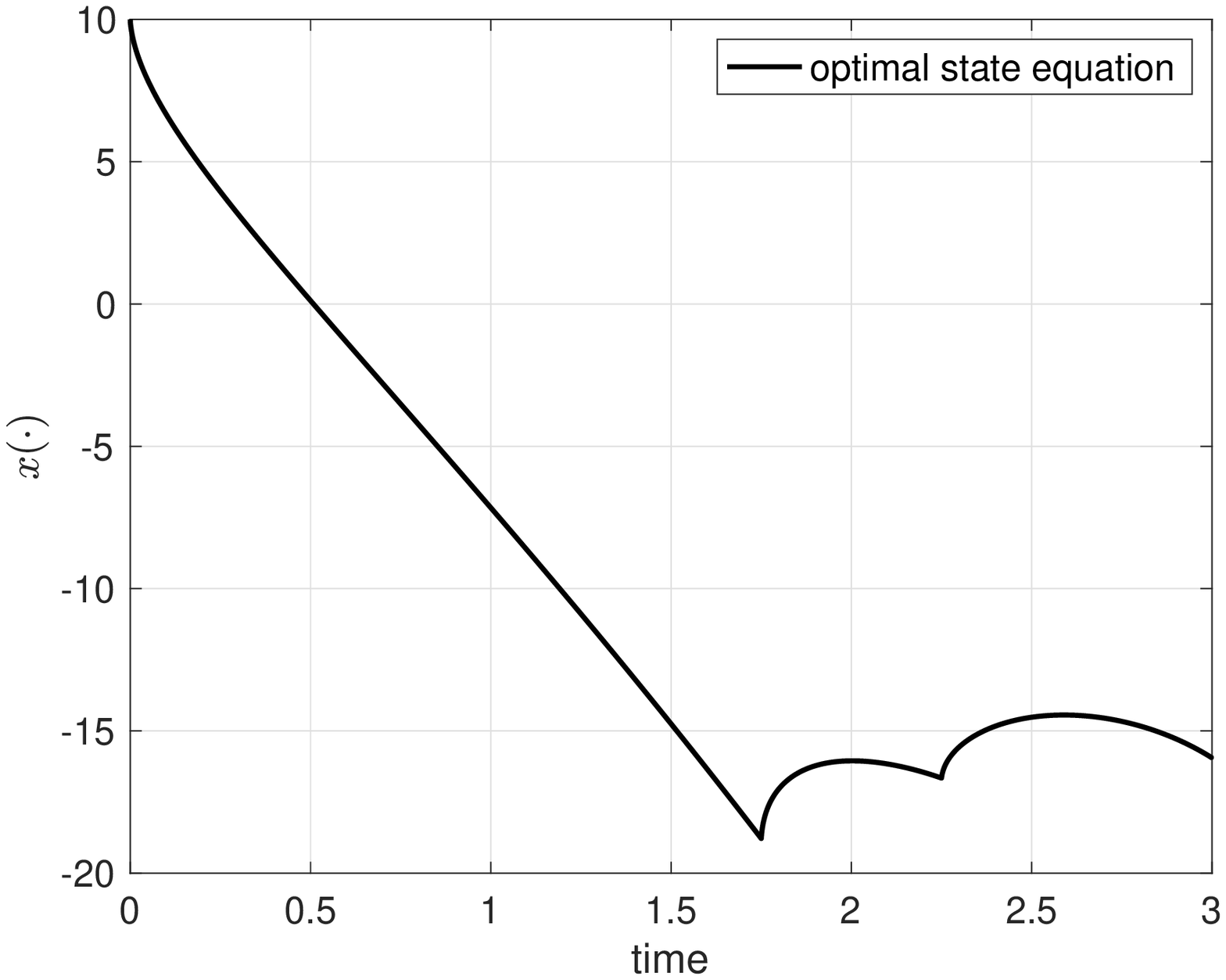}
\includegraphics[scale=0.28]{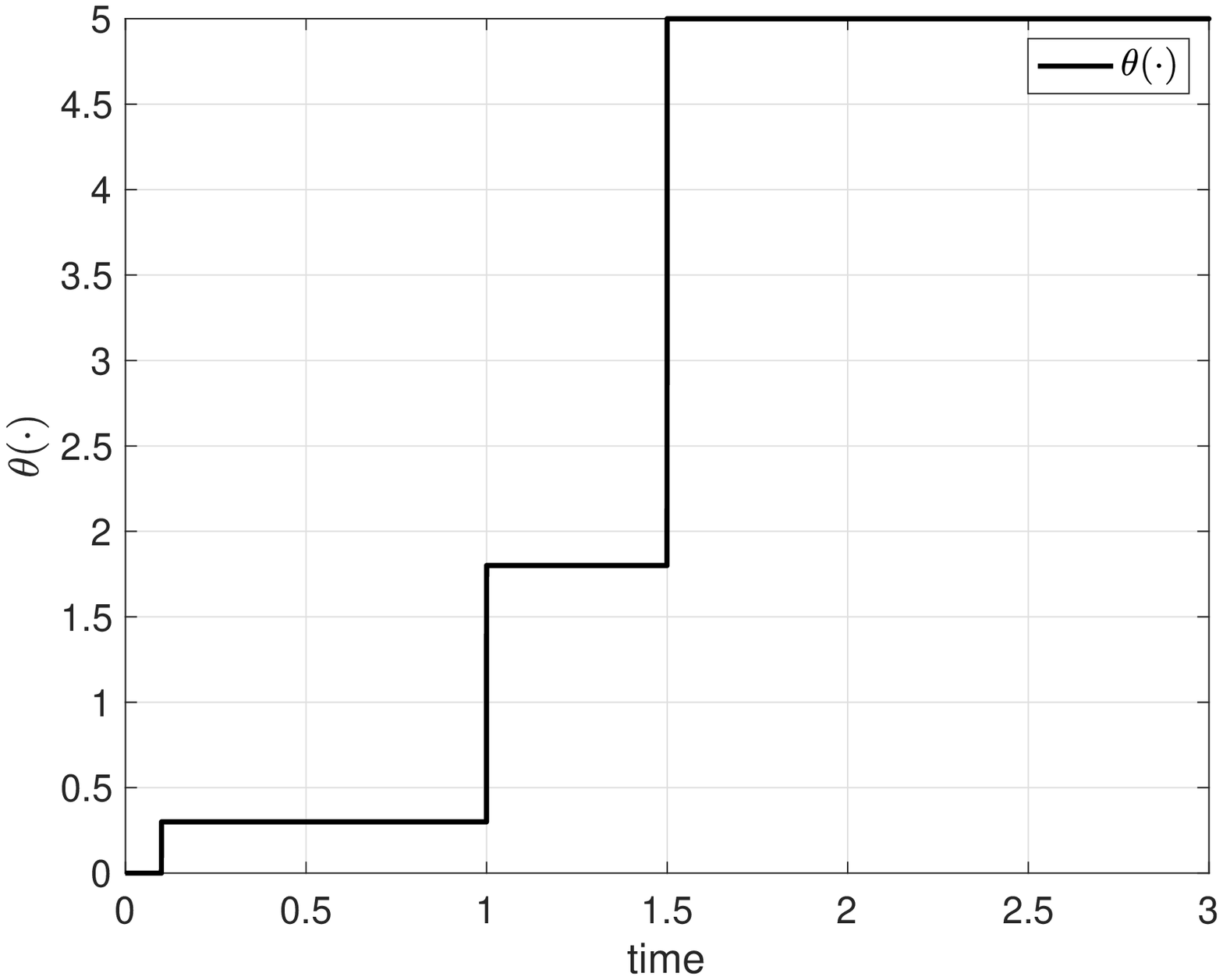}
\includegraphics[scale=0.28]{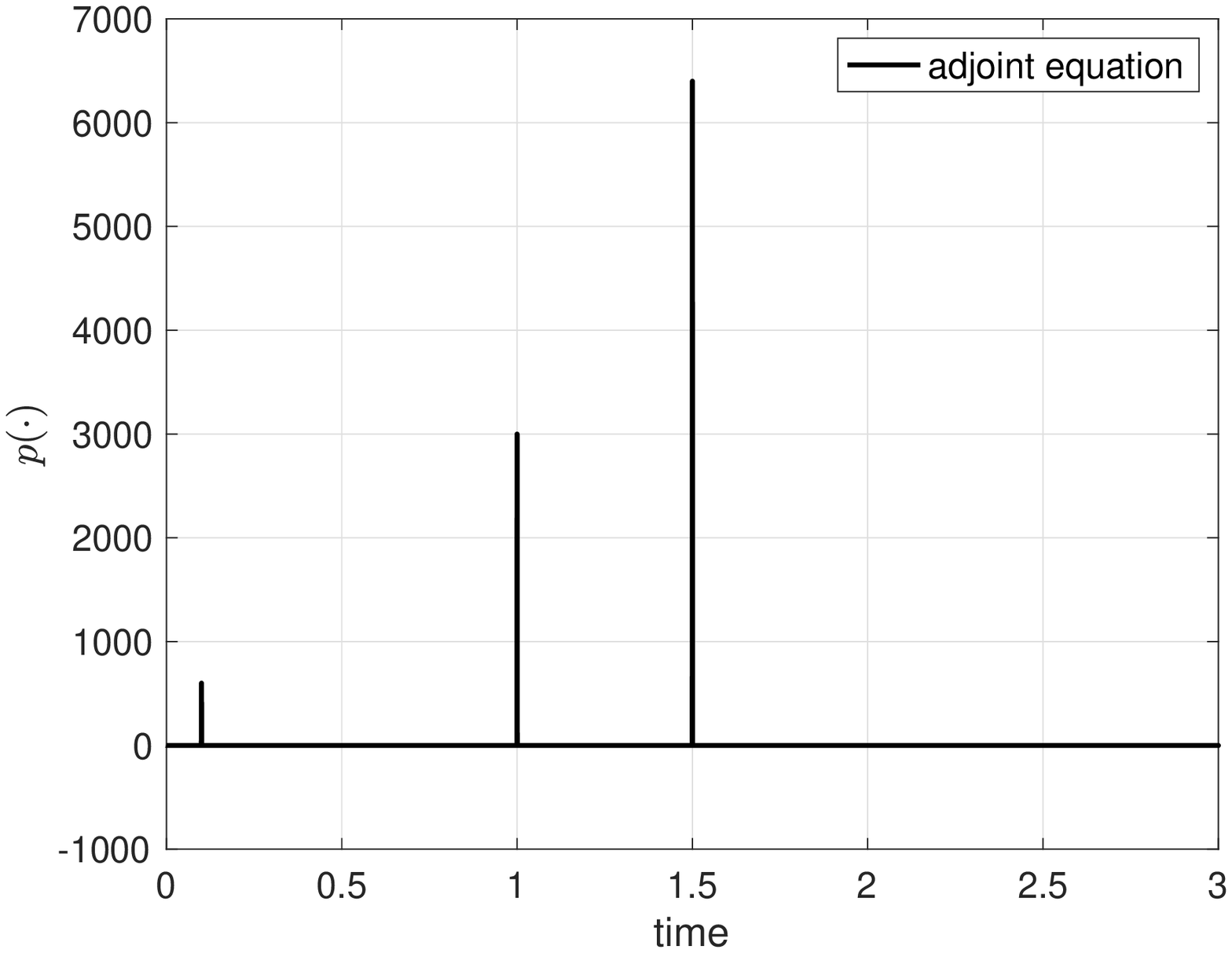}
\includegraphics[scale=0.28]{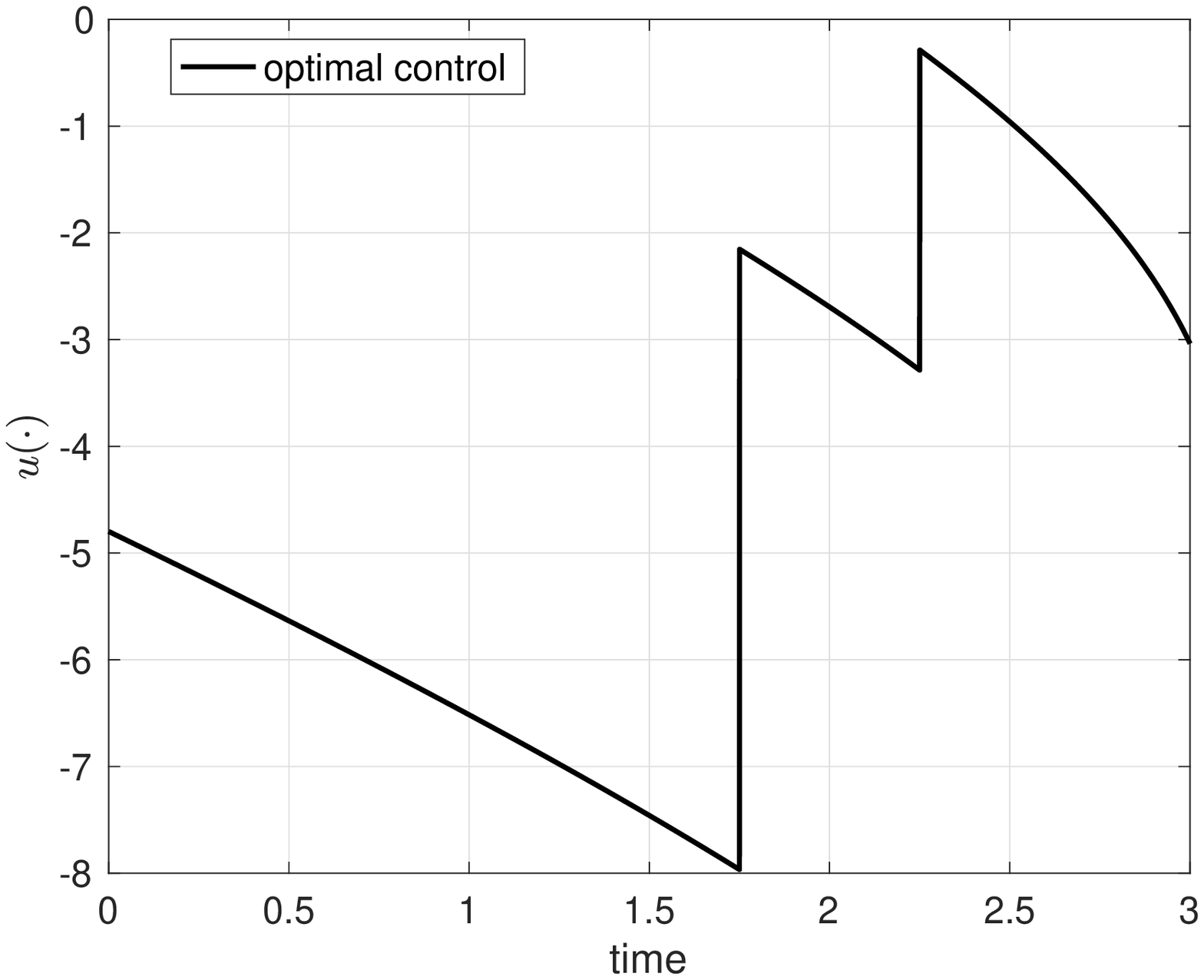}
\includegraphics[scale=0.28]{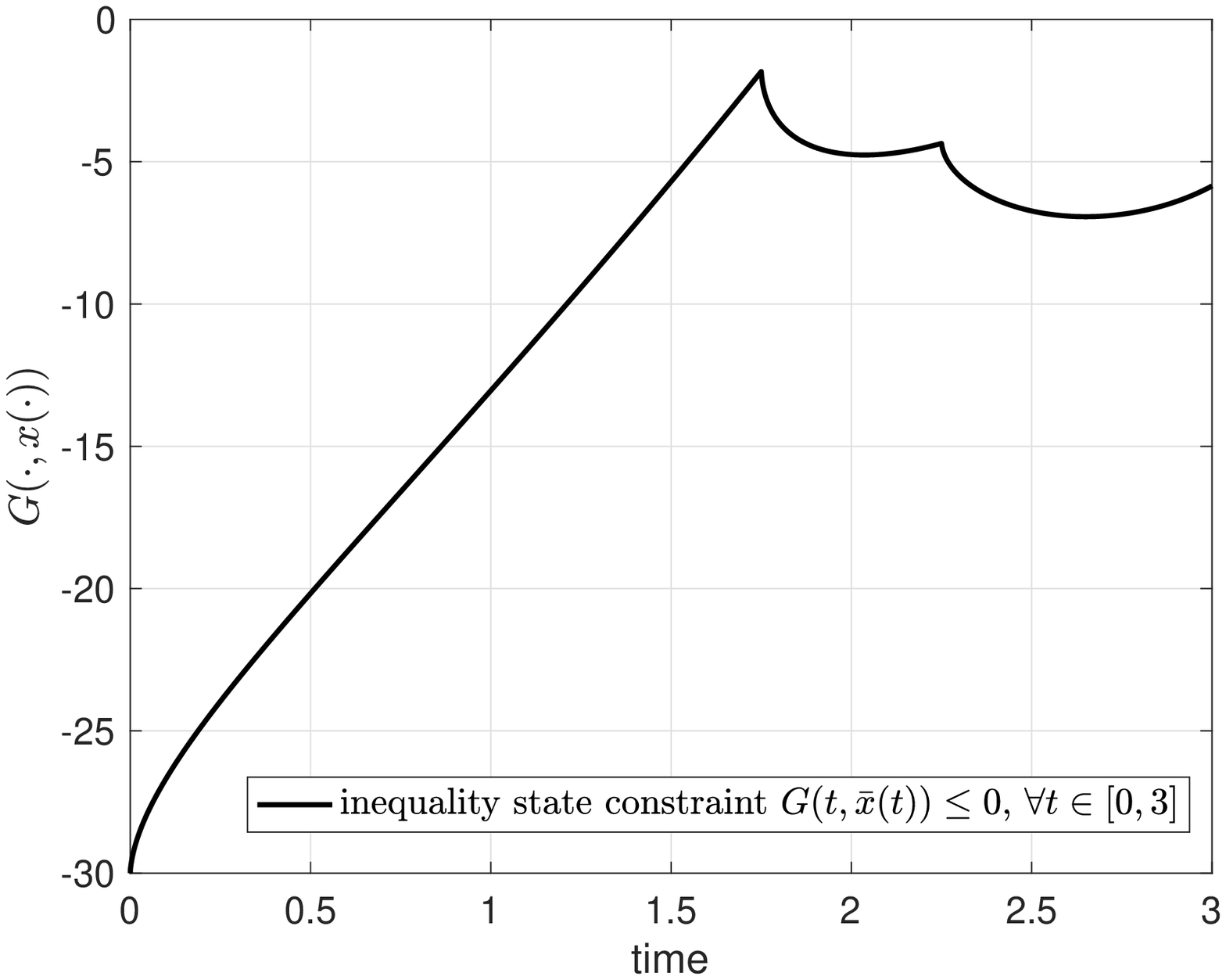}
\caption{Simulation results of Example \ref{Example_1} with $\alpha = 0.5$.}
\label{Fig_11111}
\end{figure}
\end{example}

\begin{example}	
\label{Example_2}
\normalfont
We consider the linear-quadratic problem of \textbf{(P)} without state constraints. The state equation and the objective functional are given by
\begin{align*}
x(t) & = x_0 + \int_0^t \frac{A_1(t,s) x(s) + B_1 u(s)}{(t-s)^{1-\alpha}} \dd s	+ \int_0^t \Bigl [ A_2(t,s) x(s) + B_2(t,s) u(s) \Bigr ] \dd s, \\
J(x_0,u(\cdot)) & = \frac{1}{2} \int_0^T \Bigl [ \langle x(s), Q(s) x(s) \rangle + \langle u(s), R(s) u(s) \rangle \Bigr ] \dd s + \frac{1}{2} \langle x(T), M x(T) \rangle,
\end{align*}
where (ii) of Assumption \ref{Assumption_2_1} holds and (see Lemmas \ref{Lemma_B_1}-\ref{Lemma_B_5_2323232} in Appendix \ref{Appendix_B})
\begin{align*}
\begin{cases}
%\textrm{(ii) of Assumption \ref{Assumption_2_1} holds,} \\
A_1(\cdot,\cdot),A_2(\cdot,\cdot) \in L^{\infty}(\Delta;\mathbb{R}^{n \times n}),~ B_1(\cdot,\cdot),B_2(\cdot,\cdot) \in L^{\infty}(\Delta;\mathbb{R}^{n \times d}), \\
Q(\cdot) \in L^{\infty}([0,T];\mathbb{R}^{n \times n}),~ R(\cdot) \in L^{\infty}([0,T];\mathbb{R}^{d \times d}),~ M \in \mathbb{R}^{n \times n}, \\
Q(t) = Q(t)^\top  \geq 0,~ R(t) = R(t)^\top > c I_d~(c > 0),~ M=M^\top,~ \forall t \in [0,T].
\end{cases}	
\end{align*}
We further assume that the state space $X$ and the control space $U$ are appropriate sufficiently large compact subsets of $\mathbb{R}^n$ and $\mathbb{R}^d$, respectively, to satisfy Assumption \ref{Assumption_2_2}.

By Remark \ref{Remark_3_4} and the first-order optimality condition, the corresponding optimal solution is as follows:
\begin{align*}	
\overline{u}(t) & = R(t)^{-1} \Biggl [ \int_t^T \frac{B_1(r,t)^\top p(r)}{(r-t)^{1-\alpha}} \dd r - \mathds{1}_{[0,T)}(t) \frac{B_1(T,t)^\top M \overline{x}(T)}{(T-t)^{1-\alpha}} \\
&~~~~~ + \int_t^T B_2(r,t)^\top p(r) \dd r - B_2(T,t)^\top M \overline{x}(T) \Biggr ],~\textrm{a.e. $ t \in [0,T]$,}
\end{align*}
where $p$ is the adjoint equation given by
\begin{align*}
p(t) & = \int_t^T \frac{A_1(r,t)^\top p(r)}{(r-t)^{1-\alpha}} \dd r - \mathds{1}_{[0,T)}(t)	\frac{A_1(T,t)^\top M \overline{x}(T)}{(T-t)^{1-\alpha}} \\
&~~~ + \int_t^T A_2(r,t)^\top p(r) \dd r - A_2(T,t)^\top M \overline{x}(T) - Q(t) \overline{x}(t),~\textrm{a.e. $ t \in [0,T]$.}
\end{align*}
Assume that $T=2$, $x_0 = 1$, $A_1 = -1$, $A_2 = 0.2$, $B_1 = 2$, $B_2 = 0.1$, $M=1$, $R=1$, and $Q=0$. By applying the shooting method \cite{Bourdin_MP_2020}, the numerical simulation results are obtained in Figures \ref{Fig_3} and \ref{Fig_323423423}.
\begin{figure}[t]
\centering
\includegraphics[scale=0.28]{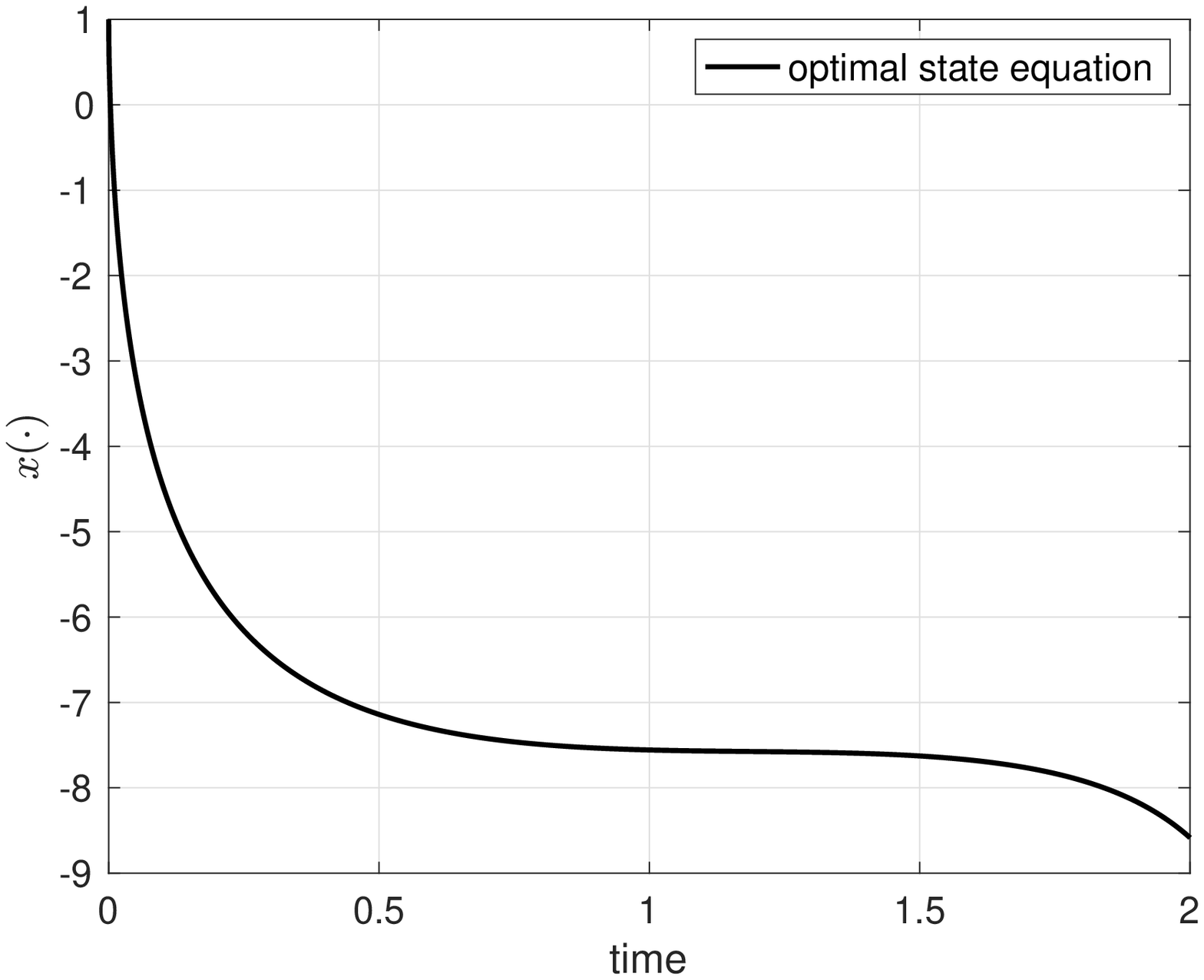}
\includegraphics[scale=0.28]{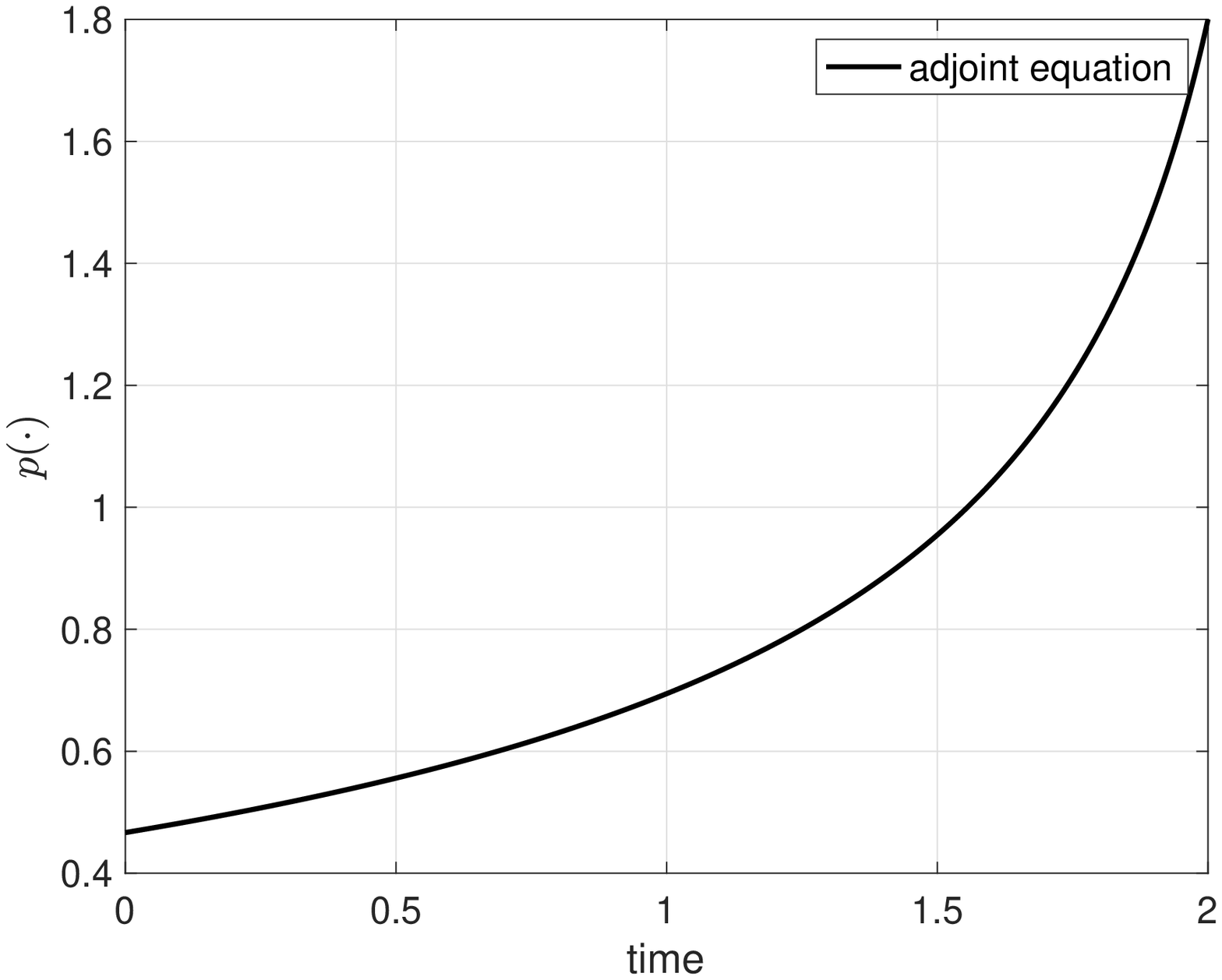}
\includegraphics[scale=0.28]{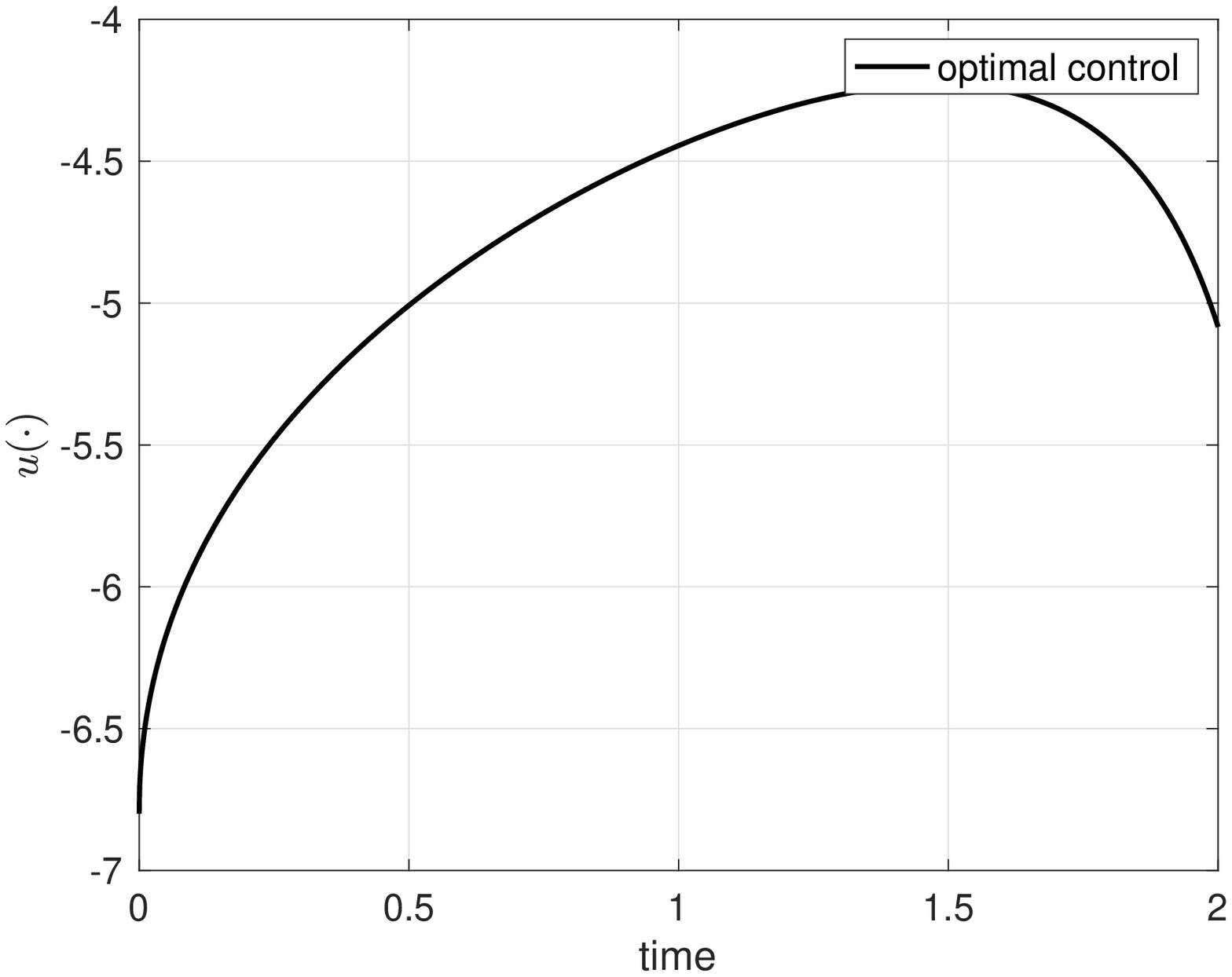}	
\caption{Simulation results of Example \ref{Example_2} with $\alpha = 0.5$.}
\label{Fig_3}
\end{figure}

\begin{figure}[t]
\centering
\includegraphics[scale=0.28]{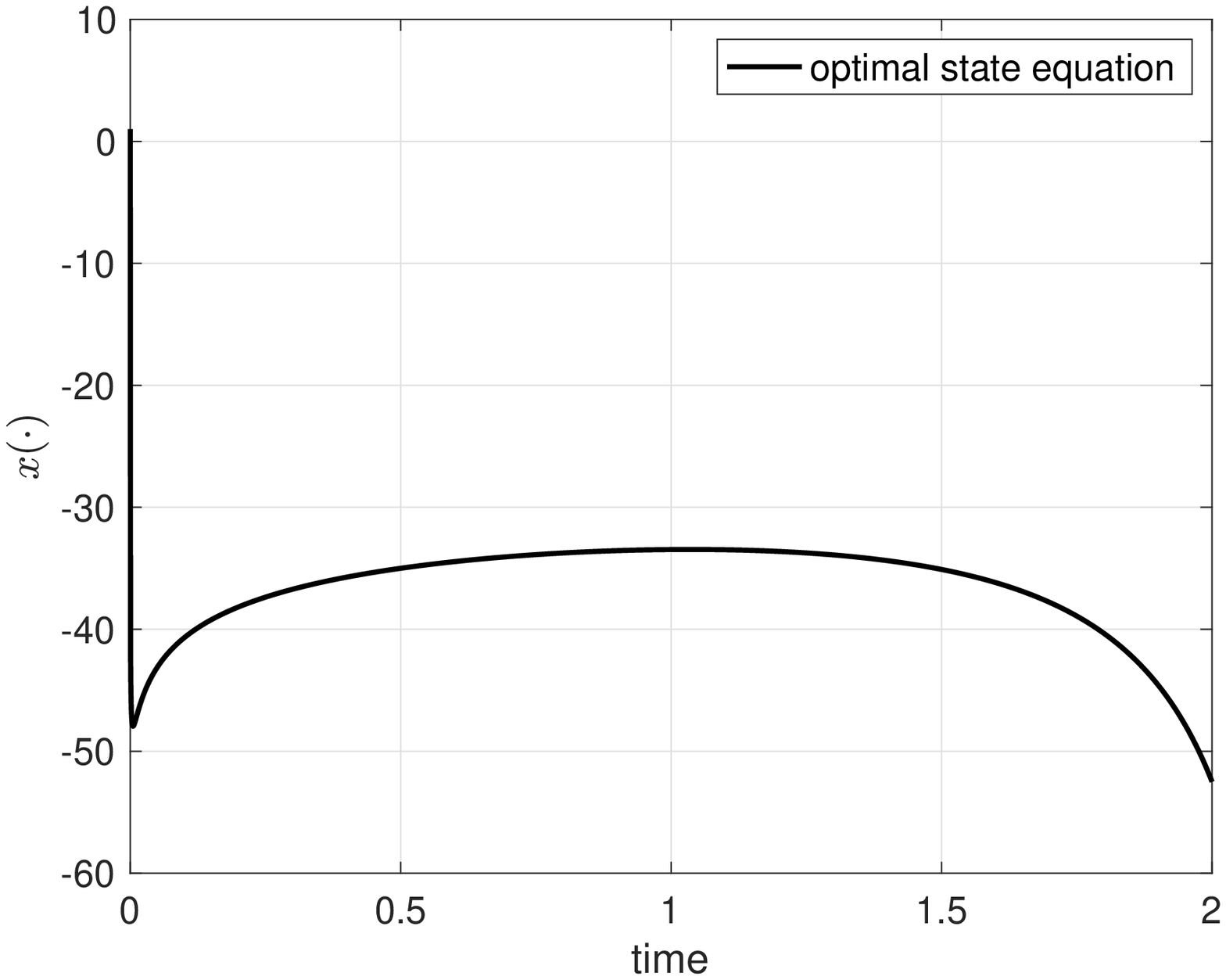}
\includegraphics[scale=0.28]{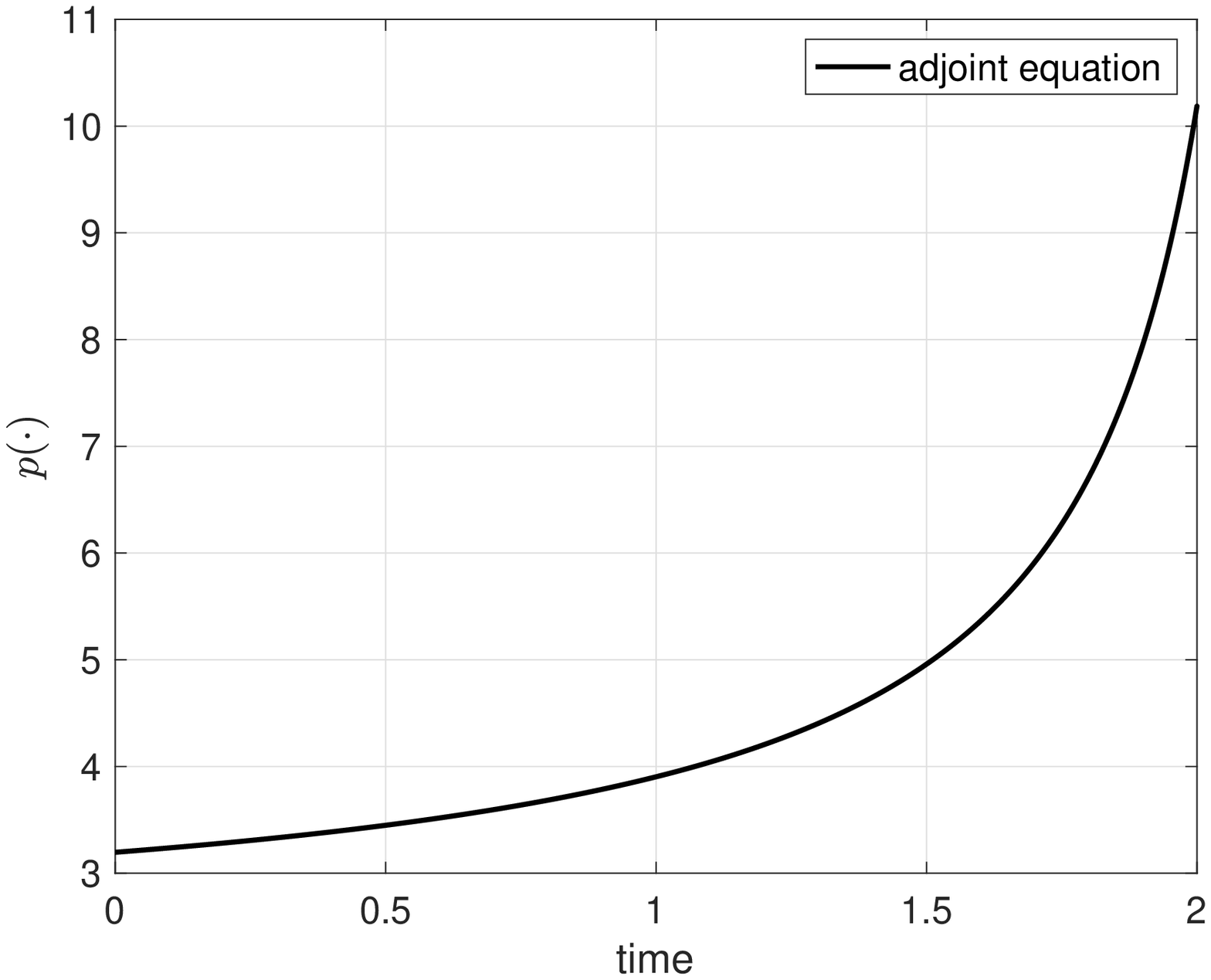}
\includegraphics[scale=0.28]{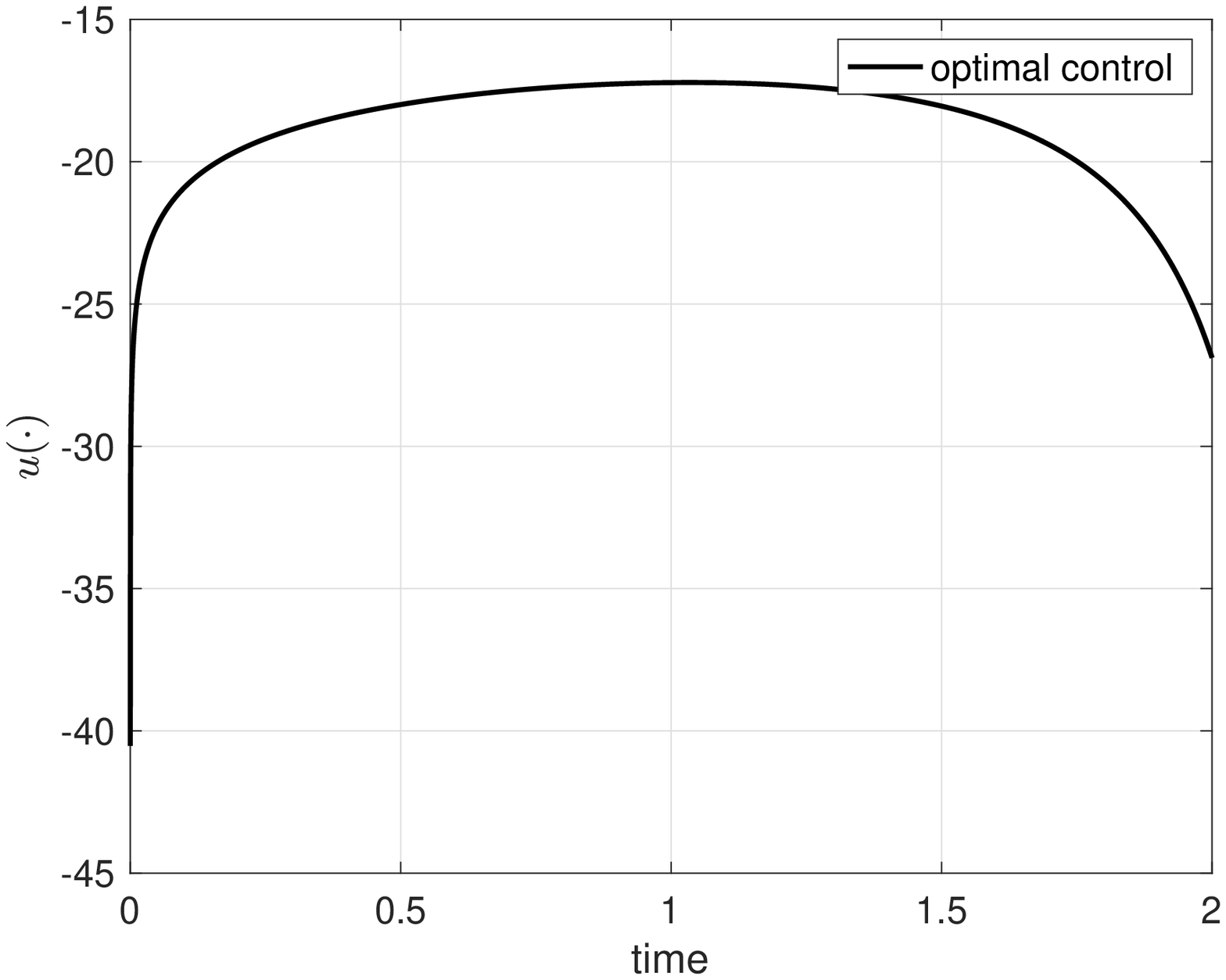}	
\caption{Simulation results of Example \ref{Example_2} with $\alpha = 0.01$.}
\label{Fig_323423423}
\end{figure}
\end{example}

\section{Proof of the Maximum Principle}\label{Section_4}
This section is devoted to prove Theorem \ref{Theorem_3_1}.

\subsection{Preliminaries on Distance Functions}\label{Section_4_1}

Let $(X,\|\cdot\|_{X})$ be a Banach space. We denote $(X^*,\|\cdot\|_{X^*})$ by the dual space of $(X,\|\cdot\|_{X})$, where $X^*$ is the space of bounded linear functionals on $X$ with the norm given by $\|\psi\|_{X^*} := \sup_{x \in X,~ \|x\|_{X} \leq 1} \langle \psi,x \rangle_{X^* \times X}$. Here, $\langle \cdot,\cdot \rangle_{X^* \times X}$ denotes the usual duality paring between $X$ and $X^*$, i.e., $\langle \psi,x \rangle_{X^* \times X} := \psi(x)$. Recall that $(X^*,\|\cdot\|_{X^*})$ is also a Banach space. 

We first deal with the terminal state constraints in (\ref{eq_3}). Recall that $F$ is a nonempty closed convex subsets of $\mathbb{R}^{2n}$. Let $d_{F} : \mathbb{R}^{2n} \rightarrow \mathbb{R}_{+}$ be the standard Euclidean distance function to $F$ defined by $d_{F}(x) := \inf_{y \in F} |x-y|_{\mathbb{R}^{2n}}$ for $x \in \mathbb{R}^{2n}$. Note that $d_{F}(x) = 0$ when $x \in F$. Then it follows from the projection theorem \cite[Theorem 2.10]{Ruszczynski_book} that there is a unique $P_{F}(x) \in F$ with $P_{F}(x) : \mathbb{R}^{2n} \rightarrow F \subset \mathbb{R}^{2n}$, the projection of $x \in \mathbb{R}^{2n}$ onto $F$, such that $d_{F}(x) = \inf_{y \in F} |x-y|_{\mathbb{R}^{2n}} = |x - P_{F}(x)|_{\mathbb{R}^{2n}}$. By \cite[Lemma 2.11]{Ruszczynski_book}, $P_{F}(x) \in F$ is the corresponding projection if and only if $\langle x - P_{F}(x),y- P_{F}(x) \rangle_{\mathbb{R}^{2n} \times \mathbb{R}^{2n}} \leq 0$ for all $y \in F$, which leads to the characterization of $P_{F}(x)$. In view of \cite[Definition 2.37]{Ruszczynski_book}, we have $x - P_{F}(x) \in N_{F}(P_{F}(x))$ for $x \in \mathbb{R}^{2n}$, where $N_{F}(x)$ is the normal cone to the convex set $F$ at a point $x \in \mathbb{R}^{2n}$ defined by
\begin{align}
\label{eq_4_1}
N_{F}(x) := \{ y \in \mathbb{R}^{2n}~|~ \langle y,y^\prime - x \rangle_{\mathbb{R}^{2n} \times \mathbb{R}^{2n}} \leq 0,~ \forall y^\prime \in F \}.	
\end{align}
%In addition, $P_{F}(x) : \mathbb{R}^{2n} \rightarrow F \subset \mathbb{R}^{2n}$ is nonexpansive, i.e., $|P_{F}(x) - P_{F}(x^\prime) |_{\mathbb{R}^{2n}} \leq |x-x^\prime|_{\mathbb{R}^{2n}}$, for all $x,x^\prime \in \mathbb{R}^{2n}$. 
Based on the distance function $d_F$, the terminal state constraint in (\ref{eq_3}) can be written as
\begin{align*}
d_F(\overline{x}_0, \overline{x}(T;\overline{x}_0,\overline{u})) = 0 ~\Leftrightarrow~ 	\begin{bmatrix}
		\overline{x}_0 \\
		\overline{x}(T;\overline{x}_0,\overline{u}))		 		
 	\end{bmatrix} \in F.
\end{align*}

%We state the following lemma
%: (see \cite[Chapter 3.1]{Flett_book} and \cite{Clarke_book})

\begin{lemma}\label{Lemma_4_1}
The function $d_F(x)^2$ is Fr\'echet differentiable on $\mathbb{R}^{2n}$ with the Fr\'echet differentiation of $d_F(x)^2$ at $x$ given by $D d_F(x)^2(h) = 2 \langle x-P_F(x),h \rangle_{\mathbb{R}^{2n} \times \mathbb{R}^{2n}}$ for $h \in \mathbb{R}^{2n}$.
\end{lemma}
\begin{proof}
Note that
\begin{align*}
d_F(x+h)^2 - d_F(x)^2  & \leq |x+h - P_F(x)|_{\mathbb{R}^{2n}}^2 - |x - P_F(x)|_{\mathbb{R}^{2n}}^2  = 2 \langle x - P_F(x),h \rangle + |h|_{\mathbb{R}^{2n}}^2,
\end{align*}
and by the fact that the projection operator is nonexpansive, i.e., $|P_{F}(x) - P_{F}(x^\prime) |_{\mathbb{R}^{2n}} \leq |x-x^\prime|_{\mathbb{R}^{2n}}$, for all $x,x^\prime \in \mathbb{R}^{2n}$ (see \cite[Theorem 2.13]{Ruszczynski_book}), 
\begin{align*}
d_F(x)^2 - d_F(x+h)^2 & \leq|x - P_F(x+h)|_{\mathbb{R}^{2n}}^2 - |x+h - P_F(x+h)|_{\mathbb{R}^{2n}}^2 \\
%& = - 2 \langle x - P_F(x+h), h \rangle + |h|^2 \\
& = - 2 \langle x - P_F(x), h \rangle +  2 \langle P_F(x+h) - P_F(x), h \rangle +  |h|_{\mathbb{R}^{2n}}^2 \\
& \leq - 2 \langle x - P_F(x), h \rangle + 3|h|_{\mathbb{R}^{2n}}^2.
\end{align*}
Since $-3 |h|_{\mathbb{R}^{2n}} \leq \frac{|d_F(x+h)^2 - d_F(x)^2 - 2 \langle x - P_F(x), h \rangle |}{|h|_{\mathbb{R}^{2n}}}	 \leq |h|_{\mathbb{R}^{2n}}$, this completes the proof.
%Then it follows that $-3 |h| \leq \frac{|d_F(x+h) - d_F(x) - 2 \langle x - P_F(x), h \rangle |}{|h|}	 \leq |h|$. We complete the proof.
%\begin{align*}
%-3 |h| \leq \frac{|d_F(x+h) - d_F(x) - 2 \langle x - P_F(x), h \rangle |}{|h|}	 \leq |h|,
%\end{align*}
%and we have the desired result. 
\end{proof}

Now, we consider the inequality state constraint given in (\ref{eq_3}). Let $\gamma: C([0,T];\mathbb{R}^n) \rightarrow C([0,T];\mathbb{R}^m)$ be defined by $\gamma(x(\cdot;x_0,u)) := (\gamma_1(x(\cdot;x_0,u)),\ldots, \gamma_m(x(\cdot;x_0,u))):= G(\cdot,x(\cdot;x_0,u)) = \begin{bmatrix}
 	G^1(\cdot,x(\cdot;x_0,u)) & \cdots & G^m(\cdot,x(\cdot;x_0,u))
 \end{bmatrix}$. Moreover, we let $S \subset C([0,T];\mathbb{R}^m)$ be the nonempty closed convex cone of $C([0,T];\mathbb{R}^m)$ defined by $S := C([0,T];\mathbb{R}_{-}^m)$, where $\mathbb{R}_{-}^m := \mathbb{R}_{-} \times \cdots \times \mathbb{R}_{-}$. Note that $S$ has a nonempty interior. Then the inequality state constraint in (\ref{eq_3}) can be expressed as follows:
\begin{align}
\label{eq_4_2}
\gamma(\overline{x}(\cdot;\overline{x}_0,\overline{u})) \in S~ \Leftrightarrow~ G^i(t,\overline{x}(t;\overline{x}_0,\overline{u})) \leq 0,~\forall t \in [0,T],~ i=1,\ldots,m.
\end{align}

Recall that $G^i$, $i=1,\ldots,m$, are continuously differentiable in $x$. Then $\gamma$ is Fr\'echet differentiable with its Fr\'echet differentiation at $\gamma(x(\cdot)) \in S$ given by $D\gamma(x(\cdot))(w) =  G_x(\cdot,x(\cdot)) w $ for all $w \in C([0,T];\mathbb{R}^n)$ \cite[page 167]{Flett_book}. The normal cone to $S$ at $x \in S$ is defined by
\begin{align}
\label{eq_4_3}
N_S(x) : = \{ \kappa \in C([0,T]; \mathbb{R}^m)^*~|~\langle \kappa,\kappa^\prime - x \rangle_{C_m^* \times C_m} \leq 0,~ \forall \kappa^\prime \in S \},
\end{align}
where $\langle \cdot,\cdot \rangle_{C_m^* \times C_m} := \langle \cdot,\cdot \rangle_{C([0,T]; \mathbb{R}^m)^* \times C([0,T]; \mathbb{R}^m)}$ stands for the duality paring between $C([0,T]; \mathbb{R}^m)$ and $C([0,T]; \mathbb{R}^m)^*$ with $C([0,T]; \mathbb{R}^m)^*$ being the dual space of $C([0,T]; \mathbb{R}^m)$.

\begin{remark}\label{Remark_4_2_3_5_2_3_2_1}
Note that $(C([0,T]; \mathbb{R}^m),\|\cdot\|_{\infty})$ is a separable Banach space \cite[Theorem 6.6, page 140]{Conway_2000_book}. Then by \cite[Theorem 2.18, page 42]{Li_Yong_book}, there exists a norm $\|\cdot \|_{C([0,T]; \mathbb{R}^m)}$ on $C([0,T]; \mathbb{R}^m)$, which is equivalent to $\|\cdot\|_{\infty}$ \cite[Definition 2.17, page 42]{Li_Yong_book}, such that $(C([0,T]; \mathbb{R}^m)^*, \|\cdot \|_{C([0,T]; \mathbb{R}^m)^*})$ is strictly convex, i.e., $\|x \|_{C([0,T]; \mathbb{R}^m)^*} = \|y \|_{C([0,T]; \mathbb{R}^m)^*} = 1$ and $\|x+y \|_{C([0,T]; \mathbb{R}^m)^*} =2 $ imply $x=y$ for $x,y \in C([0,T]; \mathbb{R}^m)^*$ \cite[Definition 2.12, page 41]{Li_Yong_book}.
\end{remark}

Let $d_{S}: C([0,T]; \mathbb{R}^m) \rightarrow \mathbb{R}_{+}$ be the distance function to $S$ defined by
\begin{align*}
%\label{eq_4_4}
d_S(x) := \inf_{y \in S} \|x-y\|_{C([0,T]; \mathbb{R}^m)}~ \textrm{for}~x \in C([0,T]; \mathbb{R}^m).
\end{align*}
By definition of $d_S$, (\ref{eq_4_2}) is equivalent to
\begin{align*}
%\label{eq_4_4_2_2_4_56_2}
d_S \Bigl ( \gamma(\overline{x}(\cdot;\overline{x}_0,\overline{u})) \Bigr ) = 0 ~ \Leftrightarrow ~ \gamma \Bigl (\overline{x}(\cdot;\overline{x}_0,\overline{u}) \Bigr ) \in S.	
\end{align*}

\begin{lemma}\label{Lemma_4_2}
The distance function $d_S$ is nonexpansive,  continuous, and convex. 
\end{lemma}
\begin{proof}
To simplify the notation, let $\|\cdot\| := \|\cdot\|_{C([0,T]; \mathbb{R}^m)}$. We fix $x,y \in S$. Let $\epsilon > 0$ be given. By definition, there is $\pi \in S$ such that $d_S(y) \geq \|\pi-y\| - \epsilon$. We then have
\begin{align*}
d_S(x) \leq \|x-\pi\| \leq \|x-y\| +\|y-\pi\| \leq 	\|x-y\| + d_S(y) + \epsilon.
\end{align*}
Similarly, we have $d_S(x) \geq |\pi-x\| - \epsilon$, and
\begin{align*}	
d_S(y) \leq \|y-\pi\| \leq \|x-y\| + \|x-\pi\| \leq \|x-y\| + d_S(x) + \epsilon.
\end{align*}
Since $\epsilon$ is arbitrary, $|d_S(x) - d_S(y)| \leq \|x-y\|$ holds, which also implies the continuity of $d_S$.

As $S$ is convex, we have $(1-\eta) x + \eta y \in S$ for $\eta \in [0,1]$. By definition of $d_S$, there are $\pi_x, \pi_y \in S$ such that $\|\pi_x - x\| \leq d_S(x) + \epsilon$ and $\|\pi_y - y\| \leq d_S(y) + \epsilon$. 
%\begin{align*}
%\|\pi_x - x\| \leq d_S(x) + \epsilon,~ 	\|\pi_y - y\| \leq d_S(y) + \epsilon.
%\end{align*}
Define $\pi := (1-\eta) \pi_x +  \eta \pi_y \in S$ for $\eta \in [0,1]$. It then follows that
\begin{align*}
d_S ((1-\eta) x + \eta y) & \leq \|\pi - ((1-\eta) x + \eta y) \|  
%\\
%& = \|  ((1-\eta) \pi_x +  \eta \pi_y) - ((1-\eta) x + \eta y) \|  \\
%& \leq (1-\eta) \|\pi_x - x \| + \eta \|\pi_y - y\| \\
\leq (1-\eta) d_S(x) + \eta d_S(y) + \epsilon.
\end{align*}
Since $\epsilon$ is arbitrary, $d_S$ is convex. We complete the proof.
\end{proof}

We define the subdifferential of $d_S$ at $x \in C([0,T];\mathbb{R}^m)$ by \cite[page 214]{Rockafellar_book}
%\cite[page 10]{Clarke_book} (or see 
%
%[page 10, Clarke] [Page 214, Rockafellar]
\begin{align}
\label{eq_4_5_6_5_43_3_5_7_3_3}
\partial d_S(x) := \{ y^\prime \in C([0,T];\mathbb{R}^m)^*~|~\langle y^\prime , y -x \rangle_{C_m^* \times C_m}	\leq d_S(y) - d_S(x),~ \forall y \in C([0,T];\mathbb{R}^m) \}.
\end{align}
By \cite[page 27]{Clarke_book} and Lemma \ref{Lemma_4_2}, since $d_S$ is continuous, $\partial d_S(x)$ is a nonempty ($\partial d_S(x) \neq \emptyset$), convex, and weak--$^*$ compact subset of $C([0,T];\mathbb{R}^m)^*$. Moreover, from \cite[Proposition 2.1.2]{Clarke_book}, it holds that $\|y^\prime\|_{C([0,T];\mathbb{R}^m)^*} \leq 1$ for all $y^\prime \in \partial d_S(x)$. 

An important consequence of Remark \ref{Remark_4_2_3_5_2_3_2_1} and Lemma \ref{Lemma_4_2} is as follows:
\begin{remark}\label{Remark_4_1}
Since $\partial d_S(x)$ is convex, $\eta y^\prime + (1-\eta) y^{\prime \prime} \in \partial d_S(x)$ for any $y^\prime,y^{\prime \prime} \in \partial d_S(x)$ and $\eta \in [0,1]$. Consider $\|\eta y^\prime + (1-\eta) y^{\prime \prime}  \|_{C([0,T]; \mathbb{R}^m)^*} = 1$ for $\eta \in [0,1]$. Then $\|y^\prime\|_{C([0,T];\mathbb{R}^m)^*} = 1$ and $ \|y^{\prime \prime} \|_{C([0,T];\mathbb{R}^m)^*} = 1$ when $\eta = 1$ and $\eta = 0$, respectively. Moreover, when $\eta = \frac{1}{2}$, $\|y^\prime + y^{\prime \prime}  \|_{C([0,T]; \mathbb{R}^m)^*} = 2$. Since $(C([0,T]; \mathbb{R}^m)^*, \|\cdot \|_{C([0,T]; \mathbb{R}^m)^*})$ is strictly convex, we must have $y^\prime = y^{\prime \prime} \in C([0,T];\mathbb{R}^m)^*$, which implies $\partial d_S(x) =\{y^\prime\}$, i.e., $\partial d_S(x)$ is a singleton, and $\|y^\prime\|_{C([0,T];\mathbb{R}^m)^*} = 1$. 
%Below, as $\partial d_S(x)$ is singleton, we use the notation $\partial d_S(x) \in C([0,T];\mathbb{R}^m)^*$ to denote the subdifferential of $d_S(x)$.
\end{remark}

\begin{lemma}\label{Lemma_4_3_1_2_2_1}
The distance function $d_S$ is strictly Hadamard differentiable on $C([0,T];\mathbb{R}^m) \setminus S$ with the Hadamard differential $D d_S$ satisfying $\|D d_S(x)\|_{C([0,T]; \mathbb{R}^m)^*} = \|y^\prime\|_{C([0,T]; \mathbb{R}^m)^*} = 1$ for all $x \in C([0,T];\mathbb{R}^m) \setminus S$. Consequently, $d_S(x)^2$ is strictly Hadamard differentiable on $C([0,T];\mathbb{R}^m) \setminus S$ with the Hadamard differential given by $D d_S(x)^2 = 2 d_S(x) D d_S(x)$ for $x \in C([0,T];\mathbb{R}^m) \setminus S$. Moreover, $d_S(x)^2$ is Fr\'echet differentiable on $S$ with the Fr\'echet differential being $D d_S(x)^2 = 0 \in C([0,T]; \mathbb{R}^m)^*$ for all $x \in S$.
\end{lemma}

\begin{proof}
The strictly Hadamard differentiability of $d_S(x)$ and $d_S(x)^2$ on $C([0,T];\mathbb{R}^m) \setminus S$ follows from on Lemma \ref{Lemma_4_2} and Remark \ref{Remark_4_1}, together with \cite[Theorem 3.54]{Mordukhovich_book}. The Fr\'echet differentiability of $d_S(x)^2$ on $S$ with $D d_S(x) = 0 \in C([0,T]; \mathbb{R}^m)^*$ for $x \in S$ follows by the fact that $d_S(x) = 0$ for $x \in S$ and that $d_S$ is nonexpansive shown in Lemma \ref{Lemma_4_2}. This completes the proof.
\end{proof}

\subsection{Ekeland Variational Principle}\label{Section_4_2}

Recall that the pair $(\overline{x}(\cdot),\overline{u}(\cdot)) \in C[0,T];\mathbb{R}^n) \times \mathcal{U}^p[0,T]$ is the optimal pair of \textbf{(P)}. We also write $\overline{x}(\cdot;\overline{x}_0,\overline{u}) := \overline{x}(\cdot)$ to emphasize the dependence of the state equation $\overline{x}(\cdot)$ on the optimal initial condition and control $(\overline{x}_0, \overline{u}(\cdot)) \in \mathbb{R}^n \times \mathcal{U}^p[0,T]$. Note that the pair $(\overline{x}_0,\overline{x}(\cdot;\overline{x}_0,\overline{u}))$ holds the state constraints in (\ref{eq_3}). The optimal cost of \textbf{(P)} under $(\overline{x}(\cdot),\overline{u}(\cdot))$ can be written by $J(\overline{x}_0, \overline{u}(\cdot))$.

Recall the distance functions $d_F$ and $d_S$ in Section \ref{Section_4_1}. For $\epsilon > 0$, we define the penalized objective functional as follows:
\begin{align}
\label{eq_4_5}
J_{\epsilon}(x_0,u(\cdot)) & = \Biggl ( \Bigl ( \bigl [ J(x_0,u(\cdot)) - J(\overline{x}_0, \overline{u}(\cdot)) + \epsilon \bigr ]^+ \Big)^2 + d_F \Bigl ( \begin{bmatrix}
x_0 \\
x(T) 	
 \end{bmatrix} \Bigr )^2 + d_S \Bigl ( \gamma(x(\cdot)) \Bigr )^2 \Biggr )^{\frac{1}{2}}.
\end{align}
We can easily observe that $J_{\epsilon}(\overline{x}_0, \overline{u}(\cdot)) = \epsilon > 0$, i.e., $(\overline{x}_0, \overline{u}(\cdot)) \in \mathbb{R}^n \times \mathcal{U}^p[0,T]$ is the $\epsilon$-optimal solution of (\ref{eq_4_5}). Define the Ekeland metric $\widehat{d}:(\mathbb{R}^n \times \mathcal{U}^p[0,T]) \times (\mathbb{R}^n \times \mathcal{U}^p[0,T]) \rightarrow \mathbb{R}_{+}$ as follows:
\begin{align}	
\label{eq_4_6}
\widehat{d} \Bigl ((x_0,u(\cdot)),(\tilde{x}_0,\tilde{u}(\cdot)) \Bigr ) := |x_0 - \tilde{x}_0 | + \overline{d}(u(\cdot),\tilde{u}(\cdot)),
\end{align}
where
\begin{align}
\label{eq_4_7}
\overline{d}(u(\cdot),\tilde{u}(\cdot)) := |\{t \in [0,T]~|~ u(t) \neq \tilde{u}(t) \} |,~ \forall u(\cdot),\tilde{u}(\cdot) \in \mathcal{U}^p[0,T].
\end{align}
It is easy to see that $(\mathbb{R}^n \times \mathcal{U}^p[0,T],\widehat{d})$ is a complete metric space \cite[Lemma 7.2]{Ekeland_JMAA_1974}. By Assumption \ref{Assumption_2_2}, together with Lemmas \ref{Lemma_4_1} and \ref{Lemma_4_2}, $J_{\epsilon}(x_0,u)$ in (\ref{eq_4_5}) is a continuous functional on $(\mathbb{R}^n \times \mathcal{U}^p[0,T],\widehat{d})$.

In view of (\ref{eq_4_5})-(\ref{eq_4_7}), we have
\begin{align}
\label{eq_4_8}
\begin{cases}
J_{\epsilon}(x_0,u(\cdot)) > 0,~ \forall (x_0,u(\cdot)) \in \mathbb{R}^n \times \mathcal{U}^p[0,T], \\
J_{\epsilon}(\overline{x}_0, \overline{u}(\cdot))  = \epsilon \leq \inf_{(x_0,u(\cdot)) \in \mathbb{R}^n \times \mathcal{U}^p[0,T]} J_{\epsilon}(x_0,u(\cdot)) + \epsilon.
\end{cases}
\end{align}
By the Ekeland variational principle \cite{Ekeland_JMAA_1974}, there exists a pair $(x_0^\epsilon,u^\epsilon) \in \mathbb{R}^n \times \mathcal{U}^p[0,T]$ such that
\begin{align}
\label{eq_4_9_1_1_1}
\widehat{d}\Bigl ( (x_0^\epsilon,u^\epsilon(\cdot)),(\overline{x}_0,\overline{u}(\cdot)) \Bigr ) \leq \sqrt{\epsilon},		
\end{align}
and
\begin{align}
\label{eq_4_9}
\begin{cases}
J_{\epsilon}(x_0^\epsilon,u^\epsilon(\cdot)) \leq J_{\epsilon}(\overline{x}_0, \overline{u}(\cdot))  = \epsilon, \\
J_{\epsilon}(x_0^\epsilon,u^\epsilon(\cdot)) \leq J_{\epsilon}(x_0,u(\cdot)) + \sqrt{\epsilon} \widehat{d} \Bigl ((x_0^\epsilon,u^\epsilon(\cdot)),(x_0,u(\cdot)) \Bigr ),~ \forall (x_0,u(\cdot)) \in \mathbb{R}^n \times \mathcal{U}^p[0,T].
\end{cases}
\end{align}
As $\widehat{d} ((x_0^\epsilon,u^\epsilon(\cdot)),(x_0^\epsilon,u^\epsilon(\cdot))) = 0$, the above condition implies that the pair $(x_0^\epsilon,u^\epsilon(\cdot)) \in \mathbb{R}^n \times \mathcal{U}^p[0,T]$ is the minimizing solution of the following Ekeland objective functional over $\mathbb{R}^n \times \mathcal{U}^p[0,T]$:
\begin{align}
\label{eq_4_10}
J_{\epsilon}(x_0,u(\cdot)) + 	\sqrt{\epsilon} \widehat{d} \Bigl ((x_0^\epsilon,u^\epsilon(\cdot)),(x_0,u(\cdot)) \Bigr ).
\end{align}
We observe that (\ref{eq_4_10}) is the unconstrained control problem. By notation, we write $(x^\epsilon(\cdot),u^\epsilon(\cdot)) := (x^{\epsilon}(\cdot;x_0^\epsilon,u^\epsilon),u^\epsilon(\cdot))  \in C([0,T];\mathbb{R}^n) \times \mathcal{U}^p[0,T]$, where $x^{\epsilon}(\cdot;x_0^\epsilon,u^\epsilon)$ is the state trajectory of (\ref{eq_1}) under $(x_0^\epsilon,u^\epsilon(\cdot)) \in \mathbb{R}^n \times \mathcal{U}^p[0,T]$.

\subsection{Spike Variations and First Variational Equation}\label{Section_4_3}

In the previous subsection, we have obtained the $\epsilon$-optimal solution to \textbf{(P}), which is also the optimal solution to the Ekeland objective functional in (\ref{eq_4_10}). The next step is to derive the necessary condition for $(x_0^\epsilon,u^\epsilon(\cdot)) \in \mathbb{R}^n \times \mathcal{U}^p[0,T]$. We employ the spike variation technique, as $U$ does not have any algebraic structure (hence, it is impossible to use standard (convex) variations).

For $\delta \in (0,1)$, define
\begin{align*}
\mathcal{E}_{\delta} := \{ E \in [0,T]~|~ |E| = \delta T\}	,
\end{align*}
where $|E|$ denotes the Lebesgue measure of $E$. For $E_{\delta} \in \mathcal{E}_{\delta}$, we introduce the spike variation associated with $u^\epsilon$, i.e., the optimal solution of (\ref{eq_4_10}): 
\begin{align*}
u^{\epsilon,\delta}(s) := \begin{cases}
 	u^\epsilon(s), & s \in [0,T] \setminus E_{\delta}, \\
	u(s), & s \in  E_{\delta},
 \end{cases}	
\end{align*}
where $u(\cdot) \in \mathcal{U}^p[0,T]$. Clearly, $u^{\epsilon,\delta}(\cdot) \in \mathcal{U}^p[0,T]$. Moreover, by definition of $\overline{d}$ in (\ref{eq_4_7}),
\begin{align}
\label{eq_4_12_345234231}
	\overline{d}(u^{\epsilon,\delta}(\cdot),u^\epsilon(\cdot)) \leq | E_{\delta}| = \delta T.
\end{align}
Consider also the variation of the initial state given by $x_0 + \delta a $, where $a \in \mathbb{R}^n$. By notation, let us define the perturbed state equation by
\begin{align}
\label{eq_5_234235457634545245234523}
x^{\epsilon,\delta}(\cdot) := x^{\epsilon,\delta}(\cdot; x_0^\epsilon + \delta a,u^{\epsilon,\delta}) \in C([0,T];\mathbb{R}^n).
\end{align}
In fact, $x^{\epsilon,\delta}(\cdot)$ is the state trajectory of (\ref{eq_1}) under $(x_0^\epsilon + \delta a, u^{\epsilon,\delta}(\cdot)) \in \mathbb{R}^n \times \mathcal{U}^p[0,T]$. We also recall $(x^\epsilon(\cdot),u^\epsilon(\cdot)) := (x^{\epsilon}(\cdot;x_0^\epsilon,u^\epsilon),u^\epsilon(\cdot)) \in C([0,T];\mathbb{R}^n) \times \mathcal{U}^p[0,T]$, where $x^{\epsilon}$ is the state trajectory of (\ref{eq_1}) under $(x_0^\epsilon,u^\epsilon(\cdot)) \in \mathbb{R}^n \times \mathcal{U}^p[0,T]$. Then by (\ref{eq_4_9}) and (\ref{eq_4_12_345234231}), we have
\begin{align}
\label{eq_4_13_1_1_1_2}
- \sqrt{\epsilon} (|a| +  T) \leq \frac{1}{\delta} \Bigl ( J_{\epsilon}(x_0^\epsilon + \delta a, u^{\epsilon,\delta}(\cdot)) - 	J_{\epsilon}(x_0^\epsilon, u^{\epsilon}(\cdot))  \Bigr ).
\end{align}

\begin{lemma}\label{Lemma_4_4_2342341234234}
The following result holds:
\begin{align*}
\sup_{t \in [0,T]} \Bigl | x^{\epsilon,\delta}(t) - x^{\epsilon}(t) - \delta  Z^{\epsilon}(t) \Bigr | = o(\delta),
\end{align*}
where $Z^\epsilon$ is the solution to the first variational equation related to the optimal pair $(x_0^\epsilon, u^{\epsilon}(\cdot)) \in \mathbb{R}^n \times \mathcal{U}^p[0,T]$ given by
\begin{align*}
%\label{eq_4_11}
Z^{\epsilon}(t) & = a + \int_0^t \Bigl [ \frac{f_x(t,s,x^\epsilon(s),u^\epsilon(s))}{(t-s)^{1-\alpha}} Z^{\epsilon}(s) \dd s 	 + \frac{\widehat{f}(t,s) }{(t-s)^{1-\alpha}}  \Bigr ] \dd s \\
&~~~ + \int_0^t \Bigl [ g_x(t,s,x^\epsilon(s),u^\epsilon(s)) Z^{\epsilon}(s) + \widehat{g}(t,s) \Bigr ] \dd s,~  \textrm{a.e.}~ t \in [0,T], \nonumber
\end{align*}
with for any $u(\cdot) \in \mathcal{U}^p[0,T]$, 
\begin{align*}
%\label{eq_4_11_1_1_1}
\begin{cases}
	\widehat{f}(t,s) := f(t,s,x^\epsilon(s),u(s)) - f(t,s,x^\epsilon(s),u^\epsilon(s)), \\
	\widehat{g}(t,s) := g(t,s,x^\epsilon(s),u(s)) - g(t,s,x^\epsilon(s),u^\epsilon(s)). 
	\\
%	\widehat{l}(s) := l(s,x^\epsilon(s),u(s)) - l(s,x^\epsilon(s),u^\epsilon(s)).
\end{cases}	
\end{align*}
\end{lemma}

\begin{proof}
By definition and (\ref{eq_5_234235457634545245234523}), 
\begin{align*}
	x^{\epsilon,\delta}(t) = x(t;x_0^\epsilon + \delta a,u^{\epsilon,\delta}) &= (x_0^{\epsilon} + \delta a) + \int_0^t \frac{f(t,s,x^{\epsilon,\delta}(s),u^{\epsilon,\delta}(s))}{(t-s)^{1-\alpha}} \dd s + \int_0^t g(t,s,x^{\epsilon,\delta}(s),u^{\epsilon,\delta}(s)) \dd s, \\
	x^{\epsilon}(t)  = x(t;x_0^\epsilon,u^{\epsilon})  & = x_0^{\epsilon} + \int_0^t \frac{f(t,s,x^{\epsilon}(s),u^{\epsilon}(s))}{(t-s)^{1-\alpha}} \dd s + \int_0^t g(t,s,x^{\epsilon}(s),u^{\epsilon}(s)) \dd s.
\end{align*}
For $\delta \in (0,1)$, let
\begin{align}
\label{eq_e_1}
Z^{\epsilon,\delta}(t) := \frac{x^{\epsilon,\delta}(t) - x^{\epsilon}(t)}{\delta},	~ t \in [0,T],
\end{align}
where based on the Taylor expansion, $Z^{\epsilon,\delta}$ holds
\begin{align*}
Z^{\epsilon,\delta}(t) & = a + \int_0^t  \Bigl [ \frac{f_{x}^{\epsilon,\delta}(t,s)}{(t-s)^{1-\alpha}} Z^{\epsilon,\delta}(s) + \frac{\mathds{1}_{E_{\delta}}(s)}{\delta} \frac{\widehat{f}(t,s) }{(t-s)^{1-\alpha}}  \Bigr ] \dd s \\
&~~~ + \int_0^t  \Bigl [ g_{x}^{\epsilon,\delta}(t,s) Z^{\epsilon,\delta}(s) + \frac{\mathds{1}_{E_{\delta}}(s)}{\delta} \widehat{g}(t,s) \Bigr ] \dd s,~ t \in [0,T]
\end{align*}
with $f_{x}^{\epsilon,\delta}$ and $g_{x}^{\epsilon,\delta}$ defined by
\begin{align*}
f_{x}^{\epsilon,\delta}(t,s) & := \int_0^1 f_x(t,s,x^{\epsilon}(s) + r (x^{\epsilon,\delta}(s) - x^{\epsilon}(s)),u^{\epsilon,\delta}(s)) \dd r, \\
g_{x}^{\epsilon,\delta}(t,s) & := \int_0^1 g_x(t,s,x^{\epsilon}(s) + r (x^{\epsilon,\delta}(s) - x^{\epsilon}(s)),u^{\epsilon,\delta}(s)) \dd r.
\end{align*}
%Then it is equivalent to prove that $\lim_{\delta \downarrow 0} |Z^{\epsilon,\delta}(t)	- Z^{\epsilon}(t)|_{\mathbb{R}^n} = 0$ for all $t \in [0,T]$.

By Assumptions \ref{Assumption_2_1} and \ref{Assumption_2_2}, we have
\begin{align}
\label{eq_4_16_4534345345}
\begin{cases}
|\widehat{f}(t,s)| + |\widehat{g}(t,s)|  \leq 4 K_0(s) + 4 K(s)|x^\epsilon(s)| + K(s)(\rho(u(s),u_0) + \rho(u^\epsilon(s),u_0)) =: \widetilde{\psi}(s), \\
|f_x^{\epsilon,\delta}(t,s)| + |g_x^{\epsilon,\delta}(s)|  \leq K(s).
\end{cases}
\end{align}
Let $q > \frac{1}{\alpha}$,
%$q \geq \frac{1}{\alpha}$, 
and we replace $p$ by $q$ in Lemma \ref{Lemma_A_5}. Recall $L^{l+}([0,T];\mathbb{R}^n) := \cup_{r > l} L^{r}([0,T];\mathbb{R}^n)$ for $1 \leq l < \infty$, and the $L^p$-spaces of this paper are induced by the finite measure on $([0,T],\mathcal{B}([0,T]))$. Since $p > \alpha p > 1$, it holds that $K(\cdot) \in L^{\frac{p}{\alpha p - 1} + } ([0,T];\mathbb{R}) \subset   L^{  \frac{p}{p-1} +}([0,T];\mathbb{R})$. Hence, we can choose $q$ so that $K(\cdot) \in L^q([0,T];\mathbb{R}) \subset L^{\frac{p}{p-1}}([0,T];\mathbb{R}) $ and $x^{\epsilon}(\cdot) \in L^p([0,T];\mathbb{R}^n) \subset L^{\frac{pq}{q-p}}([0,T];\mathbb{R}^n)$. It then follows from the H\"older's inequality that
\begin{align*}
\Bigl ( \int_0^T |K(s)^p x^{\epsilon}(s)^p| \dd s \Bigr)^{\frac{1}{p}}  \leq \Bigl ( \int_0^T |K(s)|^q \dd s \Bigr )^{\frac{1}{q}} \Bigl ( \int_0^T |x^{\epsilon}(s)|^{\frac{pq}{q-p}} \dd s \Bigr )^{\frac{q-p}{qp}} & < \infty, \\
\Bigl ( \int_0^T |K(s)^p (\rho(u(s),u_0) + \rho(u^\epsilon(s),u_0))^p| \dd s \Bigr)^{\frac{1}{p}} &< \infty.
%\leq \Bigl ( \int_0^T |K(s)|^q \dd s \Bigr )^{\frac{1}{q}} \Bigl ( \int_0^T |(|u(s)| + |u^\epsilon(s)|)|^{\frac{pq}{q-p}} \dd s \Bigr )^{\frac{q-p}{qp}} < \infty.
\end{align*}
Therefore, as $K_0(\cdot) \in L^{\frac{1}{\alpha} + } ([0,T];\mathbb{R}) \subset L^{p+} ([0,T];\mathbb{R})$, $\widetilde{\psi}(\cdot) \in L^p([0,T];\mathbb{R}) \subset L^{\frac{q}{q-1}}([0,T];\mathbb{R})$. 

Based on Assumption \ref{Assumption_2_1} and (\ref{eq_4_16_4534345345}), we can show that
\begin{align}	
\label{eq_e_2}
|x^{\epsilon,\delta}(t) - x^{\epsilon}(t) | 
%& = |\delta a  | + \int_0^t \frac{K(s)}{(t-s)^{1-\alpha}} |x^{\epsilon,\delta}(t) - x^{\epsilon}(t) |\dd s + \int_0^t \mathds{1}_{E_{\delta}}(s) \frac{\widetilde{\psi}(s)} {(t-s)^{1-\alpha}} \dd s \\
%&~~~ + \int_0^t K(s) |x^{\epsilon,\delta}(t) - x^{\epsilon}(t) |\dd s + \int_0^t \mathds{1}_{E_{\delta}}(s) \widetilde{\psi}(s) \dd s   \\
& \leq  b(t) + \int_0^t \frac{K(s)}{(t-s)^{1-\alpha}} |x^{\epsilon,\delta}(t) - x^{\epsilon}(t) |\dd s + \int_0^t K(s) |x^{\epsilon,\delta}(t) - x^{\epsilon}(t) |\dd s, 
\end{align}
where
\begin{align*}
b(t) = |\delta a  |_{\mathbb{R}^n} + \int_0^t \mathds{1}_{E_{\delta}}(s) \frac{\widetilde{\psi}(s)} {(t-s)^{1-\alpha}} \dd s + \int_0^t \mathds{1}_{E_{\delta}}(s) \widetilde{\psi}(s) \dd s.   
\end{align*}
We let $\widetilde{\psi}(t,\cdot) := \widetilde{\psi}(\cdot)$ in (\ref{eq_4_16_4534345345}). As $\widetilde{\psi}(0,\cdot) \in L^p([0,T];\mathbb{R}) \subset L^{\frac{q}{q-1}}([0,T];\mathbb{R})$, by Lemmas \ref{Lemma_A_3} and \ref{Lemma_A_4} (and using Assumption \ref{Assumption_2_1}), we have $b(\cdot) \in L^p([0,T];\mathbb{R}^n) \subset L^{\frac{q}{q-1}}([0,T];\mathbb{R}^n)$. Note also that we can choose $q$ so that $x^{\epsilon}(\cdot) \in L^p([0,T];\mathbb{R}^n) \subset L^{\frac{q}{q-1}}([0,T];\mathbb{R}^n)$ and $|x^{\epsilon,\delta}(\cdot) - x^{\epsilon}(\cdot) |_{\mathbb{R}^n} \in L^p([0,T];\mathbb{R}) \subset L^{\frac{q}{q-1}}([0,T];\mathbb{R})$. 
%In fact, by Lemma \ref{Lemma_2_1} (see also Lemmas \ref{Lemma_B_1} and \ref{Lemma_B_2} in Appendix \ref{Appendix_B}), $|x^{\epsilon,\delta}(\cdot) - x^{\epsilon}(\cdot) |_{\mathbb{R}^n} \in M([0,T];\mathbb{R})$.
In addition, from Lemmas \ref{Lemma_A_3} and \ref{Lemma_A_4}, there is a constant $C \geq 0$ such that
\begin{align*}	
\Biggl | \int_0^t \mathds{1}_{E_{\delta}}(s) \widetilde{\psi}(s) \dd s \Biggr | + \Biggl | \int_0^t \mathds{1}_{E_{\delta}}(s) \frac{\widetilde{\psi}(s)} {(t-s)^{1-\alpha}} \dd s \Biggr | \leq C |E_{\delta}|^{\frac{1}{q}}.
\end{align*}
Then applying Lemma \ref{Lemma_A_5} to (\ref{eq_e_2}) yields
\begin{align}
\label{eq_5_4564563452342342342342342}
|x^{\epsilon,\delta}(t) - x^{\epsilon}(t) |_{\mathbb{R}^n} & \leq  b(t) + C \int_0^t \frac{K(s)}{(t-s)^{1-\alpha}} b(s) \dd s + C \int_0^t K(s) b(s) \dd s \\
& \leq C \Bigl ( |\delta a|_{\mathbb{R}^n} + |E_{\delta}|^{\frac{1}{q}} \Bigr )~ \rightarrow 0,~ \textrm{as $\delta \downarrow 0$ for all $ t\in [0,T]$.} \nonumber
\end{align}

On the other hand, since $Z^{\epsilon}$ is linear, by Lemma \ref{Lemma_2_1} (see also the results in Appendix \ref{Appendix_B}), it admits a unique solution in $C([0,T];\mathbb{R}^n)$. Hence, as
\begin{align*}
|Z^{\epsilon}(t)| \leq |a| + 	\int_0^t \frac {K(s)}{(t-s)^{1-\alpha}}  |Z^{\epsilon}(s)| \dd s + \int_0^t \frac{\widetilde{\psi}(s)}{(t-s)^{1-\alpha}} \dd s + \int_0^t K(s) |Z^{\epsilon}(s)| \dd s + \int_0^t \widetilde{\psi}(s) \dd s,
\end{align*}
we use Lemmas \ref{Lemma_A_3}-\ref{Lemma_A_5} to get
\begin{align}
\label{eq_e_3}
|Z^{\epsilon}(t)| & \leq  \hat{b}(t) + C \int_0^t	\frac{K(s) }{(t-s)^{1-\alpha}}  \hat{b}(s) \dd s + C \int_0^t K(s) \hat{b}(s) \dd s  \leq C \Bigl ( |a|_{\mathbb{R}^n} + \|\widetilde{\psi}(\cdot)\|_{L^p([0,T];\mathbb{R})} \Bigr ),
\end{align}
where $\hat{b}(t) := 	|a|_{\mathbb{R}^n} + \int_0^t \frac{\widetilde{\psi}(s)}{(t-s)^{1-\alpha}} \dd s  + \int_0^t \widetilde{\psi}(s) \dd s$.

We obtain
\begin{align}
\label{eq_4_17_2343234112}
Z^{\epsilon,\delta}(t)	- Z^{\epsilon}(t) & = \int_0^t \frac{f_x^{\epsilon,\delta}(t,s)}{(t-s)^{1-\alpha}} \Bigl [ Z^{\epsilon,\delta}(s) - Z^{\epsilon}(s) \Bigr ] \dd s + \int_0^t \Bigl ( \frac{\mathds{1}_{E_{\delta}}(s)}{\delta} - 1 \Bigr ) \frac{\widehat{f}(t,s)}{(t-s)^{1-\alpha}} \dd s \\
&~~~ + \int_0^t \frac{f_x^{\epsilon,\delta}(t,s) - f_x(t,s,x^{\epsilon}(s),u^{\epsilon}(s))}{(t-s)^{1-\alpha}} Z^{\epsilon}(s) \dd s   \nonumber\\
&~~~ + \int_0^t g_x^{\epsilon,\delta}(t,s) \Bigl [ Z^{\epsilon,\delta}(s) - Z^{\epsilon}(s) \Bigr ] \dd s + \int_0^t \Bigl ( \frac{\mathds{1}_{E_{\delta}}(s)}{\delta} - 1 \Bigr ) \widehat{g}(t,s) \dd s \nonumber\\
&~~~ + \int_0^t \Bigl [ g_x^{\epsilon,\delta}(t,s) - g_x(t,s,x^{\epsilon}(s),u^{\epsilon}(s)) \Bigr ] Z^{\epsilon}(s) \dd s,~ t \in [0,T]. \nonumber
\end{align}
Notice that
\begin{align*}
\Biggl | \frac{f_x^{\epsilon,\delta}(t,s) - f_x(t,s,x^{\epsilon}(s),u^{\epsilon}(s))}{(t-s)^{1-\alpha}} Z^{\epsilon}(s) \Biggr |_{\mathbb{R}^n} & \leq \frac{4 K(s)}{(t-s)^{1-\alpha}} |Z^{\epsilon}(s)|,~ \forall s \in [0,t), \\
\Biggl | \Bigl [ g_x^{\epsilon,\delta}(t,s) - g_x(t,s,x^{\epsilon}(s),u^{\epsilon}(s)) \Bigr ] Z^{\epsilon}(s) \Biggr |_{\mathbb{R}^n} & \leq 4 K(s) |Z^{\epsilon}(s)|,~ \forall s \in [0,t].
\end{align*}
where $\lim_{\delta \downarrow 0} |f_x^{\epsilon,\delta}(t,s) - f_x(t,s,x^{\epsilon}(s),u^{\epsilon}(s))| = 0$ and $\lim_{\delta \downarrow 0} |g_x^{\epsilon,\delta}(t,s) - g_x(t,s,x^{\epsilon}(s),u^{\epsilon}(s))| = 0$. In addition, using (\ref{eq_e_3}), we get
% and (\ref{eq_5_4564563452342342342342342}), we get
\begin{align*}
\int_0^T	\frac{4 K(s)}{(t-s)^{1-\alpha}} |Z^{\epsilon}(s)| \dd s < \infty,~ \int_0^T 4 K(s) |Z^{\epsilon}(s)| \dd s < \infty.
\end{align*}

For convenience, define
\begin{align*}	
b^{(1,1)}(t) & := \int_0^t \frac{f_x^{\epsilon,\delta}(t,s) - f_x(t,s,x^{\epsilon}(s),u^{\epsilon}(s))}{(t-s)^{1-\alpha}} Z^{\epsilon}(s) \dd s \\
b^{(2,1)}(t) & := \int_0^t \Bigl [ g_x^{\epsilon,\delta}(t,s) - g_x(t,s,x^{\epsilon}(s),u^{\epsilon}(s)) \Bigr ] Z^{\epsilon}(s) \dd s \\
b^{(1,2)}(t) & := \int_0^t \Bigl ( \frac{\mathds{1}_{E_{\delta}}(s)}{\delta} - 1 \Bigr ) \frac{\widehat{f}(t,s)}{(t-s)^{1-\alpha}} \dd s \\
b^{(2,2)}(t) & := \int_0^t \Bigl ( \frac{\mathds{1}_{E_{\delta}}(s)}{\delta} - 1 \Bigr ) \widehat{g}(t,s)  \dd s.
\end{align*}
By the dominated convergence theorem, it follows that
\begin{align*}	
\lim_{\delta \downarrow 0} b^{(1,1)}(t)  = 0, ~ \lim_{\delta \downarrow 0} b^{(2,1)}(t) = 0,~ \forall t \in [0,T].
\end{align*}
%Hence, by the dominated convergence theorem, it follows that
%\begin{align*}	
%\lim_{\delta \downarrow 0} b^{(1,1)}(t) & := \lim_{\delta \downarrow 0} \int_0^t \frac{f_x^{\epsilon,\delta}(t,s) - f_x(t,s,x^{\epsilon}(s),u^{\epsilon}(s))}{(t-s)^{1-\alpha}} Z^{\epsilon}(s) \dd s = 0,~ \forall t \in [0,T], \\
%\lim_{\delta \downarrow 0} b^{(2,1)}(t) & := \lim_{\delta \downarrow 0} \int_0^t \Bigl [ g_x^{\epsilon,\delta}(t,s) - g_x(t,s,x^{\epsilon}(s),u^{\epsilon}(s)) \Bigr ] Z^{\epsilon}(s) \dd s  = 0,~ \forall t \in [0,T].
%\end{align*}
In addition, by letting
\begin{align*}
\phi(t,s) = \widehat{g}(t,s),~~~~ \psi(t,s) = \widehat{f}(t,s),
\end{align*} 
and then invoking Lemmas \ref{Lemma_D_1} and \ref{Lemma_D_2} in Appendix \ref{Appendix_D} (by Assumptions \ref{Assumption_2_1} and \ref{Assumption_2_2}, together with Remark \ref{Remark_B_1}, $\psi$ holds (\ref{eq_d_1}) in Appendix \ref{Appendix_D}), for any $\delta \in (0,1)$, there exists an $E_{\delta} \in \mathcal{E}_{\delta}$ such that
\begin{align*}
|b^{(1,2)}(t)|_{\mathbb{R}^n}  \leq \delta,~ |b^{(2,2)}(t)|_{\mathbb{R}^n} &  \leq \delta,~ \forall t \in [0,T].
\end{align*}

%\begin{align*}
%|b^{(1,2)}(t)|_{\mathbb{R}^n} &  := \Bigg | \int_0^t \Bigl ( \frac{\mathds{1}_{E_{\delta}}(s)}{\delta} - 1 \Bigr ) \frac{\widehat{f}(t,s)}{(t-s)^{1-\alpha}} \dd s	 \Biggr |_{\mathbb{R}^n}  \leq \delta,~ \forall t \in [0,T], \\
%|b^{(2,2)}(t)|_{\mathbb{R}^n} &  := \Bigg | \int_0^t \Bigl ( \frac{\mathds{1}_{E_{\delta}}(s)}{\delta} - 1 \Bigr ) \widehat{g}(t,s)  \dd s	 \Biggr |_{\mathbb{R}^n} \leq \delta,~ \forall t \in [0,T].
%\end{align*}
With $b^{(1)}(\cdot) := b^{(1,1)}(\cdot) + b^{(2,1)}(\cdot)$ and $b^{(2)}(\cdot) := b^{(1,2)}(\cdot) + b^{(2,2)}(\cdot)$ in (\ref{eq_4_17_2343234112}), we then have
\begin{align*}
|Z^{\epsilon,\delta}(t)	- Z^{\epsilon}(t)| & \leq b^{(1)}(t) + b^{(2)}(t)  + \int_0^t \frac{K(s)}{(t-s)^{1-\alpha}} \Bigl [ Z^{\epsilon,\delta}(s) - Z^{\epsilon}(s) \Bigr ] \dd s \\
&~~~ +  \int_0^t K(s) \Bigl [ Z^{\epsilon,\delta}(s) - Z^{\epsilon}(s) \Bigr ] \dd s,~ t \in [0,T],
\end{align*}
and by applying the same technique as above and using Lemma \ref{Lemma_A_5}, 
\begin{align*}
|Z^{\epsilon,\delta}(t)	- Z^{\epsilon}(t)| & \leq  b^{(1)}(t) + b^{(2)}(t) + C \int_0^t \frac{K(s)}{(t-s)^{1-\alpha}} \Bigl [ b^{(1)}(s) + b^{(2)}(s) \Bigr ] \dd s \\
&~~~ + C \int_0^t K(s) \Bigl [ b^{(1)}(s) + b^{(2)}(s) \Bigr ] \dd s,~ t \in [0,T].
\end{align*}
Hence, the dominated convergence theorem implies that
\begin{align*}	
\lim_{\delta \downarrow 0} |Z^{\epsilon,\delta}(t)	- Z^{\epsilon}(t)|_{\mathbb{R}^n} = 0,~ \forall t \in [0,T].
\end{align*}
By definition of $Z^{\epsilon,\delta}$ in (\ref{eq_e_1}), we have the desired result. This completes the proof.	
\end{proof}

\subsection{Crucial Facts from Ekeland Variational Principle, together with Passing Limit and Second Variational Equation}\label{Section_4_4}

We recall $Z^{\epsilon,\delta}(\cdot) := \frac{x^{\epsilon,\delta}(\cdot) - x^{\epsilon}(\cdot)}{\delta}$ defined in (\ref{eq_e_1}). Based on the Taylor expansion, 
\begin{align*}
& \frac{1}{\delta} \Bigl ( J(x_0^\epsilon + \delta a, u^{\epsilon,\delta}(\cdot)) - 	J(x_0^\epsilon, u^{\epsilon}(\cdot))  \Bigr ) \\
& = \frac{1}{\delta} \Biggl ( \int_0^T l(s,x^{\epsilon,\delta}(s),u^{\epsilon,\delta}(s)) \dd s + h(x_0^\epsilon + \delta a, x^{\epsilon,\delta}(T))  - \int_0^T l(s,x^{\epsilon}(s),u^{\epsilon}(s)) \dd s - h(x_0^\epsilon, x^{\epsilon}(T))  \Biggr ) \\
& =  \int_0^T l_x^{\epsilon,\delta}(s) Z^{\epsilon,\delta}(s) \dd s + \int_0^T \frac{\mathds{1}_{E_{\delta}}}{\delta} \widehat{l}(s) \dd s  + h_{x_0}^{\epsilon,\delta}(T) a + h_{x}^{\epsilon,\delta}(T)  Z^{\epsilon,\delta}(T),
%+ \Bigl \langle 
%\begin{bmatrix}
%	h_{x_0}^{\epsilon,\delta}(T) \\
%	h_{x}^{\epsilon,\delta}(T)
%\end{bmatrix}, \begin{bmatrix}
% a \\
% Z^{\epsilon,\delta}(T)	
% \end{bmatrix} \Bigr \rangle,
\end{align*}
where 
\begin{align*}
\widehat{l}(s) & := l(s,x^\epsilon(s),u(s)) - l(s,x^\epsilon(s),u^\epsilon(s)), \\
l_x^{\epsilon,\delta}(s) & := \int_0^1 l_x(s,x^{\epsilon}(s) + r (x^{\epsilon,\delta}(s) - x^{\epsilon}(s)),u^{\epsilon,\delta}(s)) \dd r,
\\
h_{x_0}^{\epsilon,\delta}(T) & := \int_0^1 
 	h_{x_0}(x_0^\epsilon + r \delta a, x^{\epsilon}(T) + r (x^{\epsilon,\delta}(T)-x^{\epsilon}(T)))   \dd r \\
h_{x}^{\epsilon,\delta}(T) & := \int_0^1 
 h_x(x_0^\epsilon + r \delta a, x^{\epsilon}(T) + r (x^{\epsilon,\delta}(T)-x^{\epsilon}(T))) \dd r.
% \\
%h_{x_0}^{\epsilon,\delta}(T) & :=  \int_0^1 h_{x_0}(x_0^\epsilon + r \delta a, x^{\epsilon}(T)) \dd r \\
%h_x^{\epsilon,\delta}(T) & := \int_0^1  h_x(x_0^\epsilon, x^{\epsilon}(T) + r (x^{\epsilon,\delta}(T)-x^{\epsilon}(T))) \dd r. 
\end{align*}

Let us define
\begin{align*}
\widehat{Z}^{\epsilon}(T) & = \int_0^T l_x(s,x^{\epsilon} (s),u^{\epsilon} (s)) Z^{\epsilon}(s) \dd s  + \int_0^T \widehat{l}(s) \dd s + 
 	h_{x_0}(x_0^{\epsilon},x^{\epsilon}(T))a + h_x(x_0^{\epsilon},x^{\epsilon}(T)) 	Z^{\epsilon}(T).
\end{align*}
By definition of $J$ in (\ref{eq_2}),
\begin{align*}
& \frac{1}{\delta} \Bigl ( J(x_0^\epsilon + \delta a, u^{\epsilon,\delta}(\cdot)) - 	J(x_0^\epsilon, u^{\epsilon}(\cdot))  \Bigr ) - 	\widehat{Z}^{\epsilon}(T) \\
& = \int_0^T l_x^{\epsilon,\delta}(s) \Bigl [ Z^{\epsilon,\delta}(s) - Z^{\epsilon}(s) \Bigr ] \dd s + \int_0^T \Bigl [ l_x^{\epsilon,\delta}(s)  - l_x(s,x^{\epsilon} (s),u^{\epsilon} (s)) \Bigr ] Z^{\epsilon}(s) \dd s \\
&~~~ + \int_0^T \Bigl ( \frac{\mathds{1}_{E_{\delta}}}{\delta} - 1 \Bigr ) \widehat{l}(s) \dd s + \Bigl [ h_{x_0}^{\epsilon,\delta}(T) - h_{x_0}(x_0^\epsilon, x^\epsilon(T)) \Bigr ] a \\
&~~~ + h_x^{\epsilon,\delta}(T)\Bigl [ Z^{\epsilon,\delta}(T) - Z^{\epsilon}(T) \Bigr ] + \Bigl [ h_x^{\epsilon,\delta}(T) - h_x(x^{\epsilon}(T))  \Bigr ] Z^{\epsilon}(T).
\end{align*}
Notice that $\lim_{\delta \downarrow 0} |Z^{\epsilon,\delta}(t)	- Z^{\epsilon}(t)| = 0$ for all $t \in [0,T]$ by Lemma \ref{Lemma_4_4_2342341234234}. Moreover, with $\phi(t,s) = \widehat{l}(s)$ in Lemma \ref{Lemma_D_1} of Appendix \ref{Appendix_D}, for any $\delta \in (0,1)$, there exists an $E_{\delta} \in \mathcal{E}_{\delta}$ such that 
\begin{align*}	
\Biggl | \int_0^t \Bigl ( \frac{1}{\delta} \mathds{1}_{E_{\delta}}(s) - 1 \Bigr ) \widehat{l}(s) \dd s \Biggr | & \leq \delta,~ \forall t \in [0,T].
\end{align*}
Hence, by using a similar technique of Lemma \ref{Lemma_4_4_2342341234234}, we can show that
\begin{align}
\label{eq_4_15}
\lim_{\delta \downarrow 0} \Biggl | 	\frac{1}{\delta} \Bigl ( J(x_0^\epsilon + \delta a, u^{\epsilon,\delta}) - 	J(x_0^\epsilon, u^{\epsilon})  \Bigr ) - 	\widehat{Z}^{\epsilon}(T) \Biggr | = 0,
\end{align}
which is equivalent to 
\begin{align}
\label{eq_4_16}
\Bigl | J(x_0^\epsilon + \delta a, u^{\epsilon,\delta}(\cdot)) - 	J(x_0^\epsilon, u^{\epsilon}(\cdot))  - 	\delta \widehat{Z}^{\epsilon}(T) \Bigr |	= o(\delta).
\end{align}

Now, from (\ref{eq_4_13_1_1_1_2}), 
\begin{align}
\label{eq_4_17_1_1_1_1}
-\sqrt{\epsilon}(|a| + T) & \leq \frac{1}{\delta} \Bigl ( J_{\epsilon}(x_0^\epsilon + \delta a, u^{\epsilon,\delta}(\cdot)) - 	J_{\epsilon}(x_0^\epsilon, u^{\epsilon}(\cdot))  \Bigr ) \\
%& = \frac{1}{J_{\epsilon}(x_0^\epsilon + \delta a, u^{\epsilon,\delta}(\cdot)) + 	J_{\epsilon}(x_0^\epsilon, u^{\epsilon}(\cdot))} \times \frac{J_{\epsilon}(x_0^\epsilon + \delta a, u^{\epsilon,\delta}(\cdot))^2 - 	J_{\epsilon}(x_0^\epsilon, u^{\epsilon}(\cdot))^2}{\delta} \nonumber \\
& = \frac{1}{J_{\epsilon}(x_0^\epsilon + \delta a, u^{\epsilon,\delta}(\cdot)) + 	J_{\epsilon}(x_0^\epsilon, u^{\epsilon}(\cdot))} \nonumber \\
&~~~ \times \frac{1}{\delta} \Biggl ( \Bigl ( \bigl [ J(x_0^\epsilon + \delta a, u^{\epsilon,\delta}(\cdot)) - J(\overline{x}_0, \overline{u}(\cdot)) + \epsilon \bigr ]^+ \Bigr  )^2  - \Bigl ( \bigl [ J(x_0^\epsilon, u^{\epsilon}(\cdot)) - J(\overline{x}_0, \overline{u}(\cdot)) + \epsilon \bigr ]^+ \Bigr  )^2 \nonumber \\
&~~~~~~~ + d_F \Bigl ( \begin{bmatrix}
x_0^\epsilon + \delta a \\
x^{\epsilon,\delta}(T)	
\end{bmatrix} \Bigr )^2 - d_F \Bigl ( \begin{bmatrix}
x_0^\epsilon \\
x^{\epsilon}(T)	
\end{bmatrix} \Bigr )^2  + d_S \Bigl ( \gamma(x^{\epsilon,\delta}(\cdot)) \Bigr )^2 - d_S \Bigl ( \gamma(x^{\epsilon}(\cdot)) \Bigr )^2 \Biggr ). \nonumber
\end{align}

By continuity of $J_{\epsilon}$ on $(\mathbb{R}^n \times \mathcal{U}^p[0,T],\widehat{d})$ and (\ref{eq_4_16}), it follows that  $\lim_{\delta \downarrow 0}  J_{\epsilon}(x_0^\epsilon + \delta a, u^{\epsilon,\delta}(\cdot)) = J_{\epsilon}(x_0^\epsilon, u^{\epsilon}(\cdot))$, which leads to
\begin{align*}
\lim_{\delta \downarrow 0} \Bigl \{ J_{\epsilon}(x_0^\epsilon + \delta a, u^{\epsilon,\delta}(\cdot)) + 	J_{\epsilon}(x_0^\epsilon, u^{\epsilon}(\cdot))	\Bigr \} = 2 J_{\epsilon}(x_0^\epsilon, u^{\epsilon}(\cdot)).
\end{align*}
In view of (\ref{eq_4_15}) and Lemma \ref{Lemma_4_4_2342341234234},
\begin{align*}
& \frac{1}{\delta} \Bigl ( \bigl [ J(x_0^\epsilon + \delta a, u^{\epsilon,\delta}(\cdot)) - J(\overline{x}_0, \overline{u}(\cdot)) + \epsilon \bigr ]^+ \Bigr  )^2  - \Bigl ( \bigl [ J(x_0^\epsilon, u^{\epsilon}(\cdot)) - J(\overline{x}_0, \overline{u}(\cdot)) + \epsilon \bigr ]^+ \Bigr  )^2 \\
& = \Biggl (  \bigl [ J(x_0^\epsilon + \delta a, u^{\epsilon,\delta}(\cdot)) - J(\overline{x}_0, \overline{u}(\cdot)) + \epsilon \bigr ]^+   +  \bigl [ J(x_0^\epsilon, u^{\epsilon}(\cdot)) - J(\overline{x}_0, \overline{u}(\cdot)) + \epsilon \bigr ]^+  \Biggr ) \\
&~~~~~~~ \times \frac{1}{\delta}  \Biggl ( \bigl [ J(x_0^\epsilon + \delta a, u^{\epsilon,\delta}(\cdot)) - J(\overline{x}_0, \overline{u}(\cdot)) + \epsilon \bigr ]^+   -  \bigl [ J(x_0^\epsilon, u^{\epsilon}(\cdot)) - J(\overline{x}_0, \overline{u}(\cdot)) + \epsilon \bigr ]^+  \Biggr )
\\
& \rightarrow  2 \bigl [ J(x_0^\epsilon, u^{\epsilon}(\cdot)) - J(\overline{x}_0, \overline{u}(\cdot)) + \epsilon \bigr ]^+ \widehat{Z}^{\epsilon}(T),~ \textrm{as $\delta \downarrow 0$.}
\end{align*}
Let us define (since $J_{\epsilon}(x_0^\epsilon,u^\epsilon(\cdot)) > 0$ by (\ref{eq_4_8}))
\begin{align}
\label{eq_4_17}
\lambda^{\epsilon} := 	\frac{\bigl [ J(x_0^\epsilon, u^{\epsilon}(\cdot)) - J(\overline{x}_0, \overline{u}(\cdot)) + \epsilon \bigr ]^+}{J_{\epsilon}(x_0^\epsilon, u^{\epsilon}(\cdot))} \geq 0.
\end{align}

By Lemmas \ref{Lemma_4_1} and \ref{Lemma_4_4_2342341234234}, and the definition of Fr\'echet differentiability, as $\delta \downarrow 0$,
\begin{align*}	
& \frac{1}{\delta} \Biggl ( d_F \Bigl ( \begin{bmatrix}
	x_0^\epsilon + \delta a \\
	x^{\epsilon,\delta}(T)
\end{bmatrix} \Bigr )^2 - d_F \Bigl ( \begin{bmatrix}
	x_0^\epsilon \\
	x^{\epsilon}(T)
\end{bmatrix} \Bigr )^2 
  \Biggr )  \rightarrow 2 \Biggl \langle \begin{bmatrix}
	x_0^\epsilon \\
	x^{\epsilon}(T)
\end{bmatrix} - P_F \Bigl (\begin{bmatrix}
	x_0^\epsilon \\
	x^{\epsilon}(T)
\end{bmatrix}  \Bigr ), \begin{bmatrix}
a \\
Z^{\epsilon}(T)
 \end{bmatrix} \Biggr \rangle_{\mathbb{R}^{2n} \times \mathbb{R}^{2n}},
\end{align*}
where $P_F: \mathbb{R}^{2n} \rightarrow F \subset \mathbb{R}^{2n}$ is the projection operator defined in Section \ref{Section_4_1}. Notice that by the statement in Section \ref{Section_4_1}, 
\begin{align*}
d_F	\Bigl (\begin{bmatrix}
	x_0^\epsilon \\
	x^{\epsilon}(T)
\end{bmatrix} \Bigr ) = \Biggl | \begin{bmatrix}
	x_0^\epsilon \\
	x^{\epsilon}(T)
\end{bmatrix} - P_F \Bigl (\begin{bmatrix}
	x_0^\epsilon \\
	x^{\epsilon}(T)
\end{bmatrix}  \Bigr ) \Biggr |_{\mathbb{R}^{2n}},
\end{align*}
and by (\ref{eq_4_1}), 
\begin{align*}	
\begin{bmatrix}
	x_0^\epsilon \\
	x^{\epsilon}(T)
\end{bmatrix} - P_F \Bigl (\begin{bmatrix}
	x_0^\epsilon \\
	x^{\epsilon}(T)
\end{bmatrix} \Bigr ) \in N_F \Bigl ( P_F \Bigl (\begin{bmatrix}
	x_0^\epsilon \\
	x^{\epsilon}(T)
\end{bmatrix} \Bigr ) \Bigr ).
\end{align*}
We define (note that $J_{\epsilon}(x_0^\epsilon,u^\epsilon (\cdot)) > 0$ by (\ref{eq_4_8}) and $\xi_1^{\epsilon},\xi_2 ^{\epsilon} \in \mathbb{R}^n$)
\begin{align}
\label{eq_4_18}
\xi^{\epsilon} := \begin{bmatrix}
 \xi^{\epsilon}_1 \\
 \xi^{\epsilon}_2	
 \end{bmatrix}
 := \frac{\begin{bmatrix}
	x_0^\epsilon \\
	x^{\epsilon}(T)
\end{bmatrix} - P_F \Bigl (\begin{bmatrix}
	x_0^\epsilon \\
	x^{\epsilon}(T)
\end{bmatrix} \Bigr ) }{J_{\epsilon}(x_0^\epsilon,u^\epsilon (\cdot) )}	\in N_F \Bigl ( P_F \Bigl (\begin{bmatrix}
	x_0^\epsilon \\
	x^{\epsilon}(T)
\end{bmatrix} \Bigr ) \Bigr ).
\end{align}

By Lemma \ref{Lemma_2_1} (see also Lemmas \ref{Lemma_B_1} and \ref{Lemma_B_2} in Appendix \ref{Appendix_B}), it holds that $Z^{\epsilon}(\cdot) \in C([0,T];\mathbb{R}^n)$. Then using Lemmas \ref{Lemma_4_3_1_2_2_1} and \ref{Lemma_4_4_2342341234234}, as $\delta \downarrow 0$, we get
\begin{align*}
& \frac{1}{\delta} \Biggl (  d_S \Bigl (\gamma(x^{\epsilon,\delta}(\cdot)) \Bigr )^2 - d_S \Bigl (\gamma(x^{\epsilon}(\cdot)) \Bigr )^2 \Biggr )  \\
&\rightarrow \begin{cases}
2 \Biggl \langle d_S \Bigl ( \gamma(x^{\epsilon}(\cdot)) \Bigr ) D d_S \Bigl ( ( \gamma(x^{\epsilon}(\cdot)) \Bigr ),   G_x(\cdot,x^{\epsilon}(\cdot)) Z^{\epsilon}(\cdot) \Biggr \rangle_{C_m^* \times C_m}, & 	 \gamma(x^{\epsilon}(\cdot)) \notin S, \\
0 \in \mathbb{R}, & \gamma(x^{\epsilon}(\cdot)) \in S .
 \end{cases}
\end{align*}
We define (since $J_{\epsilon}(x_0^\epsilon,u^\epsilon (\cdot)) > 0$ by (\ref{eq_4_8}))
\begin{align}
\label{eq_4_19}
\mu^{\epsilon} := \begin{dcases}
\frac{d_S \Bigl (\gamma(x^{\epsilon}(\cdot)) \Bigr ) D d_S \Bigl (\gamma(x^{\epsilon}(\cdot)) \Bigr )}{J_{\epsilon}(x_0^\epsilon,u^\epsilon (\cdot) )} \in C([0,T]; \mathbb{R}^m)^*, &  \gamma(x^{\epsilon}(\cdot)) \notin S,  \\
0 \in C([0,T]; \mathbb{R}^m)^*, & \gamma(x^{\epsilon}(\cdot)) \in S,
 \end{dcases} 
\end{align}
and since $D d_S \Bigl ( \gamma(x^{\epsilon}(\cdot)) \Bigr )$ is the subdifferential of $d_S \Bigl ( \gamma(x^{\epsilon}(\cdot)) \Bigr )$ at $\gamma(x^{\epsilon}(\cdot)) \in S$ (see Lemma \ref{Lemma_4_3_1_2_2_1}), by (\ref{eq_4_3}) and (\ref{eq_4_5_6_5_43_3_5_7_3_3}), we have
\begin{align}
\label{eq_4_26_23423425345}
\mu^{\epsilon} \in N_S \Bigl (\gamma(x^{\epsilon}(\cdot)) \Bigr ). 	
\end{align}

In view of Lemma \ref{Lemma_4_3_1_2_2_1} and the definitions of $J_\epsilon$, $d_F$ and $d_S$, this leads to
\begin{align}
\label{eq_4_20_1_1_1_1_2}
& |\lambda^\epsilon|^2 + 	|\xi^\epsilon|^2_{\mathbb{R}^{2n}} + \|\mu^\epsilon\|^2_{C([0,T]; \mathbb{R}^m)^*} = 1.
\end{align}
Hence, as $\delta \downarrow 0$, applying (\ref{eq_4_17})-(\ref{eq_4_26_23423425345}) to (\ref{eq_4_17_1_1_1_1}) yields
\begin{align}	
\label{eq_4_20}
-\sqrt{\epsilon}(|a| + T) & \leq  \lambda^{\epsilon} \widehat{Z}^{\epsilon}(T) + \Bigl \langle \xi_1^{\epsilon}, a \Bigr \rangle + \Bigl \langle \xi_2^{\epsilon}, Z^{\epsilon}(T) \Bigr \rangle + \Bigl \langle \mu^{\epsilon},  G_x(\cdot,x^{\epsilon}(\cdot)) Z^{\epsilon}(\cdot) \Bigr \rangle_{C_m^* \times C_m}.
\end{align}

The following lemma shows the estimate between the first and second variational equations, where the first variational equation is given in Lemma \ref{Lemma_4_4_2342341234234}.

\begin{lemma}\label{Lemma_4_5}
For any $(a,u(\cdot)) \in \mathbb{R}^n \times \mathcal{U}^p[0,T]$, the following results hold:
\begin{align*}
& \textrm{(i)}~ \lim_{\epsilon \downarrow 0} \Bigl \{ |x_0^\epsilon - \overline{x}_0 |_{\mathbb{R}^n} + \overline{d}(u^\epsilon(\cdot),\overline{u}(\cdot)) \Bigr \} = 0, \\
& \textrm{(ii)}~  \sup_{t \in [0,T]} \Bigl | Z^{\epsilon}(t;a,u) - Z(t;a,u) \Bigr | = o(\epsilon),~ \Bigl | \widehat{Z}^{\epsilon}(T;a,u) - \widehat{Z}(T;a,u) \Bigr | = o(\epsilon),
\end{align*}
where $Z(\cdot) := Z(\cdot;a,u)$ is the solution to the second variational equation related to $(\overline{x},\overline{u}(\cdot))$ and $\widehat{Z}(\cdot) := \widehat{Z}(\cdot;a,u)$ is the variational equation of $J$, both of which are given below
\begin{align*}
& Z(t)  = a + \int_0^t \Bigl [ \frac{f_x(t,s,\overline{x}(s),\overline{u}(s))}{(t-s)^{1-\alpha}} Z(s) \dd s 	 + \frac{f(t,s,\overline{x}(s),u(s)) - f(t,s,\overline{x}(s),\overline{u}(s))}{(t-s)^{1-\alpha}}  \Bigr ] \dd s \\
&~~~ + \int_0^t \Bigl [ g_x(t,s,\overline{x}(s),\overline{u}(s)) Z(s) + \bigl ( g(t,s,\overline{x}(s),u(s)) - g(t,s,\overline{x}(s),\overline{u}(s)) \bigr ) \Bigr ] \dd s,~  \textrm{a.e.}~ t \in [0,T], \\
& \widehat{Z}(T)  = \int_0^T l_x(s,\overline{x}(s),\overline{u}(s)) Z(s) \dd s + \int_0^T \Bigl [ l(s,\overline{x}(s),u(s)) - l(s,\overline{x}(s),\overline{u}(s)) \Bigr ] \dd s \\
&~~~  + h_{x_0}(\overline{x}_0,\overline{x}(T))a + h_x(\overline{x}_0,\overline{x}(T)) Z(T),~  \textrm{a.e.}~ t \in [0,T].
\end{align*}
\end{lemma}

\begin{remark}\label{Remark_4_8_4534535}
Note that (i) of Lemma \ref{Lemma_4_5} follows from the definition of the Ekeland metric in (\ref{eq_4_6}) (see also (\ref{eq_4_9_1_1_1})). The proof for (ii) of Lemma \ref{Lemma_4_5} is similar to that for Lemma \ref{Lemma_4_4_2342341234234}.	
\end{remark}

We now consider the limit of $\epsilon \downarrow 0$.  Instead of taking the limit with respect to $\epsilon \downarrow 0$, let $\{\epsilon_k\}$ be the sequence of  $\epsilon$ such that $\epsilon_k \geq 0$ and $\epsilon_k \downarrow 0$ as $k \rightarrow \infty$. We replace $\epsilon$ by $\epsilon_k$. Then by (\ref{eq_4_20_1_1_1_1_2}),  the sequences $(\{\lambda^{\epsilon_k} \}, \{\xi^{\epsilon_k} \},\{\mu^{\epsilon_k}\}) $ are bounded for $k \geq 0$. Note also from (\ref{eq_4_20_1_1_1_1_2}) that the ball generated by $\|\mu^{\epsilon_k}\|^2_{C([0,T];\mathbb{R}^m)^*} \leq 1$ is a closed unit ball in $C([0,T];\mathbb{R}^m)^*$, which is weak--$*$ compact by the Banach-Alaoglu theorem \cite[page 130]{Conway_2000_book}. Then by the standard compactness argument, we may extract a subsequence of $\{\epsilon_k\}$, still denoted by $\{\epsilon_k\}$, such that 
\begin{align}
\label{eq_4_28_34_34234343234}
(\{\lambda^{\epsilon_k} \}, \{\xi^{\epsilon_k} \},\{\mu^{\epsilon_k}\}) \rightarrow (\lambda^0, \xi^0,\mu^0) = :(\lambda, \xi,\mu),~ \textrm{as $k \rightarrow \infty$,}
\end{align}
where $\{\mu^{\epsilon_k}\} \rightarrow \mu$ (as $k \rightarrow \infty$) is understood in the weak--$*$ sense \cite{Conway_2000_book}. 

We claim that from (\ref{eq_4_17})-(\ref{eq_4_26_23423425345}), the tuple $(\lambda,\xi,\mu)$ holds
\begin{subequations}
\label{eq_4_22}
\begin{align}
\label{eq_4_22_234234234234}
	\lambda &\geq 0, \\
\label{eq_4_22_234234234234_wdsdf}
	\xi &\in N_F \Bigl (  P_F \Bigl (\begin{bmatrix}
	\overline{x}_0, \\
		\overline{x}(T)
\end{bmatrix} \Bigr ) \Bigr ), \\
\label{eq_4_22_234234234234_wdsdfdsfsdf}
\mu &\in N_S \Bigl ( \gamma(\overline{x}(\cdot)) \Bigr ).
\end{align}
\end{subequations}
Indeed, (\ref{eq_4_22_234234234234}) holds due to (\ref{eq_4_17}). Furthermore, (\ref{eq_4_22_234234234234_wdsdf}) follows from (\ref{eq_4_18}) and the property of limiting normal cones \cite[page 43]{Vinter_book}. To prove (\ref{eq_4_22_234234234234_wdsdfdsfsdf}), we note that (\ref{eq_4_26_23423425345}) and (\ref{eq_4_3}) mean that $ \langle \mu^{\epsilon_k}, z - \gamma(x^{\epsilon_k}(\cdot)) \rangle_{C_m^* \times C_m} \leq 0$ for any $z \in S$. Then (\ref{eq_4_22_234234234234_wdsdfdsfsdf}) holds, since by (\ref{eq_4_3}), (\ref{eq_4_20_1_1_1_1_2}) and (\ref{eq_4_28_34_34234343234}), together with the boundedness of $\{\mu^{\epsilon_k}\}$, Lemma \ref{Lemma_4_5}, and the weak--$*$ convergence property of $\{\mu^{\epsilon_k}\}$ to $\mu$, it holds that
\begin{align*}
0 \geq \Bigl \langle \mu^{\epsilon_k}, z - \gamma(x^{\epsilon_k}(\cdot) ) \Bigr \rangle_{C_m^* \times C_m}	& \geq  \Bigl \langle \mu, z - \gamma(\overline{x}(\cdot) ) \Bigr \rangle_{C_m^* \times C_m} - \Bigl \|\gamma(x^{\epsilon_k}(\cdot)) - \gamma(\overline{x}(\cdot)) \Bigr \|_{\infty} \\
&~~~ + \Bigl \langle \mu^{\epsilon_k}, z - \gamma(\overline{x}(\cdot)) \Bigr \rangle_{C_m^* \times C_m} - \Bigl \langle \mu, z - \gamma(\overline{x}(\cdot)) \Bigr \rangle_{C_m^* \times C_m} \\
& \rightarrow ~\Bigl \langle \mu, z - \gamma(\overline{x}(\cdot) ) \Bigr \rangle_{C_m^* \times C_m},~ \textrm{as $k \rightarrow \infty$.}
\end{align*}
%This shows $\mu \in  N_S \Bigl ( \gamma(\overline{x}(\cdot;\overline{x}_0,\overline{u})) \Bigr )$. Hence, (\ref{eq_4_22}) follows.
%and
%\begin{align*}
%|\lambda|^2 + 	|\xi|^2_{\mathbb{R}^{2n}} + \|\mu\|^2_{C([0,T];\mathbb{R}^m)^*} &\leq 1.
%\end{align*}

By (\ref{eq_4_28_34_34234343234}) and (\ref{eq_4_20_1_1_1_1_2}), together with Lemma \ref{Lemma_4_5}, it follows that
\begin{align*}
\lambda^{\epsilon_k} \widehat{Z}^{\epsilon_k}(T)  & \leq \lambda \widehat{Z}(T) +  |\widehat{Z}^{\epsilon_k}(T) - \widehat{Z}(T)| + |\lambda^{\epsilon_k} - \lambda | \widehat{Z}(T)~\rightarrow ~ \lambda \widehat{Z}(T),~ \textrm{as $k \rightarrow \infty$,}	 \\
\Bigl \langle \xi_1^{\epsilon_k}, a \Bigr \rangle  &= 	  \Bigl \langle \xi_1, a \Bigr \rangle + \Bigl \langle \xi_1^{\epsilon_k}, a \Bigr \rangle - \Bigl \langle \xi_1, a \Bigr \rangle~ \rightarrow~ \Bigl \langle \xi_1, a \Bigr \rangle,~ \textrm{as $k \rightarrow \infty$,} \\
\Bigl \langle \xi_2^{\epsilon_k}, Z^{\epsilon_k}(T) \Bigr \rangle  & \leq \Bigl \langle \xi_2, Z(T) \Bigr \rangle + |Z^{\epsilon_k}(T) - Z(T)| + |\xi_2^{\epsilon_k} - \xi_2||Z(T)|~\rightarrow \Bigl \langle \xi_2, Z(T) \Bigr \rangle,~ \textrm{as $k \rightarrow \infty$,}
\end{align*}
and similarly, together with the definition of the weak--$*$ convergence,
\begin{align*}
\Bigl \langle \mu^{\epsilon_k},  G_x(\cdot,x^{\epsilon_k}(\cdot)) Z^{\epsilon_k}(\cdot) \Bigr \rangle_{C_m^* \times C_m} & \leq \Bigl \langle \mu,  G_x(\cdot,x(\cdot)) Z(\cdot) \Bigr \rangle_{C_m^* \times C_m} + \| Z^{\epsilon_k}(\cdot) - Z(\cdot)\|_{\infty} \\
&~~~ + \Bigl \langle \mu^{\epsilon_k},  G_x(\cdot,x(\cdot)) Z(\cdot) \Bigr \rangle_{C_m^* \times C_m} - \Bigl \langle \mu,  G_x(\cdot,x(\cdot)) Z(\cdot) \Bigr \rangle_{C_m^* \times C_m}	\\
& \rightarrow ~ \Bigl \langle \mu,  G_x(\cdot,x(\cdot)) Z(\cdot) \Bigr \rangle_{C_m^* \times C_m},~ \textrm{as $k \rightarrow \infty$.}
\end{align*}
Therefore, as $k \rightarrow \infty$, (\ref{eq_4_20}) becomes for any $(a,u) \in \mathbb{R}^n \times \mathcal{U}^p[0,T]$,
\begin{align}	
\label{eq_4_23}
0 & \leq  \lambda \widehat{Z}(T) + \Bigl \langle \xi_1, a \Bigr \rangle + \Bigl \langle \xi_2, Z(T) \Bigr \rangle + \Bigl \langle \mu,  G_x(\cdot, \overline{x}(\cdot)) Z(\cdot;a,u) \Bigr \rangle_{C_m^* \times C_m}.
\end{align}
Note that (\ref{eq_4_23}) is the crucial inequality obtained from the Ekeland variational principle as well as the estimates of the variational equations in Lemmas \ref{Lemma_4_4_2342341234234} and \ref{Lemma_4_5}. 

\subsection{Proof of Theorem \ref{Theorem_3_1}: Complementary Slackness Condition}\label{Section_4_5}

We prove the complementary slackness condition in Theorem \ref{Theorem_3_1}. Let $\mu = (\mu_1,\ldots, \mu_m) \in C([0,T];\mathbb{R}^m)^*$, where $\mu_i \in C([0,T];\mathbb{R})^*$, $i=1,\ldots,m$. Then it holds that
\begin{align}	
\label{Eq_4_31_3434534234234123}
\Bigl \langle \mu, z \Bigr \rangle_{C_m^* \times C_m} = \sum_{i=1}^m \Bigl \langle \mu_i, z_i \Bigr \rangle_{C_1^* \times C_1},~ \forall z = (z_1,\ldots,z_m) \in C([0,T];\mathbb{R}^m)
\end{align}
where $\langle \cdot,\cdot \rangle_{C_1^* \times C_1} := \langle \cdot,\cdot \rangle_{C([0,T];\mathbb{R})^* \times C([0,T];\mathbb{R})}$ denotes the duality paring between $C([0,T];\mathbb{R})$ and $C([0,T];\mathbb{R})^*$. 

Recall $\gamma(\overline{x}(\cdot)) = (\gamma_1(\overline{x}(\cdot)),\ldots, \gamma_m(\overline{x}(\cdot))) = G(\cdot,\overline{x}(\cdot)) = \begin{bmatrix}
 	G^1(\cdot,\overline{x}(\cdot)) & \cdots & G^m(\cdot,\overline{x}(\cdot))
 \end{bmatrix} \in S$ and $\mu \in N_S \Bigl (\gamma(\overline{x}(\cdot;\overline{x}_0,\overline{u})) \Bigr )$ by (\ref{eq_4_22_234234234234_wdsdfdsfsdf}). Based on (\ref{Eq_4_31_3434534234234123}) and (\ref{eq_4_3}), this implies that for any $z \in S$,
\begin{align}
\label{eq_4_24}	
\Bigl \langle \mu, z - \gamma(\overline{x}(\cdot;\overline{x}_0,\overline{u}))  \Bigr \rangle_{C_m^* \times C_m} = \sum_{i=1}^m \Bigl \langle \mu_i, z_i - \gamma_i (\overline{x}(\cdot;\overline{x}_0,\overline{u})) \Bigr \rangle_{C_1^* \times C_1} \leq 0.
\end{align}
Taking $z$ in (\ref{eq_4_24}) as follows:
\begin{align*}
z &= \begin{bmatrix}
 G^1(\cdot,\overline{x}(\cdot)) & \cdots & G^{i-1}(\cdot,\overline{x}(\cdot)) & 2G^{i}(\cdot,\overline{x}(\cdot)) & G^{i+1}(\cdot,\overline{x}(\cdot)) & \cdots & G^m(\cdot,\overline{x}(\cdot))	
 \end{bmatrix} \in S, \\
 z^{(-i)} &= 
 \begin{bmatrix}
 	G^1(\cdot,\overline{x}(\cdot)) & \cdots & G^{i-1}(\cdot,\overline{x}(\cdot)) & 0_{\in C([0,T];\mathbb{R})} & G^{i+1}(\cdot,\overline{x}(\cdot)) & \cdots &  G^m(\cdot,\overline{x}(\cdot)) 
 \end{bmatrix} \in S.
\end{align*}
Then (\ref{eq_4_24}) is equivalent to
\begin{align}
\label{eq_4_25}
\Bigl \langle \mu_i, G^i(\cdot,\overline{x}(\cdot;\overline{x}_0;\overline{u})) \Bigr \rangle_{C_1^* \times C_1} & = 0,~ \forall i=1,\ldots,m ,
\\
\label{eq_4_26}
\Bigl \langle \mu_i, z_i  \Bigr \rangle_{C_1^* \times C_1} & \geq 0,~ \forall z_i \in C([0,T];\mathbb{R}_{+}),~ i=1,\ldots,m.
\end{align}

For (\ref{eq_4_25}) and (\ref{eq_4_26}), by the Riesz representation theorem (see \cite[page 75 and page 382]{Conway_2000_book} and \cite[Theorem 14.5]{Limaye_book}), there is a unique $\theta(\cdot) = (\theta_1(\cdot),\ldots,\theta_m(\cdot)) \in \textsc{NBV}([0,T];\mathbb{R}^m)$ with $\theta_i(\cdot) \in \textsc{NBV}([0,T];\mathbb{R})$, i.e., $\theta_i$, $i=1,\ldots,m$, being the normalized functions of bounded variation on $[0,T]$, such that every $\theta_i$ is finite, nonnegative, and monotonically nondecreasing on $[0,T]$ with $\theta_i(0) = 0$. Moreover, the Riesz representation theorem leads to the following representation:
\begin{align}	
\Bigl \langle \mu_i, \gamma_i(\overline{x}(\cdot;\overline{x}_0,\overline{u})) \Bigr \rangle_{C_1^* \times C_1} &= \int_0^T G^i(s,\overline{x}(s;\overline{x}_0,\overline{u})) \dd \theta_i(s) = 0,~ \forall i=1,\ldots,m,   \nonumber \\
\label{eq_4_27}
\langle \mu_i, z_i  \Bigr \rangle_{C_1^* \times C_1} &= \int_0^T z_i(s) \dd \theta_i(s) \geq 0,~ \forall z_i \in C([0,T];\mathbb{R}_{+}),~ i=1,\ldots,m.
\end{align}
%Due to the fact that $\theta_i$ is monotonically nondecreasing, it holds that $\dd \theta_i(s) \geq 0$ for $s \in [0,T]$ and $i=1,\ldots,m$, where $\dd \theta_i$ denotes the Lebesgue-Stieltjes measure corresponding to $\theta_i$, $i=1,\ldots,m$. 

Notice that (\ref{eq_4_27}) always holds as $\theta_i$ is monotonically nondecreasing on $[0,T]$ with $\theta(0) = 0$ (equivalently, $\dd \theta_i$ is nonnegative) and $z_i \in C([0,T];\mathbb{R}_{+})$. Hence,  (\ref{eq_4_25}) and (\ref{eq_4_26}) are reduced to
\begin{align}	
\label{eq_4_27_4_23_3_4_2_}
& \Bigl \langle \mu_i, \gamma_i(\overline{x}(\cdot;\overline{x}_0;\overline{u})) \Bigr \rangle_{C_1^* \times C_1} = \int_0^T G^i(s,\overline{x}(s;\overline{x}_0,\overline{u})) \dd \theta_i(s) = 0,~ \forall i=1,\ldots,m, \\
& \Leftrightarrow~ 	\textsc{supp}(\dd \theta_i(\cdot)) \subset \{ t \in [0,T]~|~ G^i(t,\overline{x}(t;\overline{x}_0,\overline{u}))= 0\},~ \forall i=1,\ldots,m, \nonumber
\end{align}
where the equivalence follows from the fact that $G^i(t,\overline{x}(t;\overline{x}_0,\overline{u})) \leq 0$, $i=1,\ldots,m$, and $\dd \theta_i$, $i=1,\ldots,m$, are finite nonnegative measures on $([0,T],\mathcal{B}([0,T]))$. The relation in (\ref{eq_4_27_4_23_3_4_2_}) proves the complementary slackness condition in Theorem \ref{Theorem_3_1}.

\subsection{Proof of Theorem \ref{Theorem_3_1}: Nontriviality and Nonnegativity Conditions}\label{Section_4_6}

We prove the nontriviality and nonnegativity conditions in Theorem \ref{Theorem_3_1}. Recall (\ref{eq_4_24}), i.e., for any $z \in S$, 
\begin{align}
\label{eq_4_24_2345234234234223423}	
\Bigl \langle \mu, z - \gamma(\overline{x}(\cdot;\overline{x}_0,\overline{u})) \Bigr \rangle_{C_m^* \times C_m} = \sum_{i=1}^m \Bigl \langle \mu_i, z_i - G^i(\cdot,\overline{x}(\cdot;\overline{x}_0,\overline{u})) \Bigr \rangle_{C_1^* \times C_1} \leq 0.
\end{align}
%which shows that $\mu$ is nontrivial, i.e., $\|\mu_i \|_{C([0,T];\mathbb{R})^*} = 0 $, $i=1,\ldots,m$, cannot be true for all cases. 
Then by the Riesz representation theorem (see \cite[page 75 and page 382]{Conway_2000_book} and \cite[Theorem 14.5]{Limaye_book}) and the fact that $\theta_i$, $i=1,\ldots,m$, is finite, nonnegative, and monotonically nondecreasing on $[0,T]$ with $\theta_i(0) = 0$ (see Section \ref{Section_4_5}), it follows that $\|\mu_i \|_{C([0,T];\mathbb{R}^m)^*} = \|\theta_i(\cdot)\|_{\textsc{NBV}([0,T];\mathbb{R})}  = \theta_i(T) \geq 0$ for $i=1,\ldots,m$. In addition, as $\theta_i$ is monotonically nondecreasing, we have $\dd \theta_i(s) \geq 0$ for $s \in [0,T]$, where $\dd \theta_i$ denotes the Lebesgue-Stieltjes measure corresponding to $\theta_i$, $i=1,\ldots,m$. 

By (\ref{eq_4_22_234234234234_wdsdf}) and the fact that $\begin{bmatrix}
\overline{x}_0 \\\overline{x}(T)) \end{bmatrix} \in F$ implies $P_F \Bigl (\begin{bmatrix}
\overline{x}_0 \\\overline{x}(T)) \end{bmatrix} \Bigr ) = \begin{bmatrix}\overline{x}_0 \\\overline{x}(T) \end{bmatrix}$ (see Section \ref{Section_4_1}), we have $\xi = \begin{bmatrix}
 \xi_1 \\
 \xi_2	
 \end{bmatrix} \in N_F \Bigl ( \begin{bmatrix}\overline{x}_0 \\\overline{x}(T) \end{bmatrix} \Bigr )$. 
% Moreover, it follows from (\ref{eq_4_1}) that
%\begin{align*}
%\Bigl \langle \xi, y -  \begin{bmatrix}\overline{x}_0 \\ \overline{x}(T) \end{bmatrix} \Bigr \rangle_{\mathbb{R}^{2n} \times \mathbb{R}^{2n}} & = \Bigl \langle \xi_1, y_1 -  \overline{x}_0 \Bigr \rangle_{\mathbb{R}^{n} \times \mathbb{R}^{n}} +  \Bigl \langle \xi_2, y_2 -  \overline{x}(T) \Bigr \rangle_{\mathbb{R}^{n} \times \mathbb{R}^{n}} \leq 0,~ \forall y = \begin{bmatrix}
% y_1 \\
% y_2
%\end{bmatrix} \in F.
%\end{align*}
%This shows that $\xi \in N_F \Bigl ( \begin{bmatrix}\overline{x}_0 \\\overline{x}(T) \end{bmatrix} \Bigr )$ is not trivial, i.e., $\xi = 0$ cannot be true for all cases. 
In addition, from the fact that $S = C([0,T];\mathbb{R}_{-}^m)$ has an nonempty interior, there are $z^\prime \in S$ and $\sigma > 0$ such that $z^\prime + \sigma z \in S$ for all $z \in \overline{B}_{(C([0,T];\mathbb{R}^n),\|\cdot\|_{C([0,T];\mathbb{R}^n)})}(0,1)$ (the closure of the unit ball in $C([0,T];\mathbb{R}^n)$). Then by (\ref{eq_4_24_2345234234234223423}), it follows that
\begin{align*}	
\sigma \Bigl \langle \mu,z \Bigr \rangle_{C_m^* \times C_m} \leq \Bigl \langle \mu,  \gamma(\overline{x}(\cdot)) - z^\prime	 \Bigr \rangle_{C_m^* \times C_m},~ \forall z \in \overline{B}_{(C([0,T];\mathbb{R}^n),\|\cdot\|_{C([0,T];\mathbb{R}^n)})}(0,1).
\end{align*}
By (\ref{eq_4_20_1_1_1_1_2}) 
%, we must have $\|\mu\|_{C([0,T]; \mathbb{R}^m)^*}  = \sqrt{1 - |\lambda|^2 - 	|\xi|^2_{\mathbb{R}^{2n}}}$. Then by 
and the definition of the norm of the dual space (the norm of linear functionals on $C([0,T];\mathbb{R}^m)$ (see Section \ref{Section_4_1})), we get
\begin{align*}
\sigma \|	\mu \|_{C([0,T];\mathbb{R}^m)^*}  = \sigma \sqrt{1-|\lambda|^2 - 	|\xi|^2_{\mathbb{R}^{2n}}} \leq \Bigl \langle \mu,  \gamma(x^{\epsilon}(\cdot)) - z^\prime	 \Bigr \rangle_{C_m^* \times C_m},~ z^\prime \in S.
\end{align*}
Notice that $\sigma > 0$. When $\mu = 0 \in C([0,T];\mathbb{R}^m)^*$ and $\xi = 0$, we must have $\lambda = 1$. When $\lambda = 0$ and $\mu = 0 \in C([0,T];\mathbb{R}^m)^*$, we must have $|\xi|_{\mathbb{R}^{2n}} = 1$. When $\lambda = 0$ and $\xi = 0$, it holds that $\mu \neq 0 \in C([0,T];\mathbb{R}^m)^*$. This implies that the tuple $(\lambda,\xi,\theta_1(\cdot),\ldots,\theta_m(\cdot))$ cannot be trivial, i.e., $(\lambda,\xi,\theta_1(\cdot),\ldots,\theta_m(\cdot)) \neq 0$ (they cannot be zero simultaneously).

In summary, based on the above discussion, it follows that the following tuple
\begin{align*}
\begin{cases}
	\lambda  \geq 0, \\
	\xi \in N_F \Bigl ( \begin{bmatrix}\overline{x}_0 \\\overline{x}(T) \end{bmatrix} \Bigr ), \\
	\|\mu_i \|_{ C([0,T];\mathbb{R}^m)^*}  = \|\theta_i(\cdot)\|_{\textsc{NBV}([0,T];\mathbb{R})} = \theta_i(T) \geq 0,~\forall i=1,\ldots,m
\end{cases}
\end{align*}
cannot be trivial, i.e., it holds that $(\lambda,\xi,\theta_1(\cdot),\ldots,\theta_m(\cdot)) \neq 0$, and
\begin{align*}
	& \begin{cases}
	\lambda \geq 0, \\
	\dd \theta_i(s) \geq 0,~ \forall s \in [0,T],~i=1, \ldots, m.
	\end{cases}
\end{align*}
This shows the nontriviality and nonnegativity conditions in Theorem \ref{Theorem_3_1}.

\subsection{Proof of Theorem \ref{Theorem_3_1}: Adjoint Equation and Duality Analysis}\label{Section_4_7}
Recall the variational inequality in (\ref{eq_4_23}), i.e., for any $(a,u) \in \mathbb{R}^n \times \mathcal{U}^p[0,T]$,
\begin{align}
\label{eq_5_40_23423402302020202}
0 & \leq  \lambda \widehat{Z}(T;a,u) + \Bigl \langle \xi_1, a \Bigr \rangle + \Bigl \langle \xi_2, Z(T;a,u) \Bigr \rangle + \Bigl \langle \mu,  G_x(\cdot, \overline{x}(\cdot)) Z(\cdot;a,u) \Bigr \rangle_{C_m^* \times C_m}.
\end{align}
Similar to (\ref{eq_4_27_4_23_3_4_2_}), by the Riesz representation theorem, it holds that
\begin{align*}	
\Bigl \langle \mu,  G_x(\cdot,\overline{x}(\cdot;\overline{x}_0,\overline{u})) Z(\cdot;a,u) \Bigr \rangle_{C_m^* \times C_m} & = \sum_{i=1}^m \Bigl \langle \mu_i,  G_x^i(\cdot,\overline{x}(\cdot;\overline{x}_0,\overline{u})) Z(\cdot;a,u)  \Bigr \rangle_{C_1^* \times C_1} \\
& = \sum_{i=1}^m \int_0^T G_x^{i} (s,\overline{x}(s)) Z(s;a,u) ) \dd \theta_i(s),
\end{align*}
where as shown in Section \ref{Section_4_5}, we have $\theta(\cdot) = (\theta_1(\cdot),\ldots,\theta_m(\cdot)) \in \textsc{NBV}([0,T];\mathbb{R}^m)$ with $\theta_i (\cdot) \in \textsc{NBV}([0,T];\mathbb{R})$ being finite and monotonically nondecreasing on $[0,T]$. 

Then by using the variational equations in Lemma \ref{Lemma_4_5}, (\ref{eq_5_40_23423402302020202}) becomes
\begin{align}
\label{eq_4_30_4_2323234_2}	
0 & \leq  
\Bigl \langle \xi_1 + \lambda h_{x_0}(\overline{x}_0,\overline{x}(T))^\top, a \Bigr \rangle + \Bigl \langle \xi_2 + \lambda h_x(\overline{x}_0,\overline{x}(T))^\top, a \Bigr \rangle  + \sum_{i=1}^m \int_0^T G_x^{i} (s,\overline{x}(s)) Z(s) \dd \theta_i(s)  \\
&~~~ + \int_0^T \Bigl [ \Bigl ( \lambda h_x(\overline{x}_0,\overline{x}(T)) + \xi_2^\top \Bigr ) \mathds{1}_{[0,T)}(s) \frac{f_x(T,s,\overline{x}(s),\overline{u}(s))}{(T-s)^{1-\alpha}}   \nonumber \\
&~~~~~~~~~~ + \Bigl ( \lambda h_x(\overline{x}_0,\overline{x}(T)) + \xi_2^\top \Bigr )  g_x(T,s,\overline{x}(s),\overline{u}(s))  + \lambda l_x(s,\overline{x}(s),\overline{u}(s)) \Bigr ] Z(s) \dd s \nonumber \\
&~~~ + \int_0^T \Bigl ( \lambda h_x(\overline{x}_0,\overline{x}(T)) + \xi_2^\top \Bigr ) \mathds{1}_{[0,T)}(s) \frac{f(T,s,\overline{x}(s),u(s)) - f(T,s,\overline{x}(s),\overline{u}(s))}{(T-s)^{1-\alpha}} \dd s \nonumber \\
&~~~ + \int_0^T \Bigl ( \lambda h_x(\overline{x}_0,\overline{x}(T)) + \xi_2^\top \Bigr ) \Bigl ( g(T,s,\overline{x}(s),u(s)) - g(T,s,\overline{x}(s),\overline{u}(s))  \Bigr ) \dd s \nonumber \\
&~~~ + \int_0^T \lambda  \Bigl [ l(s,\overline{x}(s),u(s)) - l(s,\overline{x}(s),\overline{u}(s)) \Bigr ] \dd s .\nonumber
\end{align}
Based on Lemma \ref{Lemma_B_5} and Remark \ref{Remark_3_3}, let $p(\cdot) \in L^p([0,T];\mathbb{R}^n)$ be the unique solution to the adjoint equation in Theorem \ref{Theorem_3_1}. Applying it to (\ref{eq_4_30_4_2323234_2}) yields
\begin{align}
\label{eq_4_31_32342323423423423}
0 & \leq 	
\Bigl \langle \xi_1 + \lambda h_{x_0}(\overline{x}_0,\overline{x}(T))^\top, a \Bigr \rangle + \Bigl \langle \xi_2 + \lambda h_x(\overline{x}_0,\overline{x}(T))^\top, a \Bigr \rangle \\
&~~~ + \sum_{i=1}^m \int_0^T G_x^{i} (s,\overline{x}(s)) Z(s;a,u) \dd \theta_i(s)  + \int_0^T \Biggl [ - p(s) - \sum_{i=1}^m G_x^{i} (s,\overline{x}(s))^\top \frac{\dd \theta_i(s)}{\dd s} \nonumber \\
&~~~~~~~~~~ + \int_s^T \frac{f_x(r,s,\overline{x}(s),\overline{u}(s))^\top}{(r-s)^{1-\alpha}} p(r) \dd r + \int_s^T g_x(r,s,\overline{x}(s),\overline{u}(s))^\top p(r) \dd r \Biggr ]^\top Z(s;a,u) \dd s \nonumber \\
&~~~ + \int_0^T \Bigl ( \lambda h_x(\overline{x}_0,\overline{x}(T)) + \xi_2^\top \Bigr ) \mathds{1}_{[0,T)}(s) \frac{f(T,s,\overline{x}(s),u(s)) - f(T,s,\overline{x}(s),\overline{u}(s))}{(T-s)^{1-\alpha}} \dd s \nonumber \\
&~~~ + \int_0^T \Bigl ( \lambda h_x(\overline{x}_0,\overline{x}(T)) + \xi_2^\top \Bigr ) \Bigl [ g(T,s,\overline{x}(s),u(s)) - g(T,s,\overline{x}(s),\overline{u}(s))  \Bigr ] \dd s \nonumber \\
&~~~ + \int_0^T \lambda  \Bigl [ l(s,\overline{x}(s),u(s)) - l(s,\overline{x}(s),\overline{u}(s)) \Bigr ] \dd s  \nonumber \\
& = \Bigl \langle \xi_1 + \lambda h_{x_0}(\overline{x}_0,\overline{x}(T))^\top, a \Bigr \rangle + \Bigl \langle \xi_2 + \lambda h_x(\overline{x}_0,\overline{x}(T))^\top, a \Bigr \rangle \nonumber \\
&~~~ + \int_0^T \Biggl [ - p(s) + \int_s^T \Bigl [ \frac{f_x(r,s,\overline{x}(s),\overline{u}(s))^\top}{(r-s)^{1-\alpha}} + g_x(r,s,\overline{x}(s),\overline{u}(s))^\top \Bigr ] p(r) \dd r \Biggr ]^\top Z(s;a,u) \dd s \nonumber \\
&~~~ + \int_0^T \Bigl ( \lambda h_x(\overline{x}_0,\overline{x}(T)) + \xi_2^\top \Bigr ) \mathds{1}_{[0,T)}(s) \frac{f(T,s,\overline{x}(s),u(s)) - f(T,s,\overline{x}(s),\overline{u}(s))}{(T-s)^{1-\alpha}} \dd s \nonumber \\
&~~~ + \int_0^T \Bigl ( \lambda h_x(\overline{x}_0,\overline{x}(T)) + \xi_2^\top \Bigr ) \Bigl [ g(T,s,\overline{x}(s),u(s)) - g(T,s,\overline{x}(s),\overline{u}(s))  \Bigr ] \dd s \nonumber \\
&~~~ + \int_0^T \lambda  \Bigl [ l(s,\overline{x}(s),u(s)) - l(s,\overline{x}(s),\overline{u}(s)) \Bigr ] \dd s. \nonumber
\end{align}

In (\ref{eq_4_31_32342323423423423}), the standard Fubini's formula and Lemma \ref{Lemma_4_5} lead to
\begin{align*}
& \int_0^T \Biggl [ - p(s) + \int_s^T \Bigl [ \frac{f_x(r,s,\overline{x}(s),\overline{u}(s))^\top}{(r-s)^{1-\alpha}} + g_x(r,s,\overline{x}(s),\overline{u}(s))^\top \Bigr ] p(r) \dd r \Biggr ]^\top Z(s) \dd s \\
& = \int_0^T - p(s)^\top Z(s) \dd s	+ \int_0^T \int_0^s p(s)^\top \Bigl [ \frac{f_x(s,r,\overline{x}(r),\overline{u}(r)}{(s-r)^{1-\alpha}} + g_x(s,r,\overline{x}(r),\overline{u}(r)) \Bigr ] Z(r) \dd r  \dd s \\
& = \int_0^T - p(s)^\top \Biggl [ Z(s) - \int_0^s  \Bigl [ \frac{f_x(s,r,\overline{x}(r),\overline{u}(r)}{(s-r)^{1-\alpha}} + g_x(s,r,\overline{x}(r),\overline{u}(r)) \Bigr ]Z(r) \dd r \Biggr ] \dd s \\
& = \int_0^T - p(s)^\top \Biggl [a + \int_0^s \frac{f(s,r,\overline{x}(r),u(r)) - f(s,r,\overline{x}(r),\overline{u}(r))}{(s-r)^{1-\alpha}} \dd r \\
&~~~~~~~ + \int_0^s \bigl [ g(s,r,\overline{x}(r),u(r)) - g(s,r,\overline{x}(r),\overline{u}(r)) \bigr ]   \dd r \Biggr ] \dd s.
\end{align*}
%Hence, (\ref{eq_4_31_32342323423423423}) can be written as for any $(a,u) \in \mathbb{R}^n \times \mathcal{U}^p[0,T]$,
%\begin{align}	
%\label{eq_4_31}
%0 & \leq \Bigl \langle \xi_1 + \lambda h_{x_0}(\overline{x}_0,\overline{x}(T))^\top, a \Bigr \rangle + \Bigl \langle \xi_2 + \lambda h_x(\overline{x}_0,\overline{x}(T))^\top, a \Bigr \rangle - \Bigl \langle \int_0^T p(s)  \dd s , a \Bigr \rangle \\
%&~~~ + \int_0^T - p(s)^\top \Biggl [\int_0^s \frac{f(s,r,\overline{x}(r),u(r)) - f(s,r,\overline{x}(r),\overline{u}(r))}{(s-r)^{1-\alpha}} \dd r \nonumber \\
%&~~~~~~~ + \int_0^s \bigl [ g(s,r,\overline{x}(r),u(r)) - g(s,r,\overline{x}(r),\overline{u}(r)) \bigr ]   \dd r \Biggr ] \dd s \nonumber \\
%&~~~ + \int_0^T \Bigl ( \lambda h_x(\overline{x}_0,\overline{x}(T)) + \xi_2^\top \Bigr ) \mathds{1}_{[0,T)}(s) \frac{f(T,s,\overline{x}(s),u(s)) - f(T,s,\overline{x}(s),\overline{u}(s))}{(T-s)^{1-\alpha}} \dd s \nonumber \\
%&~~~ + \int_0^T \Bigl ( \lambda h_x(\overline{x}_0,\overline{x}(T)) + \xi_2^\top \Bigr ) \Bigl [ g(T,s,\overline{x}(s),u(s)) - g(T,s,\overline{x}(s),\overline{u}(s))  \Bigr ] \dd s \nonumber \\
%&~~~ + \int_0^T \lambda  \Bigl [ l(s,\overline{x}(s),u(s)) - l(s,\overline{x}(s),\overline{u}(s)) \Bigr ] \dd s. \nonumber
%\end{align}
Moreover, by definition of $N_F$ in (\ref{eq_4_1}), it follows that 
\begin{align*}
\Bigl \langle \xi_1, a \Bigr \rangle + \Bigl  \langle \xi_2, a \Bigr  \rangle & \leq \Bigl \langle \xi_1, \overline{x}_0 - y_1 + a \Bigr \rangle + \Bigl \langle \xi_2, \overline{x}(T;\overline{x}_0,\overline{u}) - y_2 + a \Bigr \rangle,~ \forall y = \begin{bmatrix}
 y_1 \\ 
 y_2
 \end{bmatrix} \in F.
\end{align*}
Hence, (\ref{eq_4_31_32342323423423423}) becomes for any $(a,u) \in \mathbb{R}^n \times \mathcal{U}^p[0,T]$ and $y \in F$, 
\begin{align}	
\label{eq_4_32}
0 & \leq \Bigl \langle \xi_1, \overline{x}_0 - y_1 + a \Bigr \rangle_{\mathbb{R}^n \times \mathbb{R}^n} + \Bigl \langle \xi_2, \overline{x}(T;\overline{x}_0,\overline{u}) - y_2 + a \Bigr \rangle_{\mathbb{R}^n \times \mathbb{R}^n} \\
&~~~ + \lambda h_{x_0}(\overline{x}_0,\overline{x}(T)) a + \lambda  h_x(\overline{x}_0,\overline{x}(T)) a - \Bigl \langle \int_0^T p(s)  \dd s , a \Bigr \rangle \nonumber \\
&~~~ + \int_0^T - p(s)^\top \Biggl [\int_0^s \frac{f(s,r,\overline{x}(r),u(r)) - f(s,r,\overline{x}(r),\overline{u}(r))}{(s-r)^{1-\alpha}} \dd r \nonumber \\
&~~~~~~~ + \int_0^s \bigl [ g(s,r,\overline{x}(r),u(r)) - g(s,r,\overline{x}(r),\overline{u}(r)) \bigr ]   \dd r \Biggr ] \dd s \nonumber \\
&~~~ + \int_0^T \Bigl ( \lambda h_x(\overline{x}_0,\overline{x}(T)) + \xi_2^\top \Bigr ) \mathds{1}_{[0,T)}(s) \frac{f(T,s,\overline{x}(s),u(s)) - f(T,s,\overline{x}(s),\overline{u}(s))}{(T-s)^{1-\alpha}} \dd s \nonumber \\
&~~~ + \int_0^T \Bigl ( \lambda h_x(\overline{x}_0,\overline{x}(T)) + \xi_2^\top \Bigr ) \Bigl [ g(T,s,\overline{x}(s),u(s)) - g(T,s,\overline{x}(s),\overline{u}(s))  \Bigr ] \dd s \nonumber \\
&~~~ + \int_0^T \lambda  \Bigl [ l(s,\overline{x}(s),u(s)) - l(s,\overline{x}(s),\overline{u}(s)) \Bigr ] \dd s. \nonumber
\end{align}
Below, we use (\ref{eq_4_32}) to prove the transversality condition, the nontriviality of the adjoint equation, and the Hamiltonian-like maximum condition in Theorem \ref{Theorem_3_1}.

\subsection{Proof of Theorem \ref{Theorem_3_1}: Transversality Condition and Nontriviality of Adjoint Equation}\label{Section_4_8}

In (\ref{eq_4_32}), when $u=\overline{u}$, we have
\begin{align}
\label{eq_5_456345345345}	
0 & \leq \Bigl \langle \xi_1, \overline{x}_0 - y_1 + a \Bigr \rangle + \Bigl \langle \xi_2, \overline{x}(T;\overline{x}_0,\overline{u}) - y_2 + a \Bigr \rangle \\
&~~~ + \lambda h_{x_0}(\overline{x}_0,\overline{x}(T)) a + \lambda h_x(\overline{x}_0,\overline{x}(T)) a - \Bigl \langle \int_0^T p(s)  \dd s , a \Bigr \rangle,~ \forall y = \begin{bmatrix}
 y_1 \\ 
 y_2
 \end{bmatrix} \in F. \nonumber
\end{align}
When $y_1 = \overline{x}_0$ and $y_2 = \overline{x}(T;\overline{x}_0,\overline{u})$, the above inequality holds for any $a,-a \in \mathbb{R}^n$, which implies
\begin{align}
\label{eq_4_38_1_2_1_2_3_2_1}
\int_0^T p(s) \dd s = \xi_1 + \xi_2 + \lambda h_{x_0}(\overline{x}_0,\overline{x}(T))^\top  +\lambda h_x(\overline{x}_0,\overline{x}(T))^\top.
\end{align}
Under this condition, (\ref{eq_5_456345345345}) becomes
\begin{align*}	
0 \leq \Bigl \langle \xi_1, \overline{x}_0 - y_1 \Bigr \rangle + \Bigl \langle \xi_2, \overline{x}(T;\overline{x}_0,\overline{u}) - y_2 \Bigr \rangle,~ \forall y  \in F.	
\end{align*}
This proves the transversality condition in Theorem \ref{Theorem_3_1}. In addition, as $p (\cdot) \in L^p([0,T];\mathbb{R}^n)$ by Lemma \ref{Lemma_B_5},  (\ref{eq_4_38_1_2_1_2_3_2_1}), together with the nontriviality condition, shows the nontriviality of the adjoint equation in Theorem \ref{Theorem_3_1}.

\subsection{Proof of Theorem \ref{Theorem_3_1}: Hamiltonian-like Maximum Condition}\label{Section_4_9}

We finally prove the Hamiltonian-like maximum condition in Theorem \ref{Theorem_3_1}. When $y_1 = \overline{x}_0$, $y_2 = \overline{x}(T;\overline{x}_0,\overline{u})$ and $a=0$ in (\ref{eq_4_32}), by the standard Fubini's formula, (\ref{eq_4_32}) can be written as
\begin{align}
\label{eq_4_33}
0 & \leq \int_0^T -p(s)^\top \Biggl [\int_0^s \frac{f(s,r,\overline{x}(r),u(r)) - f(s,r,\overline{x}(r),\overline{u}(r))}{(s-r)^{1-\alpha}} \dd r \\
&~~~~~~~ + \int_0^s \bigl [ g(s,r,\overline{x}(r),u(r)) - g(s,r,\overline{x}(r),\overline{u}(r)) \bigr ]   \dd r \Biggr ] \dd s  \nonumber \\
&~~~ + \int_0^T \Bigl ( \lambda h_x(\overline{x}_0,\overline{x}(T)) + \xi_2^\top \Bigr ) \mathds{1}_{[0,T)}(s) \frac{f(T,s,\overline{x}(s),u(s)) - f(T,s,\overline{x}(s),\overline{u}(s))}{(T-s)^{1-\alpha}} \dd s  \nonumber \\
&~~~ + \int_0^T \Bigl ( \lambda h_x(\overline{x}_0,\overline{x}(T)) + \xi_2^\top \Bigr ) \Bigl [ g(T,s,\overline{x}(s),u(s)) - g(T,s,\overline{x}(s),\overline{u}(s))  \Bigr ] \dd s  \nonumber \\
&~~~ + \int_0^T \lambda  \Bigl [ l(s,\overline{x}(s),u(s)) - l(s,\overline{x}(s),\overline{u}(s)) \Bigr ] \dd s	 \nonumber \\
& = \int_0^T \Biggl [ \int_s^T - p(r)^\top   \frac{f(r,s,\overline{x}(s),u(s)) - f(r,s,\overline{x}(s),\overline{u}(s))}{(r-s)^{1-\alpha}} \dd r  \nonumber \\
&~~~~~~~~~~ + \int_s^T - p(r)^\top \bigl [ g(r,s,\overline{x}(s),u(s)) - g(r,s,\overline{x}(s),\overline{u}(s)) \bigr ] \dd r    \nonumber \\
&~~~~~~~~~~ + \Bigl ( \lambda h_x(\overline{x}_0,\overline{x}(T)) + \xi_2^\top \Bigr ) \mathds{1}_{[0,T)}(s) \frac{f(T,s,\overline{x}(s),u(s)) - f(T,s,\overline{x}(s),\overline{u}(s))}{(T-s)^{1-\alpha}}  \nonumber \\
&~~~~~~~~~~ + \Bigl ( \lambda h_x(\overline{x}_0,\overline{x}(T)) + \xi_2^\top \Bigr ) \bigl [ g(T,s,\overline{x}(s),u(s)) - g(T,s,\overline{x}(s),\overline{u}(s))  \bigr ]  \nonumber \\
&~~~~~~~~~~ + \lambda  \bigl [ l(s,\overline{x}(s),u(s)) - l(s,\overline{x}(s),\overline{u}(s)) \bigr ]  \Biggr ] \dd s. \nonumber
\end{align}

Let us define for $s \in [0,T]$,
\begin{align*}
& \Lambda(s,\overline{x}(s),u) := \int_s^T p(r)^\top \frac{f(r,s,\overline{x}(s),u)}{(s-r)^{1-\alpha}} \dd r - \mathds{1}_{[0,T)}(s)  \Bigl ( \lambda h_x(\overline{x}_0,\overline{x}(T)) + \xi_2^\top \Bigr ) \frac{f(T,s,\overline{x}(s),u)}{(T-s)^{1-\alpha}} \\
&~~~~~ + \int_s^T p(r)^\top g(r,s,\overline{x}(s),u) \dd r - \Bigl ( \lambda h_x(\overline{x}_0,\overline{x}(T)) + \xi_2^\top \Bigr )  g(T,s,\overline{x}(s),u) - \lambda l(s,\overline{x}(s),u).
\end{align*}
Then we observe that (\ref{eq_4_33}) becomes
\begin{align*}
\int_0^T \Lambda(s,\overline{x}(s),u(s)) \dd s \leq 	\int_0^T \Lambda(s,\overline{x}(s),\overline{u}(s)) \dd s.
\end{align*}

As $U$ is separable, there exists a countable dense set $U_i = \{u_i,~ i \geq 1\} \subset U$. Moreover, there exists a measurable set $S_i \subset [0,T]$ such that $|S_i| = T$ and any $t \in S_i$ is the Lebesgue point of $\Lambda(t,\overline{x}(t),u(t))$, i.e., $\lim_{\tau \downarrow 0} \frac{1}{2\tau} \int_{t-\tau}^{t+\tau} \Lambda(s,\overline{x}(s),u(s)) \dd s = \Lambda(t,\overline{x}(t),u(t))$ \cite[Theorem 5.6.2]{Bogachev_book}. We fix $u_i \in U_i$. For any $ t \in S_i$, define
\begin{align*}
u(s) : = \begin{cases}
 	\overline{u}(s), & s \in [0,T] \setminus (t-\tau,t+\tau),\\
 	u_i, & s \in (t-\tau,t+\tau).
 \end{cases}	
\end{align*}
It then follows that
\begin{align*}
0 \leq \lim_{\tau \downarrow 0} \frac{1}{2\tau} \int_{t - \tau}^{t + \tau} \bigl [ 	\Lambda(s,\overline{x}(s),\overline{u}(s)) - \Lambda(s,\overline{x}(s),u_i) \bigr ] \dd s  = \Lambda  (t,\overline{x}(t),\overline{u}(t)) - \Lambda  (t,\overline{x}(t),u_i).
\end{align*}
Since $\cap_{i \geq 1} S_i = [0,T]$, $\Lambda$ is continuous in $u \in U$, and $U$ is separable, we must have
\begin{align*}
\Lambda(t,\overline{x}(t),u) \leq \Lambda(t,\overline{x}(t),\overline{u}(t)),~ \forall u \in U,~  \textrm{a.e.}~ t \in [0,T],
\end{align*}
which proves the Hamiltonian-like maximum condition in Theorem \ref{Theorem_3_1}. This is the end of the proof for Theorem \ref{Theorem_3_1}.

%\section{Existence of Optimal Control} \label{Section_5}
%
%We now state existence of optimal controls for \textbf{(P)}. 

\begin{appendices}
	
\section*{Appendices}

%In Appendices \ref{Appendix_A}-\ref{Appendix_D}, 
We provide some preliminary results, and obtain the well-posedness and estimates of general Volterra integral equations having singular and nonsingular kernels. To simplify the notation, we use $\|\cdot\|_{q} := \|\cdot\|_{L^q([0,T];\mathbb{R}^n)}$ and $\|\cdot\|_{p} := \|\cdot\|_{L^p([0,T];\mathbb{R}^n)}$.

%, and some important lemmas to prove the main results of this paper. 

%We sometimes use the notation $\|\psi\|_{p} := \|\psi(\cdot)\|_{p}$ when there is no confusion.

\section{Preliminaries}\label{Appendix_A}

%\begin{lemma}[Young's Inequality: Theorem 3.9.4 of \cite{Bogachev_book}]\label{Lemma_A_1}
%Assume that $p,q,r \geq 1$ with $\frac{1}{q} + 1 = \frac{1}{p} + \frac{1}{r}$. Then for $f(\cdot) \in L^p(\mathbb{R}^n;\mathbb{R})$ and $g(\cdot) \in L^r(\mathbb{R}^n;\mathbb{R})$, $\| (f * g)(\cdot)\|_{L^q(\mathbb{R}^n;\mathbb{R})} \leq 	\| f(\cdot)\|_{L^p(\mathbb{R}^n;\mathbb{R})} \| g(\cdot)\|_{L^r(\mathbb{R}^n;\mathbb{R})}$.
%\end{lemma}

\begin{lemma}
%[Corollary 2.2 of \cite{Lin_Yong_SICON_2020}]
\label{Lemma_A_2}
Let $\alpha \in (0,1)$. Suppose that $\frac{1}{q} + 1 = \frac{1}{p} + \frac{1}{r}$ with $p,q \geq 1$ and $ r \in [1, \frac{1}{1-\alpha})$. Then for any $a<b$, $ \tau \in (0,b-a]$, and $\psi(\cdot) \in L^p([a,b];\mathbb{R})$,
\begin{align*}
\Bigl ( \int_{a}^{a+\tau} \Bigl | \int_a^{t} \frac{\psi(s)}{(t-s)^{1-\alpha}} \dd s \Big |^q \dd t \Bigr )^{\frac{1}{q}} &\leq \Bigl ( \frac{\tau^{1-r(1-\alpha)}}{1-r(1-\alpha)} \Bigr)^{\frac{1}{r}} \|\psi(\cdot) \|_{L^p([a,b];\mathbb{R})}, \\
\Bigl ( \int_a^{a+\tau} \Bigl | \int_a^t \psi(s) \dd s \Bigr |^q \dd t \Bigr )^{\frac{1}{q}} & \leq  \tau^{\frac{1}{r}} \|\psi(\cdot) \|_{L^p([a,b];\mathbb{R})}.
\end{align*}
\end{lemma}
\begin{proof}
Let $\zeta_{\tau}(t) := \frac{1}{t^{1-\alpha}}	\mathds{1}_{(0,\tau]}(t)$. Note that $\mathds{1}_{(0,\tau]}(t-s) = 1$ for $t-s \in (0,\tau]$, otherwise $\mathds{1}_{(0,\tau]}(t-s) = 0$. It follows that (note that $t \vee s := \max \{t,s\}$ for $t,s \in [0,T]$)
\begin{align*}
(\psi(\cdot)\mathds{1}_{[a,b]} * \zeta_{\tau}(\cdot) )(t) & = \int_{a}^{b} \frac{\psi(s)}{(t-s)^{1-\alpha}} \mathds{1}_{(0,\tau]}(t-s) \dd s 
= \begin{cases}
\int_{a \vee t-\tau}^{t} \frac{\psi(s)}{(t-s)^{1-\alpha}} \dd s, & t \in [a,b],  \\
0 & t \notin [a,b].  
 \end{cases}
\end{align*}
This leads to
\begin{align*}
& (\psi(\cdot)\mathds{1}_{[a,b]} * \zeta_{\tau}(\cdot) )(t) = \begin{cases}
\int_{a}^{t} \frac{\psi(s)}{(t-s)^{1-\alpha}} \dd s, & t \in [a,a+\tau],  \\
0 & t \notin [a,a+\tau].  
 \end{cases}
\end{align*}
Hence, by Young's Inequality (see \cite[Theorem 3.9.4]{Bogachev_book}), for $\frac{1}{q} + 1 = \frac{1}{p} + \frac{1}{r}$ with $r \in [1,\frac{1}{1-\alpha})$,
\begin{align*}
%\Bigl \|\psi(\cdot)\mathds{1}_{[a,b]} * \frac{1}{t^{1-\alpha}}	\mathds{1}_{(0,\tau]} \Bigr \|_{L^q([a,a+\tau];\mathbb{R})}  
%& = 
\Bigl ( \int_a^{a+\tau} \Bigl | \int_a^t \frac{\psi(s)}{(t-s)^{1-\alpha}} \dd s \Bigr|^q \dd t \Bigr )^{\frac{1}{q}} 
& \leq \|\psi(\cdot)\|_{L^p([a,b];\mathbb{R})} \|\zeta_{\tau}(\cdot) \|_{L^r([0,\tau];\mathbb{R})}.
\end{align*}
Note that as $r(1-\alpha) \in [1-\alpha,1)$ with $1-\alpha > 0$,
\begin{align*}
	\|\zeta_{\tau}(\cdot) \|_{L^r([0,\tau];\mathbb{R})}^r = \int_0^{\tau}  t^{-r(1-\alpha)} \dd t = \frac{1}{1 - r(1-\alpha)} \tau^{1 - r(1-\alpha)}.
\end{align*}
This proves the first inequality. The second inequality can be shown in a similar way by letting $\zeta_{\tau}(t) :=	\mathds{1}_{(0,\tau]}(t)$. We complete the proof.
\end{proof}

\begin{lemma}[Lemma 2.3 of \cite{Lin_Yong_SICON_2020}]\label{Lemma_A_3}
Suppose that $\alpha \in (0,1)$ and $p \geq 1$. Assume that $\psi :\Delta \rightarrow \mathbb{R}^n$ is measurable with $\psi(0,\cdot) \in L^p([0,T];\mathbb{R}^n)$ satisfying $|\psi(t,s) - \psi(t^\prime,s)|_{\mathbb{R}^n} \leq \omega(|t-t^\prime|) \psi^\prime(s)$ for $t,t^\prime \in [0,T]$ and $s \in [0,T]$, where $\psi^\prime (\cdot) \in L^p([0,T];\mathbb{R})$ and $\omega$ is some modulus of continuity. Let
\begin{align*}
\varphi(t) &:= \int_0^t \frac{\psi(t,s)}{(t-s)^{1-\alpha}} \dd s,~ \textrm{a.e.}~ t \in [0,T]. 
\end{align*}
Then $\varphi(\cdot) \in L^p([0,T];\mathbb{R}^n)$ and $\|\varphi(\cdot)\|_{p} \leq \frac{T^\alpha}{\alpha} \Bigl ( \|\psi(0,\cdot)\|_{p} + \omega(T) \|\psi^\prime (\cdot) \|_{p} \Bigr )$.
%\begin{align*}	
%\varphi \in L^q([0,T];\mathbb{R}^n)~\textrm{and}~\|\varphi(\cdot)\|_{q} \leq \frac{T^\alpha}{\alpha} \Bigl ( \|\psi(0,\cdot)\|_{q} + \omega(T) \|\psi^\prime\|_{q} \Bigr ).
%\end{align*}
Furthermore, if $p > \frac{1}{\alpha}$, then  $\varphi$ is continuous on $[0,T]$ and there is a constant $C$, independent from choice of $\varphi$, such that 
\begin{align*}
|\varphi(t)|_{\mathbb{R}^n} \leq C \Bigl ( 	\|\psi(0,\cdot)\|_{p} + \|\psi^\prime (\cdot) \|_{p} \Bigr ),~ \forall t \in [0,T].
\end{align*}
\end{lemma}

\begin{lemma}\label{Lemma_A_4}
Suppose that $\alpha \in (0,1)$ and $q \geq 1$. Assume that $\psi :\Delta \rightarrow \mathbb{R}^n$ is measurable with $\psi(0,\cdot) \in L^p([0,T];\mathbb{R}^n)$ satisfying $|\psi(t,s) - \psi(t^\prime,s)|_{\mathbb{R}^n} \leq \omega(|t-t^\prime|) \psi^\prime(s)$ for $t,t^\prime \in [0,T]$ and $s \in [0,T]$, where $\psi^\prime (\cdot)  \in L^p([0,T];\mathbb{R})$ and $\omega$ is some modulus of continuity. Let
\begin{align*}
\hat{\varphi}(t) &:= \int_0^t \psi(t,s) \dd s,~ \textrm{a.e.}~ t \in [0,T].
\end{align*}
Then $\hat{\varphi} (\cdot) \in L^p([0,T];\mathbb{R}^n)$ and $\|\hat{\varphi} (\cdot) \|_{p} \leq T \Bigl ( \|\psi(0,\cdot)\|_{p} + \omega(T) \|\psi^\prime (\cdot) \|_{p} \Bigr )$.
%\begin{align*}
%\varphi^\prime \in L^q([0,T];\mathbb{R}^n)~\textrm{and}~	\|\hat{\varphi}\|_{q} \leq T \Bigl ( \|\psi(0,\cdot)\|_{q} + \omega(T) \|\psi^\prime\|_{q} \Bigr ).
%\end{align*}
Furthermore, if $p > \frac{1}{\alpha}$, then  $\hat{\varphi}$ is continuous on $[0,T]$ and there is a constant $C$, independent from choice of $\hat{\varphi}$, such that 
\begin{align*}
|\hat{\varphi}(t)|_{\mathbb{R}^n} \leq C \Bigl ( 	\|\psi(0,\cdot)\|_{p} + \|\psi^\prime (\cdot) \|_{p} \Bigr ),~ \forall t \in [0,T].	
\end{align*}
\end{lemma}
\begin{proof}
The proof is analogous to that for Lemma \ref{Lemma_A_3}. Indeed, to prove the continuity, note that
\begin{align*}
|\psi(t,s)|_{\mathbb{R}^n} \leq |\psi(0,s)|_{\mathbb{R}^n} + \omega(t)\psi^\prime(s) =: \overline{\psi}(s),~ (t,s) \in \Delta,
\end{align*}
where $\overline{\psi} \in L^p([0,T];\mathbb{R})$, since $\psi(0,\cdot) \in L^p([0,T];\mathbb{R}^n)$ and $\psi^\prime \in L^p([0,T];\mathbb{R})$. It holds that
\begin{align*}
& |\hat{\varphi}(t) - \hat{\varphi}(t^\prime) |_{\mathbb{R}^n}	
%\\
%& = \Bigl | \int_0^t \psi(t,s) \dd s - \int_0^{t^\prime} \psi(t^\prime,s) \dd s \Bigr | \\
%& \leq \int_0^{t-\tau} |\hat{\varphi} (t,s) - \hat{\varphi}(t^\prime,s) | \dd s + \int_{t-\tau}^t \overline{\psi}(s) \dd s + \int_{t-\tau}^{t^\prime} \overline{\psi}(s) \dd s \\
%& \leq \omega(|t-t^\prime|) \int_0^{t-\tau} \psi^\prime(s) \dd s \\
%& 
\leq \omega(|t-t^\prime|) T^{\frac{q-1}{q}} \|\psi^\prime  (\cdot) \|_{p} + \|\overline{\psi} (\cdot)\|_{p} \tau^{\frac{p-1}{p}} + \|\overline{\psi} (\cdot) \|_{p}(t^\prime-t+\tau)^{\frac{p-1}{p}}.
\end{align*}
Then the rest of the proof is similar to that of Lemma \ref{Lemma_A_3}; thus completing the proof.
\end{proof}

\begin{lemma}[Gronwall-type inequality with the presence of singular and nonsingular kernels]
\label{Lemma_A_5}
Assume that $\alpha \in (0,1)$ and $p > \frac{1}{\alpha}$.
%$p \geq \frac{1}{\alpha}$. 
Let $P(\cdot) \in L^p([0,T];\mathbb{R})$ and $b(\cdot),z(\cdot) \in L^{\frac{p}{p-1}}([0,T];\mathbb{R})$, where $b$, $z$, and $P$ are nonnegative functions. Suppose that the following holds:
\begin{align}
\label{eq_a_1}
z(t) \leq b(t) + \int_0^t  \frac{P(s)z(s)}{(t-s)^{1-\alpha}} \dd s + \int_0^t 	P(s) z(s) \dd s,~ \textrm{a.e.}~ t \in [0,T].
\end{align}
Then there exists a constant $C \geq 0$ such that
\begin{align*}
z(t) \leq b(t) + C \int_0^t \frac{P(s)b(s)}{(t-s)^{1-\alpha}} \dd s + C \int_0^t P(s) b(s) \dd s,~ \textrm{a.e.}~ t \in [0,T].	
\end{align*}
\end{lemma}

\begin{proof}
By the H\"older's inequality, we have $P(\cdot)z(\cdot) \in L^1([0,T];\mathbb{R})$, which implies that the two integrals on the right-hand side of (\ref{eq_a_1}) are well-defined in the $L^1$ sense. Below, there are several generic constants, whose values vary from line to line.

We remove the singularity of the right-hand side of (\ref{eq_a_1}). As (\ref{eq_a_1}) is linear in $z$, consider,
\begin{align*}
	z(t) \leq b(t) + \int_0^t \frac{\hat{P}(s;\alpha) z(s)}{(t-s)^{1-\alpha}} \dd s,~ \textrm{a.e.}~ t \in [0,T],
\end{align*}
where $\hat{P}(s;\alpha) := P(s)  + (t-s)^{1-\alpha} P(s)$ with $(t,s) \in \Delta$. Note that $\hat{P}(\cdot;\alpha) \in L^p([0,T];\mathbb{R})$ is nonnegative, and since $1-\alpha > 0$, we have $\|\hat{P}(\cdot;\alpha)\|_{p} \leq  \|P(\cdot)\|_{p} + T^{1-\alpha} \|P(\cdot)\|_{p} \leq C \|P(\cdot)\|_{p}$.

%We first consider the case when $p > \frac{1}{\alpha}$. 
It follows that 
\begin{align}
\label{eq_A_2_23234234}
z(t) 
& \leq b(t) + \int_0^t \frac{P(s) b(s)}{(t-s)^{1-\alpha}} \dd s + \int_0^t P(s) b(s) \dd s + \int_0^t 	\frac{\hat{P}(s;\alpha)}{(t-s)^{1-\alpha}} \int_0^s \frac{\hat{P}(\tau;\alpha)}{(s-\tau)^{1-\alpha}} z(\tau) \dd \tau \dd s,
\end{align}
Notice that the double integral above represents the integration over the triangular region with base and height of $t$, where the integration is performed vertically and then horizontally with respect to $\tau$ and $s$, respectively. Alternatively, we may reverse the order of the double integration above, i.e.,
\begin{align}
\label{eq_a_3_345345}
 \int_0^t \frac{\hat{P}(s;\alpha) }{(t-s)^{1-\alpha}} \int_0^s \frac{\hat{P}(\tau;\alpha)}{(s-\tau)^{1-\alpha}} z(\tau) \dd \tau \dd s 
& = \int_0^t \int_{s}^t \frac{\hat{P}(\tau;\alpha) \hat{P}(s;\alpha)}{(t-\tau)^{1-\alpha} (\tau-s)^{1-\alpha}}   z(s) \dd \tau \dd s.
%,~ \textrm{a.e.}~ t \in [0,T]. \nonumber
\end{align}

Let $v := \frac{\tau-s}{t-s}$. Note that $\tau$ varies from $s$ to $t$, which implies $v$ varies from $0$ to $1$. Moreover, $\tau = s+ (t-s) v $ and $\dd \tau = (t-s) \dd v$.  Then using the H\"older's inequality and changing the integration variable, the integration in (\ref{eq_a_3_345345}) can be evaluated by
\begin{align*}	
& \int_0^t \int_{s}^t \frac{\hat{P}(\tau;\alpha) \hat{P}(s;\alpha)}{(t-\tau)^{1-\alpha} (\tau-s)^{1-\alpha}}  \dd \tau z(s) \dd s \\
%& \leq C \Biggl ( \int_0^t \int_s^t \frac{P(s)P(\tau)}{(t-\tau)^{1-\alpha} (\tau-s)^{1-\alpha}} \dd \tau z(s) \dd s + \int_0^t \int_s^t P(s)P(\tau) \dd \tau z(s) \dd s \Biggr )  \\
& \leq C \Biggl ( \|P(\cdot)\|_{p} \int_0^t P(s) \Bigl ( \int_s^t \frac{1}{(t-\tau)^{ \frac{(1-\alpha)p}{p-1}} (\tau-s)^{ \frac{(1-\alpha)p}{p-1}}} \dd \tau \Bigr)^{\frac{p-1}{p}} z(s) \dd s + \|P(\cdot)\|_{p} T^{\frac{p-1}{p}} \int_0^t P(s) z(s) \dd s \Biggr ) \\
%& \leq C \Biggl ( \|P(\cdot)\|_{p} \int_0^t P(s) \Bigl ( \int_0^1 \frac{(t-s)}{(t-s-(t-s)v)^{(1-\alpha) \frac{p}{p-1}} ((t-s)v)^{(1-\alpha) \frac{p}{p-1}}}  \\
%&~~~~~~~~~~ \times \frac{1}{\frac{(t-s)^{(1-\alpha) \frac{p}{p-1}}}{(t-s)^{(1-\alpha) \frac{p}{p-1}}}} \frac{1}{\frac{(t-s)^{(1-\alpha) \frac{p}{p-1}}}{(t-s)^{(1-\alpha) \frac{p}{p-1}}}} \dd v \Bigr)^{\frac{p-1}{p}} z(s) \dd s + \|P(\cdot)\|_{p} T^{\frac{p-1}{p}} \int_0^t P(s) z(s) \dd s \Biggr ) \\
& = C^{(1)} \Biggl ( \|P(\cdot)\|_{p} \int_0^t \frac{P(s) z(s) }{(t-s)^{2(1-\alpha) - \frac{p-1}{p}}}
\Bigl (\int_0^1 \frac{1}{(1-v)^{(1-\alpha) \frac{p}{p-1}} v^{(1-\alpha) \frac{p}{p-1}}  }  \dd v \Bigr )^{\frac{p-1}{p}} \dd s \\
&~~~~~~~~~~ + \|P(\cdot)\|_{p} T^{\frac{p-1}{p}} \int_0^t P(s) z(s) \dd s \Biggr ).
\end{align*}
Let $\alpha[\alpha]  := 1 - (1-\alpha) \frac{p}{p-1} = \frac{\alpha p - 1}{p-1} \in (0,1)$ and $\alpha^{(1)} := 1 - \Bigl ( 2(1-\alpha) - \frac{p-1}{p} \Bigr )  = 2 \alpha - \frac{1}{p} = \alpha + \Bigl (\alpha - \frac{1}{p} \Bigr ) > \alpha$ (note that $p > \frac{1}{\alpha}$). We can show that
\begin{align*}
& C^{(1)} \|P(\cdot)\|_{p} \int_0^t \frac{P(s) z(s) }{(t-s)^{2(1-\alpha) - \frac{p-1}{p}}}
\Bigl (\int_0^1 \frac{1}{(1-v)^{1-\alpha[\alpha]} v^{1-\alpha[\alpha]}  }  \dd v \Bigr )^{\frac{p-1}{p}} \dd s	\\
& = C^{(1)} \|P(\cdot)\|_{p} B(\alpha[\alpha], \alpha[\alpha])^{\frac{p-1}{p}}  \int_0^t \frac{P(s) z(s)}{(t-s)^{1-\alpha^{(1)}}}  \dd s =: \bar{c}^{(1)} \int_0^t \frac{P(s) z(s)}{(t-s)^{1-\alpha^{(1)}}}  \dd s,
\end{align*}
where $B$ is the beta function defined by $B(x,y) := \int_0^1  t^{x-1} (1-t)^{y-1} \dd t$ for $x,y > 0$. We also have $C^{(1)} \|P(\cdot)\|_{p} T^{\frac{p-1}{p}}	 \int_0^t P(s) z(s) \dd s =: \hat{c}^{(1)} \int_0^t P(s) z(s) \dd s$. Then with $c^{(1)} := \max \{ \bar{c}^{(1)}, \hat{c}^{(1)} \}$, (\ref{eq_a_3_345345}) is bounded above by
\begin{align*}	
\int_0^t \frac{\hat{P}(s;\alpha) }{(t-s)^{1-\alpha}} \int_0^s \frac{\hat{P}(\tau;\alpha)}{(s-\tau)^{1-\alpha}} z(\tau) \dd \tau \dd s  \leq c^{(1)} \Bigl (  \int_0^t \frac{P(s) z(s)}{(t-s)^{1-\alpha^{(1)}}}  \dd s + \int_0^t P(s) z(s) \dd s \Bigr ).
\end{align*}
Hence, by letting $c^{(0)} := 1$  and $\alpha^{(0)} := \alpha$, together with (\ref{eq_a_1}), (\ref{eq_A_2_23234234}) can be evaluated by
\begin{align}
\label{eq_a_2}	
z(t) 
& \leq b(t) + \int_0^t \frac{P(s) b(s)}{(t-s)^{1-\alpha}} \dd s + \int_0^t P(s) b(s) \dd s + \int_0^t 	\frac{\hat{P}(t,s)}{(t-s)^{1-\alpha}} \int_0^s \frac{\hat{P}(s,\tau)}{(s-\tau)^{1-\alpha}} z(\tau) \dd \tau \dd s \\
& \leq b(t) + \sum_{i=0}^{1} c^{(i)} \int_0^t \frac{P(s)b(s)}{(t-s)^{1-\alpha^{(i)}}} \dd s + \sum_{i=0}^{1} c^{(i)} \int_0^t P(s) b(s) \dd s  \nonumber \\
&~~~  + c^{(1)} \int_0^t \frac{\hat{P}(s;\alpha^{(1)})}{(t-s)^{1-\alpha^{(1)}}} \int_0^s \frac{\hat{P}(\tau;\alpha)}{(t-s)^{1-\alpha}} z(\tau) \dd \tau  \dd s,~ \textrm{a.e.}~ t \in [0,T]. \nonumber
\end{align}

Note that by using the same technique as above, we can show that
\begin{align*}
& c^{(1)} \int_0^t \frac{\hat{P}(s;\alpha^{(1)})}{(t-s)^{\alpha^{(1)}}} \int_0^s \frac{\hat{P}(\tau;\alpha)}{(s-\tau)^{1-\alpha}} z(\tau) \dd \tau  \dd s  = c^{(1)} \int_0^t \int_s^t \frac{\hat{P}(\tau;\alpha^{(1)})}{(t-\tau)^{1-\alpha^{(1)}}} \frac{\hat{P}(s;\alpha)}{(\tau-s)^{1-\alpha}} z(s) \dd \tau  \dd s \\
%& \leq C c^{(1)} \Biggl ( \int_0^t \int_s^t	 \frac{P(s)P(\tau)}{(t-\tau)^{1-\alpha^{(1)}}(\tau-s)^{1-\alpha}} \dd \tau z(s) \dd s + \int_0^t \int_s^t P(s)P(\tau) \dd \tau z(s) \dd s \Biggr ) \\
%& \leq C c^{(1)} \Biggl ( \|P\|_{p} \int_0^t P(s) \Bigl ( \int_s^t \frac{1}{(t-\tau)^{(1-\alpha^{(1)})\frac{p}{p-1}} (\tau-s)^{(1-\alpha) \frac{p}{p-1} } } \dd \tau \Bigr)^{\frac{p-1}{p}} z(s) \dd s \\
%&~~~~~~~~~~ + \|P\|_{p} T^{\frac{p-1}{p}} \int_0^t P(s) z(s) \dd s \Biggr ) \\
& \leq C^{(2)} c^{(1)} \Biggl ( \|P(\cdot)\|_{p} \int_0^t  \frac{P(s) z(s)}{(t-s)^{2-\alpha-\alpha^{(1)}  - \frac{p-1}{p} }} \Bigl ( \int_0^1 \frac{1}{(1-v)^{(1-\alpha^{(1)}) \frac{p}{p-1} } v^{ (1-\alpha) \frac{p}{p-1}  }} \dd v \Bigr )^{\frac{p-1}{p}} \dd s  \\
&~~~~~~~~~~ + \|P(\cdot)\|_{p} T^{\frac{p-1}{p}} \int_0^t P(s) z(s) \dd s \Biggr ).
\end{align*}
Let $\alpha[\alpha^{(1)}] := 1 - (1-\alpha^{(1)}) \frac{p}{p-1} = \frac{\alpha^{(1)} p - 1}{p-1} \in (0,1)$ and $\alpha^{(2)} := 1- \Bigl (2-\alpha-\alpha^{(1)} - \frac{p-1}{p} \Bigr )  = \alpha + 2 \Bigl ( \alpha - \frac{1}{p} \Bigr ) > \alpha^{(1)} > \alpha$. We can show that
\begin{align*}
& C^{(2)} c^{(1)} \|P(\cdot)\|_{p} \int_0^t  \frac{P(s) z(s)}{(t-s)^{2-\alpha-\alpha^{(1)}  - \frac{p-1}{p} }} \Bigl ( \int_0^1 \frac{1}{(1-v)^{1-\alpha[\alpha^{(1)}]} v^{1- \alpha[\alpha]}} \dd v \Bigr )^{\frac{p-1}{p}} \dd s  	 \\
& = C^{(2)} c^{(1)} \|P(\cdot)\|_{p} B(\alpha[\alpha],\alpha[\alpha^{(1)}])^{\frac{p-1}{p}} \int_0^t \frac{P(s)z(s)}{(t-s)^{1-\alpha^{(2)}}} \dd s =: \bar{c}^{(2)} \int_0^t \frac{P(s)z(s)}{(t-s)^{1-\alpha^{(2)}}} \dd s,
\end{align*}
and we have $C^{(2)} c^{(1)} \|P(\cdot)\|_{p} T^{\frac{p-1}{p}} \int_0^t P(s) z(s) \dd s =: \hat{c}^{(2)} \int_0^t P(s) z(s) \dd s$.
%\begin{align*}	
%C^{(2)} c^{(1)} \|P(\cdot)\|_{p} T^{\frac{p-1}{p}} \int_0^t P(s) z(s) \dd s =: \hat{c}^{(2)} \int_0^t P(s) z(s) \dd s.
%\end{align*}
Let $c^{(2)} := \max \{ \bar{c}^{(2)}, \hat{c}^{(2)} \}$. Then using a similar approach, (\ref{eq_a_2}) can be evaluated by
\begin{align}
\label{eq_a_5_3423234234}
z(t) 
%& \leq b(t) + \int_0^t \frac{P(s) b(s)}{(t-s)^{1-\alpha}} \dd s + \int_0^t P(s) b(s) \dd s + \int_0^t 	\frac{\hat{P}(t,s)}{(t-s)^{1-\alpha}} \int_0^s \frac{\hat{P}(s,\tau)}{(s-\tau)^{1-\alpha}} z(\tau) \dd \tau \dd s  \\
%& \leq b(t) + \int_0^t \frac{P(s) b(s)}{(t-s)^{1-\alpha}} \dd s + \int_0^t P(s) b(s) \dd s + c^{(1)} \int_0^t \frac{P(s) z(s)}{(t-s)^{1-\alpha^{(1)}}}  \dd s + c^{(1)} \int_0^t P(s) z(s) \dd s \nonumber \\
& \leq b(t) + \sum_{i=0}^{1} c^{(i)} \int_0^t \frac{P(s)b(s)}{(t-s)^{1-\alpha^{(i)}}} \dd s + \sum_{i=0}^{1} c^{(i)} \int_0^t P(s) b(s) \dd s     \\
&~~~ + c^{(2)} \int_0^t \frac{P(s)z(s)}{(t-s)^{1-\alpha^{(2)}}}\dd s + c^{(2)}  \int_0^t P(s) z(s) \dd s,~ \textrm{a.e.}~ t \in [0,T]. \nonumber
\end{align}
Proceeding similarly, using (\ref{eq_a_1}), (\ref{eq_a_5_3423234234}) is evaluated by
\begin{align*}
z(t) & \leq b(t) + \sum_{i=0}^{2} c^{(i)} \int_0^t \frac{P(s)b(s)}{(t-s)^{1-\alpha^{(i)}}} \dd s + \sum_{i=0}^{2} c^{(i)} \int_0^t P(s) b(s) \dd s   \\
&~~~ + c^{(3)} \int_0^t \frac{P(s)z(s)}{(t-s)^{1-\alpha^{(3)}}}\dd s + c^{(3)}  \int_0^t P(s) z(s) \dd s,~ \textrm{a.e.}~ t \in [0,T],
\end{align*}
where $\alpha[\alpha^{(2)}] := 1 - \frac{(1- \alpha^{(2)} )p}{p-1}  = \frac{\alpha^{(2)} p - 1}{p-1} \in (0,1)$, $\alpha^{(3)} := \alpha + 3 \Bigl ( \alpha - \frac{1}{p} \Bigr ) > \alpha$, $c^{(3)} := \max \{ \bar{c}^{(3)}, \hat{c}^{(3)} \}$, $\bar{c}^{(3)} := C^{(3)} c^{(2)} \|P(\cdot)\|_{p} B (\alpha[\alpha], \alpha[\alpha^{(2)}])$, and $\hat{c}^{(3)} := C^{(3)} c^{(2)} \|P(\cdot)\|_{p} T^{\frac{p-1}{p}}$.

Therefore, by induction, we are able to get
\begin{align*}
z(t) & \leq b(t) + \sum_{i=0}^{k-1} c^{(i)} \int_0^t \frac{P(s)b(s)}{(t-s)^{1-\alpha^{(i)}}} \dd s + \sum_{i=0}^{k-1} c^{(i)} \int_0^t P(s) b(s) \dd s   \\
&~~~ + c^{(k)} \int_0^t \frac{P(s)z(s)}{(t-s)^{1-\alpha^{(k)}}}\dd s + c^{(k)}  \int_0^t P(s) z(s) \dd s,~ \textrm{a.e.}~ t \in [0,T],
\end{align*}
where $c^{(0)} = 1$, $\alpha^{(0)} = \alpha$, and for $i=1,\ldots,k$,
\begin{align*}
\alpha[\alpha] & := \frac{\alpha p - 1}{p-1}  \in (0,1),~ \alpha^{(i)} := 	 \alpha + i \Bigl (\alpha - \frac{1}{p} \Bigr ) > \alpha,~ c^{(i)} := \max \{ \bar{c}^{(i)}, \hat{c}^{(i)} \},	 \\
\bar{c}^{(i)} &:= C^{(i)} c^{(i-1)} \|P(\cdot)\|_{p} B (\alpha[\alpha], \alpha[\alpha^{(i-1)}]),~ \hat{c}^{(i)} := C^{(i)} c^{(i-1)} \|P(\cdot)\|_{p} T^{\frac{p-1}{p}}.
\end{align*}

We observe that there is $k^\prime \geq 1$ such that $\alpha^{(k)} \geq 1$ for any $k \geq k^\prime$. Hence, with a fixed $k \geq k^\prime$, using the H\"older's inequality, it follows that
\begin{align*}
z(t) & \leq b(t) + \sum_{i=0}^{k-1} c^{(i)} \int_0^t \frac{P(s)b(s)}{(t-s)^{1-\alpha^{(i)}}} \dd s + \sum_{i=0}^{k-1} c^{(i)} \int_0^t P(s) b(s) \dd s   \\
&~~~ + c^{(k)} T^{\alpha^{(k)}-1} \int_0^t P(s)z(s) \dd s + c^{(k)} T^{\alpha^{(k)}-1}  \int_0^t P(s) z(s) \dd s,~ \textrm{a.e.}~ t \in [0,T].
\end{align*}
Notice that the integrals above do not have the singularity. Hence, there is a constant $C$ such that $\frac{c^{(i)}}{(t-s)^{1-\alpha^{(i)}}} \leq 
\frac{C}{(t-s)^{1-\alpha}}$ for all $i$ with $0 \leq i \leq k$ and $0 \leq s < t \leq T$. 
%\begin{align*}
%\frac{c^{(i)}}{(t-s)^{1-\alpha^{(i)}}} \leq 
%\frac{C}{(t-s)^{1-\alpha}},~ \forall  i=0,\ldots,k,~ 0 \leq s < t \leq T.
%\end{align*}
This implies that
\begin{align*}	
z(t) 
%& \leq b(t) + \sum_{i=0}^{k-1} c^{(i)} \int_0^t \frac{P(s)b(s)}{(t-s)^{1-\alpha^{(i)}}} \dd s + \sum_{i=0}^{k-1} c^{(i)} \int_0^t P(s) b(s) \dd s   \\
%&~~~ + c^{(k)} T^{1-\alpha^{(k)}} \int_0^t P(s)z(s) \dd s + c^{(k)} T^{1-\alpha^{(k)}}  \int_0^t P(s) z(s) \dd s \\
& \leq b(t) + C \int_0^t \frac{P(s) b(s)}{(t-s)^{1-\alpha}} \dd s + C \int_0^t P(s) b(s) \dd s\\
&~~~ + c^{(k)} T^{\alpha^{(k)}-1} \int_0^t P(s)z(s) \dd s + c^{(k)} T^{\alpha^{(k)}-1}  \int_0^t P(s) z(s) \dd s,~ \textrm{a.e.}~ t \in [0,T]. 
\end{align*}
Then we apply the standard Gronwall's inequality (see \cite[page 14]{Walter_book}) to obtain the desired result. This completes the proof of the lemma.
\end{proof}

\section{Well-posedness and Estimates of Volterra Integral Equations}\label{Appendix_B}

We prove Lemma \ref{Lemma_2_1} in a more general setting when the initial condition of (\ref{eq_1}) is also dependent on the outer time variable. Consider the following Volterra integral equation:
\begin{align}
\label{eq_b_1}
x(t) = x_0(t) + \int_0^t \frac{f(t,s,x(s),u(s))}{(t-s)^{1-\alpha}} \dd s + \int_0^t g(t,s,x(s),u(s)) \dd s,~ \textrm{a.e.}~ t \in [0,T].
\end{align}
Let $x(\cdot;x_0,u) := x(\cdot)$ be the solution of (\ref{eq_b_1}) under $(x_0(\cdot),u(\cdot)) \in L^p([0,T];\mathbb{R}^n) \times \mathcal{U}^p[0,T]$, where we recall
\begin{align*}
\mathcal{U}^p[0,T] = \Bigl \{u:[0,T] \rightarrow U~|~ \textrm{$u$ is measurable in $t \in [0,T]$} ~\&~ \rho(u(\cdot),u_0) \in L^p([0,T];\mathbb{R}_+) \Bigr \}
\end{align*}
Here, $(U,\rho)$ is a separable metric space, where $U \subset \mathbb{R}^d$ and $\rho$ is the metric induced by the standard Euclidean norm $|\cdot|_{\mathbb{R}^d}$

\begin{assumption}\label{Assumption_B_1}
%\begin{enumerate}[(i)]
%\item $(U,\rho)$ is a separable metric space, where $U \subset \mathbb{R}^d$ and $\rho$ is the metric induced by the standard Euclidean norm $|\cdot|_{\mathbb{R}^d}$;
%\item 
For $p \geq 1$ and $\alpha \in (0,1)$, there are nonnegative $K_0(\cdot) \in L^{ (\frac{p}{1 + \alpha p} \vee 1)+}([0,T];\mathbb{R})$ and $K(\cdot) \in L^{ (\frac{1}{\alpha} \vee \frac{p}{p-1})+}([0,T];\mathbb{R})$, where $L^{p+}([0,T];\mathbb{R}^n) := \cup_{r > p} L^{r}([0,T];\mathbb{R}^n)$ for $1 \leq p < \infty$ and $t \vee s := \max \{t,s\}$ for $t,s \in [0,T]$, such that
\begin{align*}
\begin{cases}
	|f(t,s,x,u) - f(t,s,x^\prime,u^\prime)| + |g(t,s,x,u) - g(t,s,x^\prime,u^\prime)| \leq K(s) (|x-x^\prime| + \rho(u,u^\prime)) , \\
	~~~~~~~~~~ \forall (t,s) \in \Delta,~ x,x^\prime \in \mathbb{R}^n,~ u,u^\prime \in U, \\
	|f(t,s,0,u)| + |g(t,s,0,u)| \leq K_0(s),~ \forall (t,s) \in \Delta,~ u \in U.
\end{cases}
\end{align*}
%\end{enumerate}	
\end{assumption}

\begin{lemma}\label{Lemma_B_1}
Let Assumption \ref{Assumption_B_1} hold. Assume that $p \geq 1$ and $\alpha \in (0,1)$. Then for any $(x_0(\cdot),u(\cdot)) \in L^p([0,T];\mathbb{R}^n) \times \mathcal{U}^p[0,T]$, (\ref{eq_b_1}) admits a unique solution in $L^p([0,T];\mathbb{R}^n)$. In addition, there is a constant $C \geq 0$ such that (\ref{eq_b_1}) holds the following estimate:
\begin{align}
\label{eq_b_1_1_1_2}
\Bigl \| x(\cdot;x_0,u) \Bigr \|_{p} \leq C \Bigl (1 + \|x_0(\cdot)\|_{p} + \|\rho(u(\cdot),u_0)\|_{L^p([0,T];\mathbb{R}_+)} \Bigr ).	
\end{align}
Furthermore, for any $x_0(\cdot),x_0^\prime(\cdot) \in L^p([0,T];\mathbb{R}^n)$ and $u(\cdot),u^\prime(\cdot) \in \mathcal{U}^p[0,T]$, there is a constant $C \geq 0$ such that
\begin{align}
\label{eq_b_1_1_1_3}
& \|x(\cdot;x_0,u) - 	x(\cdot;x_0^\prime,u^\prime) \|_{p}  \leq C  \|x_0(\cdot) - x_0^\prime(\cdot) \|_{p} \\
&~~~~~ + C\Biggl [ \int_0^T \Bigl ( \int_0^t \frac{|f(t,s,x(s;x_0,u),u(s)) - f(t,s,x(s;x_0,u),u^\prime(s))|}{(t-s)^{1-\alpha}} \dd s \Bigr )^p \dd t \Biggr]^{\frac{1}{p}}  \nonumber \\
&~~~~~ + C\Biggl [ \int_0^T \Bigl ( \int_0^t | g(t,s,x(s;x_0,u),u(s)) - g(t,s,x(s;x_0,u),u^\prime(s)) | \dd s \Bigr)^p  \dd t \Biggr ]^{\frac{1}{p}}. \nonumber
\end{align}
\end{lemma}

\begin{remark}\label{Remark_B_1}
\begin{enumerate}[(i)]
\item By Assumption \ref{Assumption_B_1}, we have
\begin{align*}
|f(t,s,x,u)| + |g(t,s,x,u)| & \leq K_0(s) + K(s)(|x| + \rho(u,u_0)),~ \forall (t,s) \in \Delta,~ (x,u) \in \mathbb{R}^n \times U.
\end{align*}
\item Unlike Assumption \ref{Assumption_2_1}, we do not assume $p > \frac{1}{\alpha}$ in Assumption \ref{Assumption_B_1}. In addition, the conditions of $K_0$ and $K$ in Assumption \ref{Assumption_B_1} are weaker than those in Assumption \ref{Assumption_2_1} when $p > \frac{1}{\alpha}$. Indeed, with $p > \frac{1}{\alpha}$, we can show that $K_0(\cdot) \in L^{\frac{1}{\alpha} + } ([0,T];\mathbb{R}) \subset L^{ (\frac{p}{1 + \alpha p} \vee 1)+}([0,T];\mathbb{R})$ and $K(\cdot) \in L^{\frac{p}{\alpha p - 1} + } ([0,T];\mathbb{R}) \subset   L^{ (\frac{1}{\alpha} \vee \frac{p}{p-1})+}([0,T];\mathbb{R})$. This means that Lemma \ref{Lemma_2_1} can be shown under the weaker assumption than Assumption \ref{Assumption_2_1}
\end{enumerate}
\end{remark}

\begin{proof}[Proof of Lemma \ref{Lemma_B_1}]
The main idea of the proof is the extension of \cite[Theorem 3.1]{Lin_Yong_SICON_2020}, where unlike \cite{Lin_Yong_SICON_2020} we have to consider the cross coupling characteristics between the singular and nonsingular kernels in (\ref{eq_b_1}). Furthermore, our proof provides a more detailed statement, which can be viewed as a refinement of \cite{Lin_Yong_SICON_2020}. 

We first use the contraction mapping argument to show the existence and uniqueness of the solution to (\ref{eq_b_1}). For $\tau \in [0,T]$, where $\tau$ will be determined below, let us define
\begin{align*}	
\mathcal{F}[x(\cdot)](t) := x_0(t) + \int_0^t \frac{f(t,s,x(s),u(s))}{(t-s)^{1-\alpha}} \dd s + \int_0^t g(t,s,x(s),u(s)) \dd s,~ \textrm{a.e.}~ t \in [0,\tau]. 
\end{align*}
For $q,p \geq 1$ and $r \in  [1,\frac{1}{1-\alpha})$, set $r= 1+\beta$, where $\beta \geq 0$ (equivalently, $\beta \in [0,\frac{\alpha}{1-\alpha})$) and $\frac{1}{p} + 1 = \frac{1}{q} + \frac{1}{1+\beta}$. By Lemma \ref{Lemma_A_2} and Remark \ref{Remark_B_1}, it follows that
\begin{align}
\label{eq_b_2}
& \|\mathcal{F}[x(\cdot)] (\cdot) \|_{L^p([0,\tau];\mathbb{R}^n)} \\ 
& \leq \|x_0(\cdot)\|_{p} + \Bigl ( \Bigl ( \frac{\tau^{1-(1+\beta)(1-\alpha)}}{1-(1+\beta)(1-\alpha)} \Bigr )^{\frac{1}{1+\beta}} + \tau^{\frac{1}{1+\beta}} \Bigr ) \Bigl \|K_0(\cdot) + K(\cdot)(|x(\cdot)| + \rho(u(\cdot),u_0) ) \Bigr \|_{L^q([0,\tau];\mathbb{R})}. \nonumber
\end{align}
Below, we consider the three different cases.

\subsubsection*{Case I: $p > \frac{1}{1-\alpha}$}
Note that $\frac{1}{p} < 1 - \alpha$.  Moreover, $\frac{1}{\alpha} > \frac{p}{p-1}$, $\frac{p}{1+\alpha p} > 1$, and $ 1+\beta < \frac{1}{1-\alpha} = 1+ \frac{\alpha}{1-\alpha}$ (equivalently, $\beta < \frac{\alpha}{1-\alpha}$). In this case, we have $K_0(\cdot) \in L^{\frac{p}{1+\alpha p}+}([0,T];\mathbb{R})$ and $K(\cdot) \in L^{\frac{1}{\alpha}+}([0,T];\mathbb{R})$. Observe that $\frac{1}{q} = \frac{1}{p} + 1 - \frac{1}{1+\beta} < 	\frac{1}{p} + 1 - \frac{1}{1 + \frac{\alpha}{1-\alpha}} = \frac{1}{p} + \alpha  < 1$ and $\frac{1}{q} - \frac{1}{p} = \frac{p-q}{pq} = 1 - \frac{1}{1+\beta} < 1 - \frac{1}{1+\frac{\alpha}{1-\alpha}} = \alpha < 1$, 
%\begin{align*}
%\frac{1}{q} &= \frac{1}{p} + 1 - \frac{1}{1+\beta} < 	\frac{1}{p} + 1 - \frac{1}{1 + \frac{\alpha}{1-\alpha}} = \frac{1}{p} + \alpha  < 1, \\
%\frac{1}{q} - \frac{1}{p} &= \frac{p-q}{pq} = 1 - \frac{1}{1+\beta} < 1 - \frac{1}{1+\frac{\alpha}{1-\alpha}} = \alpha < 1,
%\end{align*}
which implies $q \searrow \frac{p}{1+\alpha p} < 1$ and $\frac{pq}{p-q} \searrow \frac{1}{\alpha}$ as $\beta \nearrow \frac{\alpha}{1-\alpha}$. Hence, since $K_0(\cdot) \in L^{\frac{p}{1+\alpha p}+}([0,T];\mathbb{R})$ and $K(\cdot) \in L^{\frac{1}{\alpha}+}([0,T];\mathbb{R})$, we may choose $\beta$ close enough to $\frac{\alpha}{1-\alpha}$ so that $K_0 (\cdot)\in L^{q}([0,T];\mathbb{R})$ and $K (\cdot) \in L^{\frac{pq}{p-q}}([0,T];\mathbb{R})$. Therefore, as $\frac{p-q}{p} + \frac{q}{p} = 1$, it follows that
\begin{align*}
& \|K_0(\cdot) + K(\cdot)(|x(\cdot)| + u(\cdot))\|_{L^q([0,\tau];\mathbb{R})} \\
& \leq \|K_0(\cdot)\|_{L^q([0,T];\mathbb{R})} + \|K(\cdot)\|_{L^{\frac{pq}{p-q}}([0,T];\mathbb{R})} (\| x(\cdot)\|_{L^p([0,T];\mathbb{R}^n)} + \|\rho(u(\cdot),u_0) \|_{L^p([0,T];\mathbb{R})}).
\end{align*}
This, together with (\ref{eq_b_2}), implies
\begin{align}
\label{eq_b_3}
\bigl \|\mathcal{F}[x(\cdot)] (\cdot) \bigr \|_{L^p([0,\tau];\mathbb{R}^n)} & \leq \|x_0(\cdot)\|_{p} + \Bigl ( \Bigl ( \frac{\tau^{1-(1+\beta)(1-\alpha)}}{1-(1+\beta)(1-\alpha)} \Bigr )^{\frac{1}{1+\beta}} + \tau^{\frac{1}{1+\beta}} \Bigr ) \Bigl [  \|K_0(\cdot)\|_{L^q([0,T];\mathbb{R})}	  \\
&~~~  + \|K(\cdot)\|_{L^{\frac{pq}{p-q}}([0,T];\mathbb{R})} \Bigl ( \| x (\cdot) \|_{L^p([0,\tau];\mathbb{R}^n)} + \|\rho(u(\cdot),u_0) \|_{L^p([0,\tau];\mathbb{R})} \Bigr ) \Bigr ]. \nonumber
\end{align}
This shows $\mathcal{F}[x(\cdot)]: L^p([0,\tau];\mathbb{R}^n) \rightarrow L^p([0,\tau];\mathbb{R}^n)$ for $\tau \in [0,T]$. For $x(\cdot),x^\prime (\cdot) \in L^p([0,\tau];\mathbb{R}^n)$, by Lemma \ref{Lemma_A_2} and Assumption \ref{Assumption_2_1}, and using the same technique as (\ref{eq_b_2}) and (\ref{eq_b_3}), it follows that
\begin{align*}
& \Bigl \|( \mathcal{F}[x(\cdot)] - \mathcal{F}[x^\prime(\cdot)])(\cdot) \Bigr \|_{L^p([0,\tau];\mathbb{R}^n)} \\
%& \leq \Bigl ( \int_0^{\tau} \Bigl | \int_0^t \frac{f(t,s,x(s),u(s)) - f(t,s,x^\prime(s),u(s))}{(t-s)^{1-\alpha}} \dd s \Bigr |^p \dd t \Bigr )^{\frac{1}{p}} \\
%&~~~ + \Bigl ( \int_0^\tau \Bigl | \int_0^t \bigl [ g(t,s,x(s),u(s)) - g(t,s,x(s),u(s)) \bigr] \dd s \Bigr |^p \dd s \Bigr )^{\frac{1}{p}} \\
%& \leq \Bigl ( \Bigl ( \frac{\tau^{1-(1+\beta)(1-\alpha)}}{1-(1+\beta)(1-\alpha)} \Bigr )^{\frac{1}{1+\beta}} + \tau^{\frac{1}{1+\beta}} \Bigr ) \|K(\cdot)(x(\cdot) - x^\prime(\cdot)) \|_{L^q([0,\tau];\mathbb{R})} \\
& \leq \Bigl ( \Bigl ( \frac{\tau^{1-(1+\beta)(1-\alpha)}}{1-(1+\beta)(1-\alpha)} \Bigr )^{\frac{1}{1+\beta}} + \tau^{\frac{1}{1+\beta}} \Bigr ) \|K(\cdot)\|_{L^{\frac{pq}{p-q}}([0,T];\mathbb{R})} \|x(\cdot) - x^\prime(\cdot) \|_{L^p([0,\tau];\mathbb{R}^n)}.
\end{align*}
Take $\tau \in (0,T]$, independent of $x_0$, such that $\Bigl ( \Bigl ( \frac{\tau^{1-(1+\beta)(1-\alpha)}}{1-(1+\beta)(1-\alpha)} \Bigr )^{\frac{1}{1+\beta}} + \tau^{\frac{1}{1+\beta}} \Bigr ) \|K(\cdot)\|_{L^{\frac{pq}{p-q}}([0,T];\mathbb{R})} < 1$. Then the mapping $\mathcal{F}[x(\cdot)]: L^p([0,\tau];\mathbb{R}^n) \rightarrow L^p([0,\tau];\mathbb{R}^n)$ is contraction. Hence, in view of the contraction mapping theorem, (\ref{eq_b_1}) admits a unique solution on $[0,\tau]$ in $L^p([0,\tau];\mathbb{R}^n)$.

For $[\tau,2\tau]$, consider,
\begin{align*}	
\mathcal{F}[y(\cdot)](t) & := x_0(t) + \int_0^{\tau} \frac{f(t,s,x(s),u(s))}{(t-s)^{1-\alpha}} \dd s + \int_0^{\tau} g(t,s,x(s),u(s)) \dd s \\
&~~~ + \int_{\tau}^{t} \frac{f(t,s,y(s),u(s))}{(t-s)^{1-\alpha}} \dd s + \int_{\tau}^{t} g(t,s,y(s),u(s)) \dd s,~ \textrm{a.e.}~ t \in [\tau,2\tau].
\end{align*}
Note that by Lemma \ref{Lemma_A_2} and (\ref{eq_b_3}), we have
\begin{align*}
& \bigl \|\mathcal{F}[y(\cdot)] (\cdot) \bigr \|_{L^p([\tau,2\tau];\mathbb{R}^n)}  \\
%& \leq  \Bigl \| x_0(\cdot) + \int_0^{\tau} \frac{f(\cdot,s,x(s),u(s))}{(\cdot-s)^{1-\alpha}} \dd s + \int_0^{\tau} g(\cdot,s,x(s),u(s)) \dd s  \Bigr \|_{L^p([\tau,2\tau];\mathbb{R}^n)} \\
%&~~~ + \Bigl \| \int_{\tau}^{2\tau} \frac{f(\cdot,s,y(s),u(s))}{(\cdot-s)^{1-\alpha}} \dd s + \int_{\tau}^{2\tau} g(\cdot,s,y(s),u(s)) \dd s \Bigr \|_{L^p([\tau,2\tau];\mathbb{R}^n)} \\
%& \leq \|x_0(\cdot)\|_{p} + \Bigl ( \int_\tau^{2\tau} \Bigl | \int_0^{\tau} \frac{f(t,s,x(s),u(s))}{(t-s)^{1-\alpha}} \dd s \Bigr|^p \dd t \Bigr)^{\frac{1}{p}} + \Bigl ( \int_{\tau}^{2\tau} \Bigl | \int_0^{\tau} g(t,s,x(s),u(s)) \dd s \Bigr |^{p} \dd t \Bigr )^{\frac{1}{p}} \\
%&~~~ + \Bigl ( \int_{\tau}^{2\tau} \Bigl | \int_0^{2 \tau} \frac{f(t,s,y(s),u(s))}{(t-s)^{1-\alpha}} \dd s \Bigr|^p \dd t + \Bigl ( \int_{\tau}^{2\tau} \Bigl | \int_{\tau}^{2 \tau} g(t,s,y(s),u(s)) \dd s \Bigr |^{p} \dd t \Bigr )^{\frac{1}{p}} \\
& \leq \|x_0(\cdot) \|_{p} + C \Bigl [  \|K_0 (\cdot) \|_{L^q([0,T];\mathbb{R})}	 + \|K (\cdot) \|_{L^{\frac{pq}{p-q}}([0,T];\mathbb{R})} \Bigl ( \| x(\cdot) \|_{L^p([0,\tau];\mathbb{R}^n)}  \\
&~~~ + \|\rho(u(\cdot),u_0) \|_{L^p([0, \tau];\mathbb{R})} \Bigr ) + \|K (\cdot) \|_{L^{\frac{pq}{p-q}}([0,T];\mathbb{R})} \Bigl ( \| y(\cdot) \|_{L^p([\tau, 2 \tau];\mathbb{R}^n)} + \|\rho(u(\cdot),u_0) \|_{L^p([\tau, 2 \tau];\mathbb{R})} \Bigr )  \Bigr ],
\end{align*}
which shows that $\mathcal{F}[y(\cdot)]:   L^p([\tau,2\tau];\mathbb{R}^n) \rightarrow L^p([\tau,2\tau];\mathbb{R}^n)$. Moreover, by a similar argument, it follows that for any $y(\cdot),y^\prime (\cdot) \in L^p([\tau,2\tau];\mathbb{R}^n)$,
\begin{align*}
& \bigl \| (\mathcal{F}[y(\cdot)] - \mathcal{F}[y^\prime(\cdot)] )(\cdot)\bigr \|_{L^p([\tau,2\tau];\mathbb{R}^n)} \\
& \leq \Bigl ( \Bigl ( \frac{\tau^{1-(1+\beta)(1-\alpha)}}{1-(1+\beta)(1-\alpha)} \Bigr )^{\frac{1}{1+\beta}} + \tau^{\frac{1}{1+\beta}} \Bigr ) \|K(\cdot)\|_{L^{\frac{pq}{p-q}}([0,T];\mathbb{R})} \|y(\cdot) - y^\prime(\cdot) \|_{L^p([\tau, 2 \tau];\mathbb{R}^n)}.	
\end{align*}
As before, we have $\Bigl ( \Bigl ( \frac{\tau^{1-(1+\beta)(1-\alpha)}}{1-(1+\beta)(1-\alpha)} \Bigr )^{\frac{1}{1+\beta}} + \tau^{\frac{1}{1+\beta}} \Bigr ) \|K(\cdot)\|_{L^{\frac{pq}{p-q}}([0,T];\mathbb{R})} < 1$. Hence, (\ref{eq_b_1}) admits a unique solution on $[\tau,2 \tau]$ in $L^p([\tau,2 \tau];\mathbb{R}^n)$. By induction, we are able to prove the existence and uniqueness of the solution for (\ref{eq_b_1}) on $[0, \tau], [\tau,2 \tau], \ldots, 
%[\lfloor \frac{T}{\tau} \rfloor \tau  - \tau, \lfloor \frac{T}{\tau} \rfloor \tau], 
[\lfloor \frac{T}{\tau} \rfloor \tau ,T]$. This shows the existence and uniqueness of the solution for (\ref{eq_b_1}) on $[0,T]$ in  $L^p([0,T];\mathbb{R}^n)$. 

We now prove the estimates in (\ref{eq_b_1_1_1_2}) and (\ref{eq_b_1_1_1_3}). Let $z(\cdot) := |x(\cdot)	- x^\prime(\cdot) |_{\mathbb{R}^n}$, where $x(\cdot) := x(\cdot;x_0,u)$ and $x^\prime(\cdot) := x(\cdot;x_0^\prime,u^\prime)$. Then
\begin{align*}
z(t)  
%& \leq |x_0(t) - x_0^\prime(t)|_{\mathbb{R}^n} + \int_0^t \frac{|f(t,s,x(s;x_0,u),u(s))-f(t,s,x(s;x_0,u),u^\prime(s))|}{(t-s)^{1-\alpha}} \dd s \\
%&~~~ + \int_0^t \frac{|f(t,s,x(s),u^\prime(s))-f(t,s,x^\prime(s),u^\prime(s))|}{(t-s)^{1-\alpha}} \dd s \\
%&~~~ + \int_0^t |g(t,s,x(s;x_0,u),u(s))-g(t,s,x(s;x_0,u),u^\prime(s))| \dd s \\
%&~~~ + \int_0^t |g(t,s,x(s),u^\prime(s)) -g(t,s,x^\prime(s),u^\prime(s))| \dd s  \\
& \leq b(t) +  \int_0^t \frac{ K(s) z(s) }{(t-s)^{1-\alpha}} \dd s   + \int_0^t K(s) z(s) \dd s,~ \textrm{a.e.}~ t \in [0,T],
\end{align*}
where
\begin{align*}
b(t) & := |x_0(t) - x_0^\prime(t)|_{\mathbb{R}^n} + \int_0^t \frac{|f(t,s,x(s),u^\prime(s))-f(t,s,x(s),u(s))|}{(t-s)^{1-\alpha}} \dd s \\
&~~~ + \int_0^t |g(t,s,x(s),u^\prime(s))-g(t,s,x(s),u(s))| \dd s,~ \textrm{a.e.}~ t \in [0,T].
\end{align*}
Note that $z(\cdot),b(\cdot) \in L^p([0,T];\mathbb{R})$ and $K(\cdot) \in L^{\frac{1}{\alpha}+}([0,T];\mathbb{R})$. We replace $p$ by $q$ in Lemma \ref{Lemma_A_5}. Recall that the $L^p$-spaces of this paper are induced by the finite measure on $([0,T],\mathcal{B}([0,T]))$. Then as $\frac{1}{q} < \frac{1}{p} + \alpha < 1$, we may increase $q$ enough to get $K (\cdot) \in L^{q}([0,T];\mathbb{R}) \subset L^{\frac{1}{\alpha}+}([0,T];\mathbb{R})$ and $z(\cdot),b(\cdot) \in L^p([0,T];\mathbb{R}) \subset L^{\frac{q}{q-1}}([0,T];\mathbb{R})$. By Lemma \ref{Lemma_A_5}, it follows that
\begin{align*}	
z(t) = |x(t)	- x^\prime(t) |_{\mathbb{R}^n} \leq  b(t) + C \int_0^t \frac{K(s) b(s)}{(t-s)^{1-\alpha}} \dd s + C \int_0^t K(s) b(s) \dd s,~ \textrm{a.e.}~ t \in [0,T].
\end{align*}
Hence, similar to (\ref{eq_b_2}), 
\begin{align*}
& \|x(\cdot;x_0,u)	- x(\cdot;x_0^\prime,u^\prime)	\|_{p} 
%& \leq C  \Bigl ( \int_0^T |b(t)|^p \dd t  \Bigr )^{\frac{1}{p}} + C \Bigl ( \int_0^T \Bigl | \int_0^t \frac{K(s) b(s)}{(t-s)^{1-\alpha}} \dd s \Bigr |^p \dd t  \Bigr)^{\frac{1}{p}} + C \Bigl ( \int_0^T \Bigl | \int_0^t K(s) b(s) \dd s \Bigr |^p \dd t  \Bigr)^{\frac{1}{p}} \\
%& \leq C \Biggl ( \|b\|_{p} +  \Bigl ( \frac{T^{1-(1+\beta)(1-\alpha)}}{1-(1+\beta)(1-\alpha)} \Bigr )^{\frac{1}{1+\beta}} + T^{\frac{1}{1+\beta}} \Bigr ) \|K(\cdot)b(\cdot) \|_{L^q([0,T];\mathbb{R})} \Biggr ) \\
 \leq C \Bigl ( \|b(\cdot)\|_{p} + \|K(\cdot)\|_{L^{\frac{pq}{p-q}}([0,T];\mathbb{R})}\|b(\cdot)\|_{p} \Bigr ) \leq C \|b(\cdot)\|_{p}.
\end{align*}
This shows the estimate in (\ref{eq_b_1_1_1_3}). The estimate in (\ref{eq_b_1_1_1_2}) can be shown in a similar way.

\subsubsection*{Case II: $1 < p \leq \frac{1}{1-\alpha}$} 

This case implies $1-\alpha \leq \frac{1}{p} < 1$, $\frac{1}{\alpha} \leq \frac{p}{p-1}$, and $\frac{p}{1+\alpha p} \leq 1$. Moreover, for $\beta \in (0,p-1)$ (equivalent to $1 + \beta  \in (1,p)$),  we have $1-\alpha \leq \frac{1}{p} < \frac{1}{1+\beta}$. Hence,  $K_0 (\cdot) \in L^{1+}([0,T];\mathbb{R})$ and $K (\cdot)  \in L^{\frac{p}{p-1}+}([0,T];\mathbb{R})$. Then since $\frac{1}{p} + 1 = \frac{1}{q} + \frac{1}{1+\beta}$, we observe that $\frac{1}{p}	< \frac{1}{q} = \frac{1}{p} + 1 - \frac{1}{1+\beta} \nearrow 1$ and $\frac{p-q}{pq} = \frac{1}{q} - \frac{1}{p} = 1 - \frac{1}{1+\beta} \nearrow  \frac{p-1}{p}$ as $\beta \nearrow p-1$.
%\begin{align*}
%\frac{1}{p}	< \frac{1}{q} = \frac{1}{p} + 1 - \frac{1}{1+\beta} \nearrow  1,~~~ \frac{p-q}{pq} = \frac{1}{q} - \frac{1}{p} = 1 - \frac{1}{1+\beta} \nearrow  \frac{p-1}{p},~~~ \textrm{as $\beta \nearrow p-1$.}
%\end{align*}

As $K_0(\cdot) \in L^{1+}([0,T];\mathbb{R})$ and $K(\cdot) \in L^{\frac{p}{p-1}+}([0,T];\mathbb{R})$, we are able to choose $\beta$ close enough to $p-1$ to get $q > 1$ and $\frac{p-q}{pq} < \frac{p-1}{p}$, which implies $K_0(\cdot) \in L^{q}([0,T];\mathbb{R})$ and $K(\cdot) \in L^{\frac{pq}{p-q}}([0,T];\mathbb{R})$. We replace $p$ by $q$ in Lemma \ref{Lemma_A_5}. Since $p \in [1,\frac{1}{1-\alpha}]$, choose $q$ to get $q > \frac{p}{p-1} \geq \frac{1}{\alpha}$, which implies $K(\cdot) \in L^{q}([0,T];\mathbb{R}) \subset L^{\frac{p}{p-1}+}([0,T];\mathbb{R})$ and $z(\cdot),b(\cdot) \in L^p([0,T];\mathbb{R}) \subset L^{\frac{q}{q-1}}([0,T];\mathbb{R})$. Then the technique for Case I can be applied to prove Case II.

\subsubsection*{Case III: $p=1$} 

We have $K_0(\cdot) \in L^{1+}([0,T];\mathbb{R})$ and $K(\cdot) \in L^{\infty}([0,T];\mathbb{R})$. Choose $\beta = 0$ and use (\ref{eq_b_2}) to get 
\begin{align*}
\bigl \|\mathcal{F}[x(\cdot)](\cdot) \bigr \|_{L^1([0,\tau];\mathbb{R}^n)} & \leq \|x_0(\cdot)\|_{p} +  \frac{\tau^{\alpha}}{\alpha} \Bigl [  \|K_0(\cdot)\|_{L^1([0,T];\mathbb{R})}	 \\
&~~~~~~~ + \|K(\cdot)\|_{L^{\infty}([0,T];\mathbb{R})}( \| x(\cdot) \|_{L^1([0,\tau];\mathbb{R}^n)} + \|\rho(u(\cdot),u_0)\|_{L^1([0,\tau];\mathbb{R})}) \Bigr ].
\end{align*}
Then the rest of the proof is similar to that for Case I. This completes the proof of the theorem.
\end{proof}

\begin{remark}\label{Remark_B_2}
	The integrability of $K_0$ and $K$ in Assumption \ref{Assumption_B_1} is crucial in the proof of Lemma \ref{Lemma_B_1}. Comparing between Cases I and II, we see that $(x_0(\cdot),u(\cdot)) \in L^p([0,T];\mathbb{R}^n) \times \mathcal{U}^p[0,T]$ has weaker integrability in Case II. Hence, we need stronger integrability of $K$ from $K(\cdot) \in L^{\frac{1}{\alpha}+}([0,T];\mathbb{R})$ to $K(\cdot) \in L^{\frac{p}{p-1}+}([0,T];\mathbb{R})$, and $K_0$ from $K_0(\cdot) \in L^{\frac{p}{1+\alpha p} + } ([0,T];\mathbb{R})$ to $K_0(\cdot) \in L^{1 + } ([0,T];\mathbb{R})$  (note that in Case II, $\frac{1}{\alpha} \leq \frac{p}{p-1}$ and $\frac{p}{1+\alpha p} \leq 1$). Notice that for Case III, by the weakest integrability of $(x_0(\cdot),u(\cdot)) \in L^p([0,T];\mathbb{R}^n) \times \mathcal{U}^p[0,T]$, we need the essential boundedness of $K$, i.e., the strongest integrability condition for $K$. Finally, as the proof relies on the contraction mapping argument, the solution of (\ref{eq_b_1}) can be constructed via the standard Picard iteration algorithm, which is applied to Examples \ref{Example_1} and \ref{Example_2} in Section \ref{Section_5}.
\end{remark}

We state the continuity of the solution under the stronger assumption (see Remark \ref{Remark_B_1}).
\begin{lemma}\label{Lemma_B_2}
Let Assumption \ref{Assumption_2_1} hold and $x_0(\cdot) \in C([0,T];\mathbb{R}^n)$. Then (\ref{eq_b_1}) admits a unique solution in $C([0,T];\mathbb{R}^n)$. 
\end{lemma}

\begin{proof}
Based on Lemma \ref{Lemma_B_1} and Remark \ref{Remark_B_1}, (\ref{eq_b_1}) admits a unique solution in $L^p([0,T];\mathbb{R}^n)$.  Notice that under Assumption \ref{Assumption_2_1},  $K_0(\cdot) \in L^{\frac{1}{\alpha} + } ([0,T];\mathbb{R})$ and $K(\cdot) \in L^{\frac{p}{\alpha p - 1} + } ([0,T];\mathbb{R})$. Let $q = \frac{sp}{s+p}$, where $s > \frac{p}{\alpha p - 1}$. We observe $K(\cdot) \in L^s([0,T];\mathbb{R}^n)$. Since $s = \frac{pq}{p-q}$, we have $\frac{1}{q} = \frac{s+p}{sp} = \frac{1}{p} + \frac{1}{s} > \frac{1}{p}$ and $\frac{pq}{p-q} > \frac{p}{\alpha p - 1}	\Rightarrow \frac{p-q}{pq} < \frac{\alpha p - 1}{p} \Rightarrow  \alpha > \frac{1}{q}$. This implies $ \frac{1}{p} < \frac{1}{q} < \alpha$, i.e., $p > q > \frac{1}{\alpha}$. 

Note that $K_0(\cdot) \in L^q([0,T];\mathbb{R})$ and 
\begin{align}
\label{eq_b_6}
|f(t,s,x(s),u(s))| + |g(t,s,x(s),u(s))| & \leq K_0(s) + K(s)(|x(s)|_{\mathbb{R}^n} + |\rho(u(s),u_0)|) =: \overline{\psi}(s).
\end{align}
Consequently, using the H\"older's inequality, we get
\begin{align*}
& \Bigl ( 	\int_0^T |K(s)|^q(|x(s)|_{\mathbb{R}^n} + \rho(u(s),u_0))^q \dd s   \Bigr)^{\frac{1}{q}}  \\
& \leq \|K(\cdot)\|_{L^{ \frac{pq}{p-q} }([0,T];\mathbb{R})}  ( \|x (\cdot) \|_{p} + \| \rho(u(\cdot),u_0)\|_{L^p([0,T];\mathbb{R})}  ) < \infty.
\end{align*}
This implies $\overline{\psi}$ defined in (\ref{eq_b_6}) holds $\overline{\psi}(\cdot) \in L^q([0,T];\mathbb{R})$. As $q > \frac{1}{\alpha}$, the continuity of (\ref{eq_b_1}) follows from Lemmas \ref{Lemma_A_3} and \ref{Lemma_A_4}, together with Assumption \ref{Assumption_2_1}. This completes the proof.
\end{proof}

We study linear Volterra integral equations having singular and nonsingular kernels. For $\alpha \in (0,1)$ and $x_0(\cdot) \in L^p([0,T];\mathbb{R}^n)$, consider
\begin{align}
\label{eq_b_34543534234234}
x(t) = x_0(t) + \int_{0}^t \frac{F(t,s)x(s)}{(t-s)^{1-\alpha}} \dd s + \int_0^t H(t,s) x(s) \dd s,~  \textrm{a.e.}~ t \in [0,T],
\end{align}
where $F,G:\Delta \rightarrow \mathbb{R}^{n \times n}$ satisfy $F(\cdot,\cdot), H(\cdot,\cdot)  \in L^{\infty}(\Delta;\mathbb{R}^{n \times n})$.

\begin{lemma}\label{Lemma_B_5_2323232}
The solution of (\ref{eq_b_34543534234234}) can be written as
\begin{align}
\label{eq_b_34543534234234_1_1_1}
x(t) = x_0(t) + \int_0^t \Psi(t,s) x_0(s) \dd s,~  \textrm{a.e.}~ t \in [0,T],
\end{align}
where $\Psi$ is the state transition equation defined by
\begin{align*}
\Psi(t,s) = \frac{F(t,s)}{(t-s)^{1-\alpha}} + \int_s^t \frac{F(t,\tau) \Psi(\tau,s) }{(t-\tau)^{1-\alpha}} \dd \tau  + H(t,s) + \int_s^t H(t,\tau) \Psi(\tau,s) \dd \tau,~  \textrm{a.e.}~ t \in (s,T].
\end{align*}
\end{lemma}
\begin{proof}
The well-posedness of (\ref{eq_b_34543534234234}) follows from Lemma \ref{Lemma_B_1}. From (\ref{eq_b_34543534234234_1_1_1}), it follows that
\begin{align*}
& \int_{0}^t \frac{F(t,s)}{(t-s)^{1-\alpha}} x(s) \dd s + \int_0^t H(t,s) x(s) \dd s \\
& = \int_{0}^t \frac{F(t,s)}{(t-s)^{1-\alpha}}  \Bigl [x_0(s) + \int_0^s \Psi(s,\tau) x_0(\tau) \dd \tau \Bigr ] \dd s  + \int_0^t H(t,s) \Bigl [x_0(s) + \int_0^s \Psi(s,\tau) x_0(\tau) \dd \tau \Bigr ] \dd s \\
& = \int_0^t \Bigl [ \frac{F(t,s) }{(t-s)^{1-\alpha}} + \int_s^t \frac{F(t,\tau) \Psi(\tau,s) }{(t-\tau)^{1-\alpha} } \dd \tau + H(t,s) + \int_s^t H(t,\tau) \Psi(\tau,s) \dd \tau \Bigr ] x_0(s) \dd s \\
& = \int_0^t \Psi(t,s) x_0(s) \dd s = x(t) - x_0(t),
\end{align*}
which completes the proof.
\end{proof}

Consider the following $\mathbb{R}^n$-valued backward Volterra integral equation having singular and nonsingular kernels, which covers the adjoint equation in Theorem \ref{Theorem_3_1}:
\begin{align}
\label{eq_b_7}
z(t) & = z_0(t) + \int_t^T \frac{F(r,t)^\top}{(r-t)^{1-\alpha}} z(r) \dd r + \int_t^T H(r,t)^\top z(r) \dd r \\
&~~~ + \sum_{i=1}^m C_i(t)^\top \frac{\dd \theta_i(t)}{\dd t} + D(t)^\top,~  \textrm{a.e.}~ t \in [0,T]. \nonumber
\end{align} 

\begin{assumption}\label{Assumption_6}
\begin{enumerate}[(i)]
		\item $z_0(\cdot) \in C([0,T];\mathbb{R}^n)$, $\theta (\cdot)  = (\theta_1(\cdot),\ldots,\theta_m(\cdot)) \in \textsc{NBV}([0,T];\mathbb{R}^m)$, and $\dd \theta_i \ll \dd t$, i.e., $\dd \theta_i$ is absolutely continuous with respect to $\dd t$ for $i=1,\ldots,m$;
		\item $F,H:\Delta \rightarrow \mathbb{R}^{n \times n}$ and $C_i,D:[0,T] \rightarrow \mathbb{R}^n$, $i=1,\ldots,m$, satisfy $F(\cdot,\cdot),H(\cdot,\cdot) \in L^{\infty}(\Delta;\mathbb{R}^{n \times n})$ and $C_i(\cdot),D(\cdot) \in L^{\infty}([0,T];\mathbb{R}^{n})$, $i=1,\ldots,m$. 
%		In addition, there exists a constant $K \geq 0$ such that $|A(t,s)| + |B(t,s)| + |C_1(r)| + \cdots +  |C_m(r)| + |D(r)| \leq K$ for $(t,s) \in \Delta$ and $r \in [0,T]$.
\end{enumerate}
\end{assumption}

\begin{lemma}\label{Lemma_B_5}
Let Assumption \ref{Assumption_6} hold. Assume that $p \geq 1$ and $\alpha \in (0,1)$. Then for any $z_0(\cdot) \in C([0,T];\mathbb{R}^n)$, (\ref{eq_b_7}) admits a unique solution in $L^p([0,T];\mathbb{R}^n)$.
\end{lemma}

\begin{proof}
Note that by Remark \ref{Remark_3_3} and the Radon-Nikodym theorem, there is a unique $\Theta_i(\cdot) \in L^1([0,T];\mathbb{R})$, $i=1,\ldots,m$, such that   $\frac{\dd \theta_i(t)}{\dd t} = \Theta_i(t)$ for $i=1,\ldots,m$. Hence, we may replace $\frac{\dd \theta_i(t)}{\dd t}$ by $\Theta_i(t)$ in (\ref{eq_b_7}). Let us define
\begin{align*}
\mathcal{G}[z(\cdot)](t) & := z_0(t) + \int_t^T \frac{F(r,t)^\top}{(r-t)^{1-\alpha}} z(r) \dd r + \int_t^T H(r,t)^\top z(r) \dd r \\
&~~~ + \sum_{i=1}^m C_i(t)^\top \Theta_i(t) + D(t)^\top,~  \textrm{a.e.}~ t \in [0,T].
\end{align*}
Clearly, $\mathcal{G}[z(\cdot)]: L^p([0,T];\mathbb{R}^n) \rightarrow L^p([0,T];\mathbb{R}^n)$. In addition, for $\tau > 0$, we apply a similar technique of Lemma \ref{Lemma_B_1} (with Lemma \ref{Lemma_A_2} for $p=q$ and $r=1$) to show that
\begin{align*}
\|(\mathcal{G}[z(\cdot)]- \mathcal{G}[z^\prime(\cdot)])(\cdot)\|_{L^p([T-\tau,T];\mathbb{R}^n)} \leq 	\Bigl ( \frac{\tau K}{\alpha} + \tau K \Bigr ) \|z(\cdot) - z^\prime(\cdot)\|_{L^p([T-\tau,T];\mathbb{R}^n)}.
\end{align*}
We may choose $\tau$, independent of $z_0$, such that $\Bigl ( \frac{\tau K}{\alpha} + \tau K \Bigr ) < 1$. Then by the contraction mapping theorem, (\ref{eq_b_7}) admits a unique solution on $[T-\tau,T]$ in $L^p([T-\tau,T];\mathbb{R}^n)$. By induction, we are able to show that (\ref{eq_b_7}) admits a unique solution on $[0,T]$ in $L^p([0,T];\mathbb{R}^n)$. 
%Then the fact that $z(\cdot) \in C([0,T];\mathbb{R}^n)$ follows from Lemma \ref{Lemma_B_2}. 
We complete the proof.
\end{proof}

\section{Auxiliary Lemmas}\label{Appendix_D}

\begin{lemma}[Corollary 3.9 and page 144 of \cite{Li_Yong_book} or Lemma 3 of \cite{Bourdin_arxiv_2016}]\label{Lemma_D_1}
Assume that $(X,\|\cdot\|_{X})$ is a Banach space. For $\delta \in (0,1)$, define $\mathcal{E}_{\delta} := \{ E \in [0,T]~|~ |E| = \delta T\}$, where $|E|$ denotes the Lebesgue measure of $E$. Suppose that $\phi:\Delta \rightarrow X$ satisfies the properties such that (i) $\|\phi(t,s)\|_{X} \leq \overline{\phi}(s)$ for all $(t,s) \in \Delta$, where $\overline{\phi}(\cdot) \in L^1([0,T];\mathbb{R})$, and (ii) for almost all $s \in [0,T]$, $\phi(\cdot,s):[s,T] \rightarrow X$ is continuous. Then there is an $E_{\delta} \in \mathcal{E}_{\delta}$ such that
\begin{align*}
\sup_{t \in [0,T]} \Biggl | \int_0^t \Bigl ( \frac{1}{\delta} \mathds{1}_{E_{\delta}}(s) - 1 \Bigr ) \phi(t,s) \dd s \Biggr | \leq \delta.
\end{align*}
\end{lemma}

\begin{lemma}[Lemma 4.2 of \cite{Lin_Yong_SICON_2020}]\label{Lemma_D_2}
Let $\mathcal{E}_{\delta}$ be the set in Lemma \ref{Lemma_D_1}. Assume $\psi:\Delta \rightarrow \mathbb{R}$ holds  the following property:
\begin{align}
\label{eq_d_1}
\begin{cases}
	|\psi(0,s)| \leq \overline{\psi}(s),~ s \in [0,T], \\
	|\psi(t,s) - \psi(t^\prime,s)| \leq \omega(|t-t^\prime|) \overline{\psi}(s),~ (t,s),(t^\prime,s) \in \Delta,
\end{cases}	
\end{align}
where $\overline{\psi} \in L^p([0,T];\mathbb{R})$ with $p > \frac{1}{\alpha}$ and $\omega:[0,\infty) \rightarrow [0,\infty)$ is some modulus of continuity. Then there is an $E_{\delta} \in \mathcal{E}_{\delta}$ such that
\begin{align*}
	\sup_{t \in [0,T]} \Biggl | \int_0^t \Bigl (\frac{1}{\delta} \mathds{1}_{E}(s) - 1 \Bigr ) \frac{\psi(t,s)}{(t-s)^{1-\alpha}} \dd s \Biggr | \leq \delta.
\end{align*}
\end{lemma}

\end{appendices}

%\section*{References}
%{\footnotesize
\bibliographystyle{IEEEtranS}
\bibliography{researches_1_NEW.bib}
%}
\end{document}